\documentclass[12pt]{article}

\usepackage{amsmath, amsthm, amssymb, accents}
\usepackage{enumitem}
\usepackage{pdflscape}
\usepackage{caption}

\usepackage{ifpdf}
\ifpdf
\usepackage[pdftex]{graphicx}
\else
\usepackage[dvips]{graphicx}
\fi
\usepackage{tikz}
 	 \usetikzlibrary{arrows,backgrounds}
\usepackage[all]{xy}

\usepackage{multicol}
\usepackage{adjustbox}

\input xy
\xyoption{all}

\usepackage[pdftex,plainpages=false,hypertexnames=false,pdfpagelabels]{hyperref}
\newcommand{\arxiv}[1]{\href{http://arxiv.org/abs/#1}{\tt arXiv:\nolinkurl{#1}}}
\newcommand{\arXiv}[1]{\href{http://arxiv.org/abs/#1}{\tt arXiv:\nolinkurl{#1}}}
\newcommand{\doi}[1]{\href{http://dx.doi.org/#1}{{\tt DOI:#1}}}

\newcommand{\googlebooks}[1]{(preview at \href{http://books.google.com/books?id=#1}{google books})}

\usepackage{xcolor}
\definecolor{dark-red}{rgb}{0.7,0.25,0.25}
\definecolor{dark-blue}{rgb}{0.15,0.15,0.55}
\definecolor{medium-blue}{rgb}{0,0,.8}
\definecolor{DarkGreen}{RGB}{0,150,0}
\definecolor{rho}{named}{red}
\hypersetup{
   colorlinks, linkcolor={purple},
   citecolor={medium-blue}, urlcolor={medium-blue}
}

\usepackage{longtable}
\usepackage{fullpage}

\theoremstyle{plain}
\newtheorem{thm}{Theorem}[section]
\newtheorem*{thm*}{Theorem}
\newtheorem{thmalpha}{Theorem}

\newtheorem{cor}[thm]{Corollary}

\newtheorem*{cor*}{Corollary}
\newtheorem{conj}[thm]{Conjecture}

\newtheorem*{conj*}{Conjecture}
\newtheorem{lem}[thm]{Lemma}
\newtheorem{prop}[thm]{Proposition}

\newtheorem*{quest*}{Question}
\newtheorem*{claim*}{Claim}

\theoremstyle{definition}

\newtheorem{construction}[thm]{Construction}
\newtheorem{defn}[thm]{Definition}

\newtheorem{nota}[thm]{Notation}

\newtheorem{rem}[thm]{Remark}

\newtheorem{warn}[thm]{Warning}

\DeclareMathOperator{\coev}{coev}

\DeclareMathOperator{\End}{End}
\DeclareMathOperator{\ev}{ev}
\DeclareMathOperator{\Forget}{Forget}
\DeclareMathOperator{\Hom}{Hom}
\DeclareMathOperator{\mate}{mate}
\DeclareMathOperator{\op}{op}

\DeclareMathOperator{\id}{id}
\DeclareMathOperator{\Irr}{Irr}
\DeclareMathOperator{\Tr}{Tr}
\DeclareMathOperator{\tr}{tr}

\newcommand{\comment}[1]{}

\newcommand{\be}{\begin{enumerate}[label=(\arabic*)]}
\newcommand{\ee}{\end{enumerate}}

\newcommand{\cZ}{\mathcal{Z}}
\newcommand{\cD}{\mathcal{D}}

\newcommand{\cM}{\mathcal{M}}

\newcommand{\cC}{\mathcal{C}}

\newcommand{\bTr}{\mathbf{Tr}}
\newcommand{\Vect}{\mathsf{Vect}}

\newcommand{\Cat}{\mathsf{Cat}}

\def\semicolon{;}
\def\applytolist#1{
    \expandafter\def\csname multi#1\endcsname##1{
        \def\multiack{##1}\ifx\multiack\semicolon
            \def\next{\relax}
        \else
            \csname #1\endcsname{##1}
            \def\next{\csname multi#1\endcsname}
        \fi
        \next}
    \csname multi#1\endcsname}

\def\calc#1{\expandafter\def\csname c#1\endcsname{{\mathcal #1}}}
\applytolist{calc}QWERTYUIOPLKJHGFDSAZXCVBNM;
\def\bbc#1{\expandafter\def\csname bb#1\endcsname{{\mathbb #1}}}
\applytolist{bbc}QWERTYUIOPLKJHGFDSAZXCVBNM;
\def\bfc#1{\expandafter\def\csname bf#1\endcsname{{\mathbf #1}}}
\applytolist{bfc}QWERTYUIOPLKJHGFDSAZXCVBNM;
\def\sfc#1{\expandafter\def\csname s#1\endcsname{{\sf #1}}}
\applytolist{sfc}QWERTYUIOPLKJHGFDSAZXCVBNM;

\newcommand{\APA}{\mathsf{APA}}
\newcommand{\UAPA}{\mathsf{UAPA}}
\newcommand{\ModTens}{\mathsf{ModTens}}
\newcommand{\UModTens}{\mathsf{UModTens}}
\newcommand{\Fun}{{\sf Fun}}

\newcommand{\Bim}{{\sf Bim}}

\newcommand{\Hilb}{{\sf Hilb}}

\newcommand{\fdHilb}{{\sf Hilb_{fd}}}

\newcommand{\noshow}[1]{}
\newcommand{\MR}[1]{}

\usetikzlibrary{shapes}
\usetikzlibrary{cd}
\usetikzlibrary{backgrounds}
\usetikzlibrary{decorations,decorations.pathreplacing,decorations.markings}
\usetikzlibrary{fit,calc,through}
\usetikzlibrary{external}
\tikzset{
	super thick/.style={line width=3pt}
}
\tikzstyle{shaded}=[fill=red!10!blue!20!gray!30!white]
\tikzstyle{unshaded}=[fill=white]
\tikzstyle{empty box}=[circle, draw, thick, fill=white, opaque, inner sep=2mm]
\tikzstyle{annular}=[scale=.7, inner sep=1mm, baseline]
\tikzstyle{rectangular}=[scale=.75, inner sep=1mm, baseline=-.1cm]
\tikzstyle{mid>}=[decoration={markings, mark=at position 0.5 with {\arrow{>}}}, postaction={decorate}]
\tikzstyle{mid<}=[decoration={markings, mark=at position 0.5 with {\arrow{<}}}, postaction={decorate}]
\tikzstyle{over}=[double, draw=white, super thick, double=]
\tikzstyle{knot}=[preaction={super thick, white, draw}]

\newcommand{\roundNbox}[6]{
	\draw[rounded corners=5pt, very thick, #1] ($#2+(-#3,-#3)+(-#4,0)$) rectangle ($#2+(#3,#3)+(#5,0)$);
	\coordinate (ZZa) at ($#2+(-#4,0)$);
	\coordinate (ZZb) at ($#2+(#5,0)$);
	\node at ($1/2*(ZZa)+1/2*(ZZb)$) {#6};
}

\newcommand{\ncircle}[5]{
	\draw[very thick, #1] #2 circle (#3);
	\node at #2 {#5};
	\filldraw[red] ($#2+(#4:#3cm)$) circle (.05cm);
}

  \newcommand{\tikzmath}[2][]
     {\vcenter{\hbox{\begin{tikzpicture}[#1]#2
                     \end{tikzpicture}}}
     }

\newcommand{\plane}[3]{
	\pgfmathsetmacro{\planeWidth}{#2};
	\pgfmathsetmacro{\planeDepth}{#3};
	\draw[thick] ($ #1 + (-\planeDepth,\planeDepth) $) -- #1 -- ($ #1 + (\planeWidth,0) $) -- ($ #1 + (\planeWidth,0) + (-\planeDepth,\planeDepth) $) -- ($ #1 + (-\planeDepth,\planeDepth) $);
}

\newcommand{\CMbox}[6]{
	\coordinate (#1) at #2;
	\pgfmathsetmacro{\boxWidth}{#3};
	\pgfmathsetmacro{\boxHeight}{#4};
	\pgfmathsetmacro{\boxDepth}{#5};
	\draw[unshaded, very thick] ($(#1) + (\boxWidth,\boxHeight) $) -- ($(#1) + (\boxWidth,\boxHeight) - (\boxDepth,-\boxDepth) $) -- ($(#1) + (0,\boxHeight) - (\boxDepth,-\boxDepth) $) -- ($(#1) - (\boxDepth,-\boxDepth) $);
	\draw[unshaded, very thick] ($(#1) + (0,\boxHeight) $)  -- ($(#1) + (\boxWidth,\boxHeight) $) -- ($(#1) + (\boxWidth,0) $) -- (#1) -- ($(#1) - (\boxDepth,-\boxDepth) $);
	\draw[very thick] ($(#1) + (0,\boxHeight) - (\boxDepth,-\boxDepth) $) -- ($(#1) + (0,\boxHeight) $) -- (#1);
	\node at ($(#1) + 1/2*(\boxWidth,\boxHeight) $) {#6}; 
}

\newcommand{\straightTubeWithString}[4]{
	\coordinate (ZZq) at #1;
	\pgfmathsetmacro{\tubeLength}{#3};
	\pgfmathsetmacro{\tubeWidth}{#2};
	\pgfmathsetmacro{\buffer}{.05};	
	\fill[unshaded] ($ (ZZq) + (-\tubeLength,0) + 2*(0,-\buffer) $) -- ($ (ZZq) + 2*(0,-\buffer) $) arc(-90:90:{\tubeWidth+\buffer} and {2*(\tubeWidth+\buffer)}) -- ($ (ZZq) + (-\tubeLength,0) + 4*(0,\tubeWidth) + 2*(0,\buffer) $) arc(90:270:{\tubeWidth+\buffer} and {2*(\tubeWidth+\buffer)}) ;
	\draw[unshaded, thick]  ($ (ZZq) + (-\tubeLength,0) $) -- (ZZq) arc(-90:90:{\tubeWidth} and {2*\tubeWidth}) -- ($ (ZZq) + (-\tubeLength,0) + 4*(0,\tubeWidth) $) ;
	\draw[thick] ($ (ZZq) + (-\tubeLength,0) $) arc(-90:90:{\tubeWidth} and {2*\tubeWidth}) arc(90:270:{\tubeWidth} and {2*\tubeWidth});
	\draw[#4] ($(ZZq) + (\tubeWidth,0) + 2*(0,\tubeWidth) $) -- ($ (ZZq) + (-\tubeLength,0) + 2*(0,\tubeWidth) + (\tubeWidth,0)$);
}

\newcommand{\halfDottedEllipse}[3]{
	\draw[thick] #1 arc(-180:0:{#2} and {#3});
	\draw[thick, dotted] ($ #1 + 2*(#2,0)$) arc(0:180:{#2} and {#3});
}

\newcommand{\pairOfPants}[1]{
	\draw[thick] #1 .. controls ++(90:.8cm) and ++(270:.8cm) .. ($ #1 + (.7,1.5) $);
	\draw[thick] ($ #1 + (2,0) $) .. controls ++(90:.8cm) and ++(270:.8cm) .. ($ #1 + (2,0) + (-.7,1.5) $);
	\draw[thick] ($ #1 + (.6,0) $).. controls ++(90:.8cm) and ++(90:.8cm) .. ($ #1 + (1.4,0) $); 
	\halfDottedEllipse{($ #1 + (.7,1.5) $)}{.3}{.1}
	\halfDottedEllipse{#1}{.3}{.1}
	\halfDottedEllipse{($ #1 + (1.4,0) $)}{.3}{.1}
}

\newcommand{\topPairOfPants}[1]{
	\draw[thick] #1 .. controls ++(90:.8cm) and ++(270:.8cm) .. ($ #1 + (.7,1.5) $);
	\draw[thick] ($ #1 + (2,0) $) .. controls ++(90:.8cm) and ++(270:.8cm) .. ($ #1 + (2,0) + (-.7,1.5) $);
	\draw[thick] ($ #1 + (.6,0) $).. controls ++(90:.8cm) and ++(90:.8cm) .. ($ #1 + (1.4,0) $); 
	\draw[thick] ($ #1 + (1,1.5) $) ellipse (.3cm and .1cm);
	\halfDottedEllipse{#1}{.3}{.1}
	\halfDottedEllipse{($ #1 + (1.4,0) $)}{.3}{.1}
}

\newcommand{\straightTubeWithCap}[3]{
	\coordinate (ZZq) at #1;
	\pgfmathsetmacro{\tubeLength}{#3};
	\pgfmathsetmacro{\tubeWidth}{#2};
	\pgfmathsetmacro{\buffer}{.05};	

	\fill[unshaded] ($ (ZZq) + (-\tubeLength,0) + 2*(0,-\buffer) $) -- ($ (ZZq) + 2*(0,-\buffer) $) arc(-90:90:{\tubeWidth+\buffer} and {2*(\tubeWidth+\buffer)}) -- ($ (ZZq) + (-\tubeLength,0) + 4*(0,\tubeWidth) + 2*(0,\buffer) $) arc(90:270:{\tubeWidth+\buffer} and {2*(\tubeWidth+\buffer)}) ;
	\draw[unshaded, thick]  ($ (ZZq) + (-\tubeLength,0) $) -- (ZZq) arc(-90:90:{\tubeWidth} and {2*\tubeWidth}) -- ($ (ZZq) + (-\tubeLength,0) + 4*(0,\tubeWidth) $) ;
	\draw[thick] ($ (ZZq) + (-\tubeLength,0) $) arc(270:90:{2*\tubeWidth});
}

\newcommand{\straightTubeNoString}[3]{
	\coordinate (ZZq) at #1;
	\pgfmathsetmacro{\tubeLength}{#3};
	\pgfmathsetmacro{\tubeWidth}{#2};
	\pgfmathsetmacro{\buffer}{.05};	

	\fill[unshaded] ($ (ZZq) + (-\tubeLength,0) + 2*(0,-\buffer) $) -- ($ (ZZq) + 2*(0,-\buffer) $) arc(-90:90:{\tubeWidth+\buffer} and {2*(\tubeWidth+\buffer)}) -- ($ (ZZq) + (-\tubeLength,0) + 4*(0,\tubeWidth) + 2*(0,\buffer) $) arc(90:270:{\tubeWidth+\buffer} and {2*(\tubeWidth+\buffer)}) ;
	\draw[unshaded, thick]  ($ (ZZq) + (-\tubeLength,0) $) -- (ZZq) arc(-90:90:{\tubeWidth} and {2*\tubeWidth}) -- ($ (ZZq) + (-\tubeLength,0) + 4*(0,\tubeWidth) $) ;
	\draw[thick] ($ (ZZq) + (-\tubeLength,0) $) arc(-90:90:{\tubeWidth} and {2*\tubeWidth}) arc(90:270:{\tubeWidth} and {2*\tubeWidth});
}

\newcommand{\tensor}[6]{
	\draw[rounded corners=5pt, very thick, unshaded] ($ #1 - (#2,#2) $) rectangle ($ #1 + (#2,#2) $);
	\draw ($ #1 + 14/23*(0,#2) - 1/3*(#2,0) $) -- ($ #1 - 14/23*(0,#2) - 1/3*(#2,0) $);
	\draw ($ #1 + 14/23*(0,#2) + 1/3*(#2,0) $) -- ($ #1 - 14/23*(0,#2) + 1/3*(#2,0) $);
	\draw[thick, red] ($ #1 - 1/3*(#2,0) - 1/5*(#2,0) $) -- ($ #1 - 5/6*(#2,0) $);
	\draw[thick, red] ($ #1 + 1/3*(#2,0) - 1/5*(#2,0) $) .. controls ++(180:.2cm) and ++(0:.2cm) .. ($ #1 - 1/3*(#2,0) + 2/5*(0,#2) $) .. controls ++(180:.2cm) and ++(0:.2cm) .. ($ #1 - 5/6*(#2,0) $);
	\draw[very thick] #1 ellipse ( {5/6*#2} and {2/3*#2});
	\filldraw[very thick, unshaded] ($ #1 + 1/3*(#2,0) $) circle (1/5*#2);
	\filldraw[very thick, unshaded] ($ #1 - 1/3*(#2,0) $) circle (1/5*#2);
	\node at ($ #1 + (.2,0) - 1/3*(#2,0) - 1/3*(0,#2) $) {\scriptsize{$#3$}};
	\node at ($ #1 + (.2,0) - 1/3*(#2,0) + 1/3*(0,#2) $) {\scriptsize{$#4$}};
	\node at ($ #1 + (.2,0) + 1/3*(#2,0) - 1/3*(0,#2) $) {\scriptsize{$#5$}};
	\node at ($ #1 + (.2,0) + 1/3*(#2,0) + 1/3*(0,#2) $) {\scriptsize{$#6$}};
}

\newcommand{\evaluationMap}[3]{
	\draw[rounded corners=5pt, very thick, unshaded] ($ #1 - (#2,#2) $) rectangle ($ #1 + (#2,#2) $);
	\draw ($ #1 - 2/3*(#2,0) - 1/2*(0,#2) $) .. controls ++(90:{#2}) and ++(90:{#2}) .. ($ #1 + 2/3*(#2,0) - 1/2*(0,#2) $);
	\draw[very thick] #1 circle ({5/6*#2});
	\filldraw[red] ($ #1 - 5/6*(#2,0) $) circle ({1/10*#2});
	\node at ($ #1 - (0,.1) $) {\scriptsize{$#3$}};
}

\newcommand{\coevaluationMap}[3]{
	\draw[rounded corners=5pt, very thick, unshaded] ($ #1 - (#2,#2) $) rectangle ($ #1 + (#2,#2) $);
	\draw ($ #1 - 2/3*(#2,0) + 1/2*(0,#2) $) .. controls ++(270:{#2}) and ++(270:{#2}) .. ($ #1 + 2/3*(#2,0) + 1/2*(0,#2) $);
	\draw[very thick] #1 circle ({5/6*#2});
	\filldraw[red] ($ #1 - 5/6*(#2,0) $) circle ({1/10*#2});
	\node at ($ #1 + (0,.1) $) {\scriptsize{$#3$}};
}

\newcommand{\identityMap}[3]{
	\draw[rounded corners=5pt, very thick, unshaded] ($ #1 - (#2,#2) $) rectangle ($ #1 + (#2,#2) $);
	\draw ($ #1 + 5/6*(0,#2) $) -- ($ #1 - 5/6*(0,#2) $);
	\draw[very thick] #1 circle ({5/6*#2});
	\filldraw[red] ($ #1 - 5/6*(#2,0) $) circle ({1/10*#2});
	\node at ($ #1 + (.2,0) $) {\scriptsize{$#3$}};
}

\newcommand{\curvedTubeNoString}[4]{
	\filldraw[white, super thick, fill=white] #1 arc (-90:90:.2cm) arc (270:180:{#2 - .2}) -- ($#1 -(#2,0) - (.2,0) + (0,#2) + (0,.2) $) arc (180:270:{#2+.2});
	\filldraw[unshaded, thick] #1 arc (-90:90:.2cm) arc (270:180:{#2 - .2}) -- ($#1 -(#2,0) - (.2,0) + (0,#2) + (0,.2) $) arc (180:270:{#2+.2});
	\draw ($ #1 + (.18,.3) $) -- ($ #1 + (0,.3) $) arc (270:180:{#2 - .1});
	\draw ($ #1+ (.18,.1) $) -- ($ #1+ (0,.1) $) arc (270:180:{#2+.1});
	\roundNbox{unshaded}{($#1 - (#2,0) + (0,#2) + (0,.6) $)}{.4}{#3}{#3}{#4};
}

\newcommand{\multiplication}[5]{
	\draw[rounded corners=5pt, very thick, unshaded] ($ #1 - (#2,#2) $) rectangle ($ #1 + (#2,#2) $);
	\draw ($ #1 + 5/6*(0,#2) $) -- ($ #1 - 5/6*(0,#2) $);
	\draw[thick, red] ($ #1 + 1/3*(0,#2) - 1/5*(#2,0) $) -- ($ #1 - 2/3*(#2,0) $);
	\draw[thick, red] ($ #1 - 1/3*(0,#2) - 1/5*(#2,0) $) -- ($ #1 - 2/3*(#2,0) $);
	\draw[very thick] #1 ellipse ({2/3*#2} and {5/6*#2});
	\filldraw[very thick, unshaded] ($ #1 + 1/3*(0,#2) $) circle (1/5*#2);
	\filldraw[very thick, unshaded] ($ #1 - 1/3*(0,#2) $) circle (1/5*#2);
	\node at ($ #1 + (.2,0) + 5/8*(0,#2)$) {\scriptsize{$#5$}};
	\node at ($ #1 + (.2,0) $) {\scriptsize{$#4$}};
	\node at ($ #1 + (.2,0) - 5/8*(0,#2)$) {\scriptsize{$#3$}};
}

\newcommand{\tensorLeftIdEv}[4]{
	\draw[rounded corners=5pt, very thick, unshaded] ($ #1 - (#2,#2) $) rectangle ($ #1 + (#2,#2) $);
	\draw ($ #1 + 14/23*(0,#2) - 1/3*(#2,0) $) -- ($ #1 - 14/23*(0,#2) - 1/3*(#2,0) $);
	\draw ($ #1 + 1/3*(#2,0) $) -- ($ #1 - 14/23*(0,#2) + 1/3*(#2,0) $);
	\draw[thick, red] ($ #1 + 1/3*(#2,0) - 1/5*(#2,0) $) -- ($ #1 - 5/6*(#2,0) $);
	\draw[very thick] #1 ellipse ( {5/6*#2} and {2/3*#2});
	\filldraw[very thick, unshaded] ($ #1 + 1/3*(#2,0) $) circle (1/5*#2);
	\node at ($ #1 - (.2,0) - 1/3*(#2,0) + (0,.2) $) {\scriptsize{$#3$}};
	\node at ($ #1 - (.2,0) + 1/3*(#2,0) - 1/3*(0,#2) $) {\scriptsize{$#4$}};
}

\newcommand{\tensorRightIdCoev}[4]{
	\draw[rounded corners=5pt, very thick, unshaded] ($ #1 - (#2,#2) $) rectangle ($ #1 + (#2,#2) $);
	\draw ($ #1 + 14/23*(0,#2) - 1/3*(#2,0) $) -- ($ #1 - 1/3*(#2,0) $);
	\draw ($ #1 + 14/23*(0,#2) + 1/3*(#2,0) $) -- ($ #1 - 14/23*(0,#2) + 1/3*(#2,0) $);
	\draw[thick, red] ($ #1 - 1/3*(#2,0) - 1/5*(#2,0) $) -- ($ #1 - 5/6*(#2,0) $);
	\draw[very thick] #1 ellipse ( {5/6*#2} and {2/3*#2});
	\filldraw[very thick, unshaded] ($ #1 - 1/3*(#2,0) $) circle (1/5*#2);
	\node at ($ #1 + (.2,0) - 1/3*(#2,0) + 1/3*(0,#2) $) {\scriptsize{$#3$}};
	\node at ($ #1 + (.2,0) + 1/3*(#2,0) $) {\scriptsize{$#4$}};
}

\begin{document}
\title{Unitary anchored planar algebras}
\author{Andr\'e Henriques, David Penneys, and James Tener}
\date{}
\maketitle
\begin{abstract}
In our previous article [arXiv:1607.06041], we established an equivalence between pointed pivotal module tensor categories 
and anchored planar algebras.
This article introduces the notion of unitarity for both module tensor categories and anchored planar algebras, and establishes the unitary analog of the above equivalence.
Our constructions use Baez's 2-Hilbert spaces (i.e., semisimple $\rm C^*$-categories equipped with unitary traces), the unitary Yoneda embedding, and the notion of unitary adjunction for dagger functors between 2-Hilbert spaces.
\end{abstract}

\tableofcontents

\section{Introduction}

Planar algebras were introduced by Vaughan Jones in \cite{MR4374438}.
A planar algebra $\cP$ is a collection of vector spaces $\cP[n]$, called \emph{box-spaces}, indexed by the nonnegative integers, along with a linear map
$$
Z(T): \cP[n_1]\otimes \cdots \otimes \cP[n_k]\longrightarrow \cP[n_0]
$$
for every planar tangle $T$. 
These maps are required to be compatible with the operation of composition of tangles in the sense that
$
Z(S\circ_i T)
=
Z(S)\circ (\id\otimes\ldots\otimes\id \otimes Z(T)\otimes\id\otimes\ldots\otimes\id)$.
See \S\ref{sec: Planar algebras} for more details.

In \cite{MR4528312}, we internalised the notion of planar algebra to the context of a pivotal braided tensor category $\cV$, and called the resulting notion an \emph{anchored planar algebra}.
The box-spaces $\cP[n]$ of an anchored planar algebra are now objects of our ambient category $\cV$, and we have morphisms
$$
Z(T): \cP[n_1]\otimes \cdots \otimes \cP[n_k]\longrightarrow \cP[n_0]
$$
for every anchored planar tangle. For example:
$$
Z\left(
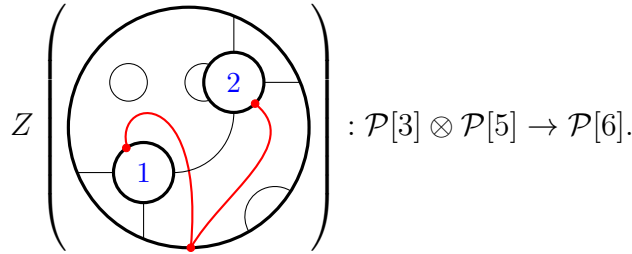
\begin{tikzpicture}[baseline =-.1cm]
\pgftransformyscale{-1}
	\coordinate (a) at (0,0);
	\coordinate (b) at ($ (a) + (1.4,1) $);
	\coordinate (c) at ($ (a) + (.6,-.6) $);
	\coordinate (d) at ($ (a) + (-.6,.6) $);
	\coordinate (e) at ($ (a) + (-.8,-.6) $);
	\ncircle{}{(a)}{1.6}{89}{}
	\draw[thick, red] (d) .. controls ++(225:1.2cm) and ++(275:2.6cm) .. ($ (a) + (89:1.6) $);
	\draw[thick, red] ($ (a) + (89:1.6) $) to[out=-60, in=40] (.9,-.3);
	\draw (60:1.6cm) arc (150:300:.4cm);
	\draw ($ (c) + (0,.4) $) arc (0:90:.8cm);
	\draw ($ (c) + (-.4,0) $) circle (.25cm);
	\draw ($ (d) + (0,.88) $) -- (d) -- ($ (d) + (-.88,0) $);
	\draw ($ (c) + (0,-.88) $) -- (c) -- ($ (c) + (.88,0) $);
	\draw (e) circle (.25cm);
	\ncircle{unshaded}{(d)}{.4}{235}{}
	\ncircle{unshaded}{(c)}{.4}{225+180}{}
	\node[blue] at (c) {\small 2};
	\node[blue] at (d) {\small 1};
\end{tikzpicture}
\right)
:
\cP[3]\otimes \cP[5] \to \cP[6].
$$
The full definition is given in \S\ref{sec:APA}. 

We then proceeded to establish
an equivalence of categories
\[
\left\{\,\text{\rm Anchored planar algebras in $\cV$}\left\}
\,\,\,\,\cong\,\,
\left\{\,\parbox{4.5cm}{\rm Pointed pivotal module tensor categories over $\cV$}\,\right\}\right.\right.\!\!
\]
between the category $\APA$ of anchored planar algebras in $\cV$, and the 
category $\ModTens_*$ of pointed\footnote{Here, ``pointed'' means equipped with a distinguished symmetrically self-dual object that generates the module tensor category.} pivotal module tensor categories over $\cV$ whose action functor admits a right adjoint.
We refer the reader to \S\ref{sec:ModTens} below for definitions of these notions.

In this paper we further specialise/generalise the notion of anchored planar algebra to the context when $\cV$ is a \emph{unitary} pivotal braided tensor category, and we prove an analog of the above result.
This is motivated by enriched subfactor theory \cite{MR3623255}, bicommutant categories \cite{MR3747830,MR3663592,MR4581741}, topological orders \cite{MR4640433},
and higher unitary categories \cite{MR4419534,2208.14992}. 

Let $\cV$ be a braided unitary tensor category equipped with a unitary dual functor (the latter is the unitary analog of a pivotal structure \cite{MR2091457,MR2767048,MR4133163}; see \S\ref{sec:UDF and SphericalStates} below for more details).
A unitary module tensor category $\cC$ is called \emph{pivotal} if it is equipped with a unitary dual functor, compatibly with the action of $\cV$.
Our main result is the unitary analog of Theorem \cite[Thm.~A]{MR4528312}:

\begin{thmalpha}
\label{thm:MainUAPA}
There is an equivalence of categories\footnote{
The 2-category of unitary pointed module tensor categories is equivalent to a 1-category. Moreover, the unique 2-morphism between two 1-morphisms, when it exists, is always unitary.}
\[
\left\{\,\parbox{4.5cm}{\rm Unitary anchored planar algebras in $\cV$}\left\}
\,\,\,\,\cong\,\,
\left\{\,\parbox{7.2cm}{\rm Unitary pointed pivotal module tensor\footnotemark\;\! categories over $\cV$, with chosen state}\,\right\}\right.\right.\!\!.
\]
Moreover,
\footnotetext{We do not require the unit of our module tensor category to be simple. 
So, strictly speaking, it is a unitary module \emph{multitensor} category.}\!\!
when $\cV$ is ribbon,
spherical unitary anchored planar algebras correspond under this equivalence to unitary module tensor categories whose chosen state is spherical.
\end{thmalpha}

\subsection{Motivations for this article}
\label{sec:Motivations}

\subsubsection{Enriched subfactor theory}

The development of planar algebras \cite{MR4374438} is intimately linked to subfactor theory. 
We expect a similar relation to hold between anchored planar algebras and \emph{enriched} subfactor theory.
In their paper \cite{MR3687214}, Corey Jones and the second author introduced a notion of $\rm W^*$-algebra object $A$ internal to a unitary tensor category $\cV$.
One could imagine formulating an analog of the notion of $\rm II_1$ factor internal to $\cV$ (perhaps just the condition that the neutral part of $A$, i.e. $\cV(1\to A)$, is a $\rm II_1$ factor) and, similarly, a notion of subfactor internal to $\cV$.

We conjecture that, in this context, the correct analog of the standard invariant is that of a unitary anchored planar algebra:


\begin{conj}
The enriched standard invariant of a subfactor internal to $\cV$ is a 2-shaded unitary anchored planar algebra in $Z(\cV)$.
Finite index, finite depth hyperfinite $\rm II_1$ subfactors internal to $\cV$ are classified by their enriched standard invariants\footnote{The relevant anchored planar algebras should satisfy $\cP[0]=I(1)$, where $I:\cV\to Z(\cV)$ is the adjoint of the forgetful functor $\cZ(\cV)\to \cV$.}.
\end{conj}

As an example, in \cite[Rem.~6.1]{MR3623255},
Jaffe and Liu construct a subfactor from the inductive limit of parafermion algebras, and  obtain the parafermion \emph{subfactor planar para algebra} from its `graded standard invariant.'
Planar para algebras are anchored planar algebras internal to $\cV=\Vect(\bbZ/N)$ with a particular braiding (see \cite[Examples 3.7 and 3.8]{MR4528312}), and the parafermion planar para algebra corresponds to a Tambara-Yamagami module tensor category over $\cV$.


There is a classification of $\rm II_1$ subfactors with index less than 4 \cite{MR0696688,MR999799,MR1308617,MR1193933,MR1313457} in terms of ADE Coxeter-Dynkin diagrams (where $D_{\text{odd}}$ and $E_{7}$ do not occur, and $E_6$ and $E_8$ occur twice).
For subfactors enriched over super vector spaces, the non-simply laced Coxeter-Dynkin diagrams $C_{\text{even}}$ and $F_4$ also appear \cite[\S6-7]{1709.01941}.



\subsubsection{Internal structure of bicommutant categories}

Bicommutant categories were introduced in \cite{MR3747830}, as higher categorical analogs of von Neumann algebras. 
The simplest example of a bicommutant category is $\Bim(M)$, the tensor category of all bimodules over a von Neumann algebra $M$.
Examples corresponding to unitary fusion categories were constructed in \cite{MR3663592}, and further studied in \cite{MR4581741}.
Examples corresponding to conformal nets were constructed in \cite{1701.02052}.

When $M=R$ is a hyperfinite $\rm II_1$, $\rm II_\infty$, or $\rm III_1$ factor, 
there is a well-known correspondence \cite{MR1055708,MR1339767,MR4374438,MR4236062,2010.01067} between conjugacy classes of finite depth $R$-$R$ bimodules and isomorphism classes of unitary finite depth planar algebras\footnote{Unoriented planar algebras correspond to symmetrically self-dual bimodules. 
Arbitrary finite depth bimodules are classified by planar algebras whose strands are oriented.}:
\[
\left\{
\parbox{5.1cm}{
Finite depth $R$-$R$-bimodules
}
\!\left\}\left/\text{conj.}
\,\leftrightarrow\,\,\,
\left\{
\parbox{4.5cm}{
Connected finite depth \centerline{unitary planar algebras}
}
\right\}\right.\right.\right.
\!\!\bigg/\text{iso.}
\]
We conjecture that, under suitable assumptions, conjugacy clases of finite depth objects in a bicommutant category $\cT$ are in bijective correspondence with finite depth unitary anchored planar algebras in $Z(\cT)$.

\begin{conj}\label{conjIntro}
Let $\cT$ be a bicommutant category whose Drinfeld center $Z(\cT)$ is fusion. Then, under suitable assumptions, there exists a natural bijective correspondence
\[
\left\{
\parbox{4.6cm}{
\rm Finite depth objects of $\cT$
}
\!\left\}\left/\text{\rm conj.}
\,\leftrightarrow\,\,\,
\left\{
\parbox{6.2cm}{
\rm Connected finite depth unitary \centerline{\rm anchored planar algebras in $\cZ(\cT)$}
}
\right\}\right.\right.\right.
\!\!\bigg/\text{\rm iso.}
\]
\end{conj}
\noindent
For bicommutant categories coming from fusion categories, Conjecture~\ref{conjIntro} was proven in our recent paper \cite{2307.13822}.

\subsubsection{(2+1)D topological orders}

Topological order is a phenomenon in (theoretical) condensed matter physics beyond Landau's symmetry breaking paradigm. 
In $(2$+$1)$ dimensions, the low energy effective field theory of a topologically ordered phase of matter is a topological quantum field theory, and the low energy localised excitations form a unitary modular tensor category (UMTC).


\begin{conj}
Let $\cX$ and $\cY$ be $(2+1)$D topological orders with the same anomaly, let $\cM$ be a topological domain wall between them, and let $m$ be a symmetrically self-dual point-like excitation that lives on $\cM$.
Then the collection of objects
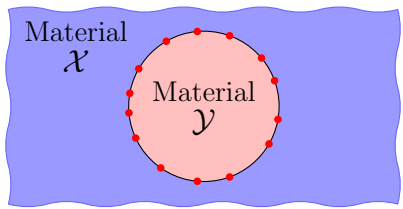
\[
\begin{tikzpicture}
\fill[fill=blue!40, draw = blue!60] (-2.6,-1.3)
arc(-105:-75:1) arc(105:75:1)
arc(-105:-75:1) arc(105:75:1)
arc(-105:-75:1) arc(105:75:1)
arc(-105:-75:1) arc(105:75:1)
arc(-105:-75:1) arc(105:75:1)
arc(-15:15:1) arc(195:165:1)
arc(-15:15:1) arc(195:165:1)
arc(-15:15:1)
arc(-75:-105:1) arc(75:105:1)
arc(-75:-105:1) arc(75:105:1)
arc(-75:-105:1) arc(75:105:1)
arc(-75:-105:1) arc(75:105:1)
arc(-75:-105:1) arc(75:105:1)
arc(15:-15:1) arc(165:195:1)
arc(15:-15:1) arc(165:195:1)
arc(15:-15:1)
;
\node[scale =.9] at (-1.7,1) {\rm Material};
\node at            (-1.7,.6) {$\cX$};
\filldraw[fill=pink] circle (1) ;
\foreach \angl in {20,40,70,95, 120, 150, 170, -10, -30, -70, -95, -125, -155, -175}
{\fill[fill=red] (\angl:1) circle (.05) ; }
\node[scale =.9] at (0,.2) {\rm Material};
\node at (0,-.2) {$\cY$};
\end{tikzpicture}
\]
(where the red dots represent $m$)
in the unitary modular tensor category associated to $\cX$ are the box-spaces of a unitary anchored planar algebra.
\end{conj}

More generally, if $\cU$ is any unitary 3-category (a notion yet to be defined) and $c:v\to w$ is a dualizable 1-morphism between dualizable objects,
then $\cV:=\End(1_{v})$ should be a unitary modular ribbon category and $\cC:=\End(c)$ should be a unitary $\cV$-module multifusion category.
Moreover, for $m:c\Rightarrow c$ symmetrically self-dual, the objects $\underline{\Hom}(1_c, m^{\otimes k})\in\cV$ should form the box-spaces of a unitary anchored planar algebra.

Indeed, it should be the case that (2+1)D topological orders with the a given anomaly form a unitary 3-category whose 1-morphisms are topological domain walls and whose higher morphisms correspond to topological defects of higher codimension. 
More precisely, \cite{MR4640433} posits that for a UMTC $\cA$ whose Witt class represents an anomaly, the 3-category of $\cA$-enriched unitary fusion categories describes this putative 3-category of (2+1)D topological orders.

\subsubsection{Unitary higher categories}

Over the past decades, a theory of higher linear algebra has
emerged from work of many authors, e.g., \cite{MR1355899,MR1448713,MR2555928,MR3971584}.
The higher $n$-category $n\Vect$ of $n$-\emph{vector spaces} may be used as target for $n$-dimensional topological quantum field
theories, where the fully dualizable objects give fully extended
theories by the corbordism hypothesis. 
Until recently, definitions of
$n\Vect$ were bespoke, chosen to agree with well-known state-sum models. 
The recent breakthrough \cite{1905.09566} provides a uniform framework to
construct (the fully dualizable part of) $n\Vect$ via a formal inductive procedure starting from just the complex numbers. 

Starting with $\bbC$, a commutative algebra,
we can deloop and complete to obtain the 1-category $\mathsf{Vect_{fd}}$.
Since $\Vect$ is a symmetric monoidal tensor category,
we can deloop and complete again to obtain the 2-category $\mathsf{Alg_{fd}^{sep}}$ of finite dimensional separable algebras, bimodules, and intertwiners.
Assigning to an algebra its category of modules gives an equivalence to $2\Vect$, the 2-category of finite semisimple categories, linear functors, and natural transformations.
Since $2\Vect$ is a symmetric monoidal 2-category, we may deloop and complete again to obtain the 3-category of (separable) multifusion categories \cite{MR4444089}, which is equivalent to the 3-category of semisimple 2-categories \cite{1812.11933,MR4372801}.

\begin{equation}
\label{eq:HigherVectorSpaces}
\bbC
\xrightarrow{\Sigma}
\mathsf{Vect_{fd}}
\xrightarrow{\Sigma}
\mathsf{Alg_{fd}^{sep}}
\simeq
\mathsf{LinCat_{fin}^{ss}}
\xrightarrow[\text{\cite{MR4444089}}]{\Sigma}
\mathsf{MultFusCat}
\underset{\text{\cite{MR4372801}}}{\simeq}
\mathsf{Lin2Cat_{fin}^{ss}}
\xrightarrow{\Sigma}
\cdots
\end{equation}

For physical applications in topologically ordered phases of matter, it is important to have a version of the above construction that incorporates unitary structures at all levels.
However, extending this construction to the unitary setting is quite delicate. 
Whereas dualizability is a \emph{property} in the non-unitary setting, it is additional \emph{structure} in the unitary setting \cite{MR2091457,MR4133163}.
Thus $n\Hilb$ should come with canonical involutions corresponding to duals at various levels.

Writing this article has allowed us to clarify these notions for $2\Hilb$. Here, a \emph{2-Hilbert space} \cite{MR1448713} is a linear category (enriched in finite dimensional vector spaces) equipped with Hilbert space structures on hom spaces satisfying
$$
\langle f , h \circ g^\dag \rangle_{b\to c}
=
\langle f\circ g , h \rangle_{a\to c}
=
\langle g , f^\dag \circ h \rangle_{a\to b}
\qquad\qquad
\forall\,
g:a\to b,\,
f:b\to c,\,
h: a\to c.
$$
As pointed out in \cite[Rem.~3.6.1]{MR4598730}, 2-Hilbert spaces are the natural setting for defining the unitary Yoneda embedding \cite{MR4079745}.
In turn, we have a well-behaved notion of \emph{unitary adjunction} for dagger functors, which we describe in \S\ref{sec:UnitaryYoneda} below.
In Proposition \ref{prop:UnitaryAdjointsGivesUDF} (see also Remark \ref{rem:UnitaryAdjointsGivesUDF}) we prove that unitary adjunction gives a canonical unitary dual functor on $2\Hilb$.
We expect this to be a key ingredient in defining higher dagger idempotents, for the unitary analog of \eqref{eq:HigherVectorSpaces}.




\subsection{Outline}

In \S\ref{sec:Preliminaries}, we provide the preliminary background for this article, together with some new contributions for unitary and involutive categories.
We discuss the unitary Yoneda embedding and unitary adjunction in \S\ref{sec:UnitaryYoneda}, 
and we review unitary dual functors and spherical states in \S\ref{sec:UDF and SphericalStates}.
In \S\ref{sec:InvolutiveFunctors}, 
we use the graphical calculus for the 2-category $\Cat$ to transport involutive structures through adjunctions.

In \S\ref{sec:UMTC}, we review the notion of pivotal module tensor category over a braided pivotal category before introducing the unitary counterpart in \S\ref{sec:Unitary module tensor categories}.
We briefly review the graphical calculus of strings on tubes for the categorified trace from \cite{MR3578212}, and we study the adjoint of the unit map in \S\ref{sec:idag}.
In the spherical setting, we prove unitarity of the traciator in \S\ref{sec:UnitaryTraciator}.

In \S\ref{sec:UAPA}, we review the notions of planar algebra and anchored planar algebra before introducing the unitary counterparts in \S\ref{subsec:UAPA}. 
We show that for a spherical unitary anchored planar algebra, the adjoint is compatible with the inside-out reflection of tangles in Proposition~\ref{prop: inside out reflection}.
Finally, in \S\ref{sec:Equivalence}, we prove our main Theorem \ref{thm:MainUAPA}.

\subsection{Acknowledgements}

The authors would like to thank
Peter Huston
and
David Reutter
for helpful conversations.
David Penneys was supported by NSF DMS 1500387/1655912, 1654159, and 2154389.
James Tener was supported by Australian Research Council Discovery Project DP200100067, as well as by the Max Planck Institute for Mathematics, Bonn, during the initial stages of this work.


\section{Preliminaries}
\label{sec:Preliminaries}

Our standard references for (unitary) tensor categories include \cite{MR1444286,MR808930,MR2767048,MR3242743,MR3578212,MR4133163}.

Let $\Cat_\bbC$ be the 2-category of $\bbC$-linear categories, linear \emph{and} anti-linear functors, and natural transformations. 
We only allow natural transformations between functors if they are either both linear or both anti-linear.
In this section, we shall often use the graphical calculus for $\Cat_\bbC$, denoting linear categories by two-dimensional regions, functors by strands, and natural transformations by junctures.
As in \cite{MR4369356}, we may identify a category $\cA$ with the category of $\bbC$-linear functors $\Vect\to \cA$.
Given a functor $F: \cA\to \cB$ and objects $a\in \cA$ and $b\in \cB$,
this allows us to give a graphical representation e.g. for a morphism $f:F(a)\to b$, as follows:
$$
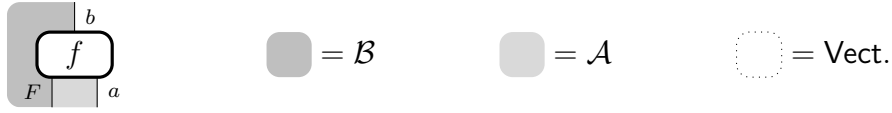
\tikzmath{
\begin{scope}
\clip[rounded corners=5pt] (-.9,-.7) rectangle (.9,.7);
\fill[gray!50] (0,.7) -- (0,.3) -- (-.3,-.3) -- (-.3,-.7) -- (-.9,-.7) -- (-.9,.7);
\fill[gray!30] (-.3,-.7) rectangle (.3,-.3);
\end{scope}
\draw (0,.3) --node[right]{$\scriptstyle b$} (0,.7);
\draw (.3,-.3) --node[right]{$\scriptstyle a$} (.3,-.7);
\draw (-.3,-.3) --node[left]{$\scriptstyle F$} (-.3,-.7);
\roundNbox{fill=white}{(0,0)}{.3}{.2}{.2}{$f$}
}
\qquad\qquad
\tikzmath{
\fill[gray!50, rounded corners = 5pt] (0,0) rectangle (.6,.6);
}
=
\cB
\qquad\qquad
\tikzmath{
\fill[gray!30, rounded corners = 5pt] (0,0) rectangle (.6,.6);
}
=
\cA
\qquad\qquad
\tikzmath{
\draw[rounded corners=5, thin, dotted] (0,0) rectangle (.6,.6);
}
=
\Vect.
$$
Given an adjunction $F\dashv G$ between two functors $F: \cA\to \cB$ and $G: \cB\to \cA$,
recall that $f: F(a) \to b$ and $g:A\to G(b)$ are called \emph{mates} if they are mapped to one another under the natural isomorphism
\begin{align*}
\cB\big(F(a) \to b\big)
&
\,\,\cong\,\,
\cA\big(a \to G(b)\big) 
\\
f \qquad&\,\,\leftrightarrow\,\,\,\,\,\,\, \mate(f)
\\
\mate(g) \,\,\,\,&\,\,\leftrightarrow\,\,\qquad g.
\end{align*}
Mates are represented in the graphical calculus as follows:
$$
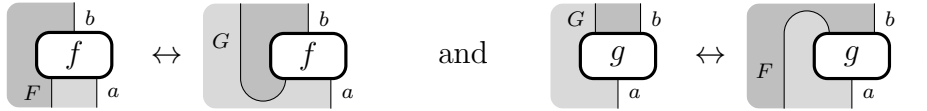
\tikzmath{
\begin{scope}
\clip[rounded corners=5pt] (-.9,-.7) rectangle (.9,.7);
\fill[gray!50] (0,.7) -- (0,.3) -- (-.3,-.3) -- (-.3,-.7) -- (-.9,-.7) -- (-.9,.7);
\fill[gray!30] (-.3,-.7) rectangle (.3,-.3);
\end{scope}
\draw (0,.3) --node[right]{$\scriptstyle b$} (0,.7);
\draw (.3,-.3) --node[right]{$\scriptstyle a$} (.3,-.7);
\draw (-.3,-.3) --node[left]{$\scriptstyle F$} (-.3,-.7);
\roundNbox{fill=white}{(0,0)}{.3}{.2}{.2}{$f$}
}
\leftrightarrow\,\,
\tikzmath{
\begin{scope}
\clip[rounded corners=5pt] (-1.4,-.7) rectangle (.9,.7);
\fill[gray!50] (0,.7) -- (0,.3) -- (-.3,-.3) arc (0:-180:.3cm) -- (-.9,.7);
\fill[gray!30] (-.3,-.3) arc (0:-180:.3cm) -- (-.9,.7) -- (-1.4,.7) -- (-1.4,-.7) -- (.3,-.7) -- (.3,-.3);
\end{scope}
\draw (0,.3) --node[right]{$\scriptstyle b$} (0,.7);
\draw (.3,-.3) --node[right]{$\scriptstyle a$} (.3,-.7);
\draw (-.3,-.3) arc (0:-180:.3cm) --node[left]{$\scriptstyle G$} (-.9,.7);
\roundNbox{fill=white}{(0,0)}{.3}{.2}{.2}{$f$}
}
\qquad\text{and}\qquad
\tikzmath[yscale=-1]{
\begin{scope}
\clip[rounded corners=5pt] (-.9,-.7) rectangle (.9,.7);
\fill[gray!30] (0,.7) -- (0,.3) -- (-.3,-.3) -- (-.3,-.7) -- (-.9,-.7) -- (-.9,.7);
\fill[gray!50] (-.3,-.7) rectangle (.3,-.3);
\end{scope}
\draw (0,.3) --node[right]{$\scriptstyle a$} (0,.7);
\draw (.3,-.3) --node[right]{$\scriptstyle b$} (.3,-.7);
\draw (-.3,-.3) --node[left]{$\scriptstyle G$} (-.3,-.7);
\roundNbox{fill=white}{(0,0)}{.3}{.2}{.2}{$g$}
}
\leftrightarrow\,\,
\tikzmath[yscale=-1]{
\begin{scope}
\clip[rounded corners=5pt] (-1.4,-.7) rectangle (.9,.7);
\fill[gray!30] (0,.7) -- (0,.3) -- (-.3,-.3) arc (0:-180:.3cm) -- (-.9,.7);
\fill[gray!50] (-.3,-.3) arc (0:-180:.3cm) -- (-.9,.7) -- (-1.4,.7) -- (-1.4,-.7) -- (.3,-.7) -- (.3,-.3);
\end{scope}
\draw (0,.3) --node[right]{$\scriptstyle a$} (0,.7);
\draw (.3,-.3) --node[right]{$\scriptstyle b$} (.3,-.7);
\draw (-.3,-.3) arc (0:-180:.3cm) --node[left]{$\scriptstyle F$} (-.9,.7);
\roundNbox{fill=white}{(0,0)}{.3}{.2}{.2}{$g$}
}\,,
$$
where the cup and cap represent the
unit and counit of the adjunction.

The operations of taking the mate are natural with respect to pre-composition and post-composition by another morphism:
\begin{enumerate}[label=(M\arabic*)]
\item
\label{Mate:LeftMateIdentity}
$\mate(f_2\circ f_1) = G(f_2) \circ \mate(f_1)$
for all
$f_1:F(a)\to b_1$ and $f_2: b_1\to b_2$.
\item
\label{Mate:RightMateIdentity}
$\mate(g_2\circ g_1) = \mate(g_2) \circ F(g_1)$
for all
$g_1: a_1\to a_2$ and $g_2:a_2\to G(b)$.
\end{enumerate}

\subsection{The unitary Yoneda lemma and unitary adjunctions}
\label{sec:UnitaryYoneda}

In this section, we work with semisimple $\rm C^*$ categories. (Note that every Cauchy complete $\rm C^*$ category with finite dimensional hom spaces is semisimple \cite[\S3.1.1]{MR4598730}.)

\begin{defn}\label{def: 2.1}
A unitary trace on a semisimple $\rm C^*$ category $\cA$ is a collection of linear maps $\Tr_a: \cA(a\to a) \to \bbC$, for all $a\in \cA$, satisfying 
\begin{itemize}
\item 
$\Tr_a(f\circ g)= \Tr_b(g\circ f)$ for all $f: a\to b$ and $g: b\to a$, and
\item 
The sesquilinear forms $\langle f, g\rangle_{a\to b}:=\Tr_a(g^\dag \circ f)$ are positive definite.
\end{itemize}

\end{defn}

The above inner products satisfy
\begin{equation}
\label{eq:2HilbertSpaceEqualities}
\langle f , h \circ g^\dag \rangle_{b\to c}
=
\langle f\circ g , h \rangle_{a\to c}
=
\langle g , f^\dag \circ h \rangle_{a\to b}
\qquad\quad
\forall\,
g:a\to b,\,
f:b\to c,\,
h: a\to c,
\end{equation}
equivalently
\begin{equation}
\label{eq:2HilbertSpaceEqualities BIS}
(-\circ g)^\dagger=-\circ g^\dagger
\qquad\text{and}\qquad
(f\circ -)^\dagger=f^\dagger \circ -,
\end{equation}
and thus equip $\cA$ with the structure of a \emph{2-Hilbert space} in the sense of \cite{MR1448713}.
Conversely, one may recover the trace from the 2-Hilbert space structure 
by the formula $\Tr_a(f):=\langle f, \id_a\rangle_{a\to a}$.
The notion of semisimple $\rm C^*$ category equipped with a unitary trace is thus equivalent to the notion of $2$-Hilbert space. 

The first condition in \eqref{eq:2HilbertSpaceEqualities BIS} implies that for all $c\in \cA$, the functor $\cA(-\to c): \cA^{\op} \to \Hilb$ is a $\dagger$-functor, and the second equality implies that the \emph{unitary Yoneda embedding}
\begin{equation}\label{eq: Y Functor}
\cA\,\,\,\hookrightarrow\,\,\, \Fun^\dag(\cA^{\op}\to \Hilb)
\end{equation}
is a $\dagger$-functor (where $\Fun^\dag(\cA^{\op}\to \Hilb)$ denotes the $\dagger$-category of $\dagger$-functors $\cA^{\op}\to \Hilb$).

The above facts were first observed in \cite[Rem.~3.61 and footnote]{MR4598730}, and
we shall refer to them collectively as the \emph{unitary Yoneda lemma}.
Note that the essential image of \eqref{eq: Y Functor} is the same as its unitary essential image (using polar decomposition in $\Fun^\dag(\cA^{\op}\to \Hilb)$), so a $\dagger$-functor $\cA^{\op} \to \Hilb$ is unitarily representable if and only if the underlying functor is representable.
Given a $\dagger$-functor $F\in \Fun^\dag(\cA^{\rm op}\to \Hilb)$ in the essential image of \eqref{eq: Y Functor}, a representing object is given by
\begin{equation}
\label{eq:CoendRepresentingObject}
c := \bigoplus_{a\in \Irr(\cA)} d_a^{-1}F(a)\otimes a,
\end{equation}
where $d_a:=\Tr^\cA_a(\id_a)$, $\Irr(\cA)$ is a basis\footnote{A \emph{basis} of a semisimple linear category is a collection of representatives of each isomorphism class of simple objects.} of $\cA$, and we use the notation $\lambda H$ to denote the Hilbert space with the same underlying vector space as $H$ and inner product $\lambda\langle \,\cdot\,,\,\cdot\,\rangle_H$. 
The rescaling of the inner product of $F(a)$ ensures  the unitarity of the isomorphism $\cA(-\to c)\cong F$ given by
\begin{equation}
\label{eq:UnitaryRepresentability}
\cA\left(b\to \bigoplus_{a\in \Irr(\cA)} d_a^{-1}F(a)\otimes a\right)
\cong
d_b^{-1}F(b)\otimes \cA(b\to b)
\xrightarrow{\id_{F(b)}\otimes(\id_b\mapsto 1)}
F(b)
\end{equation}
for $b\in \Irr(\cA)$.



\begin{rem}
If $\cA$ is finitely semisimple and $\Hilb$ is taken to mean finite dimensional Hilbert spaces, then \eqref{eq: Y Functor} is a $\dag$-equivalence.
\end{rem}

\begin{defn}\label{def: anti+unitary adjunction}
Let $(\cA,\Tr^\cA)$ and $(\cB,\Tr^\cB)$ be semisimple $\rm C^*$ categories equipped with unitary traces.
A \emph{unitary adjunction}, denoted $F\dashv^\dag G$, consists of linear $\dagger$-functors $F: \cA \to \cB$ and $G: \cB \to \cA$ and a family of unitary and natural isomorphisms
\begin{equation}\label{eq: (anti)unitary adj}
\psi_{a,b}:\cB(F(a) \to b) \xrightarrow{\,\,\cong\,\,} \cA(a\to G(b)).
\end{equation}
If $F$ and $G$ are instead antilinear $\dagger$-functors, and
\eqref{eq: (anti)unitary adj} is antiunitary, then adjunction
is called an \emph{antiunitary adjunction}.
\end{defn}

\begin{rem}
In the above definition, if we had merely required $F$ and $G$ to be (anti)linear functors, they would nevertheless automatically be $\dagger$-functors.
We prove that $G$ is a dagger functor (the argument for $F$ is similar).
For any $f \in \cB(b_1\to b_2)$, the following diagram commutes:
$$
\begin{tikzcd}[column sep = 4em]
\cA(a\to G(b_1)) \arrow[d,"\psi_{a,b_1}^{-1}","\cong"'] \arrow[r,"G(f) \circ -"]& \cA(a\to G(b_2))
\\
\cB(F(a)\to b_1) \arrow[r, "(f\circ -)"] & \cA(F(a)\to b_2). \arrow[u,"\psi_{a,b_2}"',"\cong"]
\end{tikzcd}
$$
Applying $\dagger$ to all the arrows, by \eqref{eq:2HilbertSpaceEqualities BIS}, 
the following diagram also commutes:
$$
\begin{tikzcd}[column sep = 4em]
\cA(a\to G(b_1))  & \cA(a\to G(b_2)) \arrow[l,"G(f)^\dag \circ -"']  \arrow[d,"\psi_{a,b_2}^{-1}", "\cong"']
\\
\cB(F(a)\to b_1) \arrow[u,"\psi_{a,b_1}"', "\cong"] & \cA(F(a)\to b_2). \arrow[l, "(f^\dagger\circ -)"']
\end{tikzcd}
$$
As $G(f^\dagger)\circ-$ also makes that second diagram commute, 
$G(f^\dagger)\circ- = G(f)^\dagger\circ-$,
and hence $G(f^\dag)=G(f)^\dag$.

Similarly,
whenever $F\dashv G$ are adjoint functors between linear categories,
if \eqref{eq: (anti)unitary adj} is (anti)linear,
then $F$ and $G$ are automatically (anti)linear functors.

The above facts are reminiscent of the well-known fact that if $H,K$ are Hilbert spaces and $S: H \to K$ and $T: K\to H$ are two functions such that $\langle S\eta, \xi\rangle = \langle \eta, T\xi\rangle$ for all $\eta\in H$ and $\xi \in K$, then $S$ and $T$ are automatically bounded linear maps.
\end{rem}

\begin{lem}\label{lem: unitary right adjoints}
Let $(\cA,\Tr^\cA)$ and $(\cB,\Tr^\cB)$ be semisimple $\rm C^*$ categories with unitary traces, and let $F: \cA\to \cB$ be a linear $\dagger$-functor that has a right adjoint. Then $F$ also has a unitary right adjoint. 
The unitary right adjoint $G$ is unique up to unique unitary natural isomorphism, 
and given by
\begin{equation}\label{eq: formula for G}
G(b)\,= \bigoplus_{a\in\Irr(\cA)} d_a^{-1}\cB\big(F(a)\to b\big) \otimes a.
\end{equation}
The same holds true for antilinear $\dagger$-functors.
\end{lem}
\begin{proof}
The unitary right adjoint $G$, if it exists, sends $b\in\cB$ to the object representing the $\dagger$-functor
\begin{equation}\label{eq: a mapsto B(F(a)-> b)}
a\mapsto \cB\big(F(a)\to b\big).
\end{equation}
By the unitary Yoneda Lemma, such a representing object, if it exists, is unique up to unique unitary isomorphism.
We get \eqref{eq: formula for G} by 
substituting \eqref{eq: a mapsto B(F(a)-> b)} into
\eqref{eq:CoendRepresentingObject}.
If $F$ admits a right adjoint, the functors \eqref{eq: a mapsto B(F(a)-> b)} are representable, hence unitarily representable, hence $F$ admits a unitary right adjoint.

Finally, antilinear $\dagger$-functors are the same thing as linear $\dagger$-functors $\cA\to\overline{\cB}$, so the same results hold true for antilinear $\dagger$-functors.
\end{proof}

\begin{lem}\label{lem: F-|G versus G-|F}
If $\coev$ and $\ev$ are the unit and counit of an (anti-)unitary adjunction $F\dashv^\dag G$, then $\ev^\dag$ and $\coev^\dag$ are the unit and counit of an (anti-)unitary adjunction $G\dashv^\dag F$.
\end{lem}
\begin{proof}
The adjunction is given by

\begin{equation}\label{eq:3antiunitaries}
\cA(G(b)\to a)
\underset{\text{anti-unitary}}{\cong}
\cA(a\to G(b))
\underset{\text{(anti-)unitary}}{\cong}
\cA(F(a)\to b)
\underset{\text{anti-unitary}}{\cong}
\cA(b\to F(a)).
\end{equation}
The first equivalence is anti-unitary as
$$
\langle f, g\rangle_{G(b)\to a} 
=
\Tr^\cA(g^\dag \circ f)
=
\Tr^\cA(f \circ g^\dag)
=
\langle g^\dag, f^\dag\rangle_{a\to G(b)},
$$
and similarly for the third one.

If we set $a=G(b)$, the image of $\id_{G(b)}$ under \eqref{eq:3antiunitaries} is $\ev^\dagger$. So $\ev^\dagger$ is the counit of the adjunction \eqref{eq:3antiunitaries}. Similarly, setting $b=F(a)$, we see that $\coev^\dagger$ is the unit of the adjunction.
\end{proof}

\begin{lem}\label{lem: Tr as a loop}
Let $(\cA,\Tr^\cA)$ be a semisimple $\rm C^*$ category with a unitary trace.
Let $a\in\cA$ be an object, viewed as linear $\dagger$-functor $a:\Hilb\to\cA$, and let $a^*:\cA\to \Hilb$ be its unitary adjoint. Then
\[
\Tr_a^{\cA}(f)\,=
\tikzmath{
\fill[gray!30] (.8,.3) -- (.8,1.3) arc (0:180:.4cm) -- (0,.3) arc (-180:0:.4cm) ; 
\draw (.8,.3) -- (.8,1.3) node[right, xshift=-.1cm, yshift=.2cm]{$\scriptstyle a$} arc (0:180:.4cm) node[left, yshift=2mm, xshift=1.5mm]{$\scriptstyle a^*$} -- (0,.3);
\draw (0,.3) arc (-180:0:.4cm) node[right, xshift=-.1cm, yshift=-.25cm]{$\scriptstyle a$} -- +(0,.7);
\roundNbox{fill=white}{(.8,.8)}{.3}{.1}{.1}{$f$}
\node[scale=.9] at (.45,-.24) {$\scriptstyle \ev_a^\dagger$};
\node[scale=.9] at (.4,1.85) {$\scriptstyle \ev_a$};
}
\qquad\qquad
\tikzmath{
\fill[gray!30, rounded corners = 5pt] (0,0) rectangle (.6,.6);
}
=
\cA
\qquad
\tikzmath{
\draw[rounded corners=5, thin, dotted] (0,0) rectangle (.6,.6);
}
=
\Hilb,
\]
where $\ev_a$ is the mate of $\id_a$ under the unitary adjunction
$
\Hilb(a^*(a)\to \bbC)
\cong
\cA(a\to a(\bbC))
$.
\end{lem}
\begin{proof}
\[
\tikzmath{
\fill[gray!30] (.8,.3) -- (.8,1.3) arc (0:180:.4cm) -- (0,.3) arc (-180:0:.4cm) ; 
\draw (.8,.3) -- (.8,1.3) node[right, xshift=-.1cm, yshift=.2cm]{$\scriptstyle a$} arc (0:180:.4cm) node[left, yshift=2mm, xshift=1.5mm]{$\scriptstyle a^*$} -- (0,.3);
\draw (0,.3) arc (-180:0:.4cm) node[right, xshift=-.1cm, yshift=-.25cm]{$\scriptstyle a$} -- +(0,.7);
\roundNbox{fill=white}{(.8,.8)}{.3}{.1}{.1}{$f$}
\node[scale=.9] at (.45,-.24) {$\scriptstyle \ev_a^\dagger$};
\node[scale=.9] at (.4,1.85) {$\scriptstyle \ev_a$};
}
=
\langle \mate(f),\mate(\id_a)\rangle_{\Hilb(a^*(a)\to \bbC)}
=
\langle f, \id_a\rangle_{\cA(a\to a)}
=
\Tr_a^\cA(f).
\qedhere
\]
\end{proof}

Anticipating the notion of unitary dual functor (see \S\ref{sec:UDF and SphericalStates} below), we have the following result:

\begin{prop}
\label{prop:UnitaryAdjointsGivesUDF}
Let $(\cA,\Tr^\cA)$ be a semisimple $\rm C^*$ category equipped with a unitary trace, and let $\End^\dag_d(\cA)$ be its category of dualizable linear dagger endofunctors. 
Then the operation which sends a dagger functor $F:\cA\to\cA$ to its unitary adjoint defines a unitary dual functor on $\End^\dag_d(\cA)$.

The same holds true for the category of linear or antilinear $\dagger$-functors from $\cA$ to itself.
\end{prop}
\begin{proof}
Given a dualizable $\dagger$-functor $F:\cA\to\cA$, we 
write $F^*$ for its unitary adjoint (which is unique up to unique unitary isomorphism).
We must show that 
the canonical isomorphism $F^* G^* \Rightarrow (G F)^*$ is unitary, and that for all $\theta:F\Rightarrow G$ we have $\theta^{*\dag}=\theta^{\dag*}: F^*\Rightarrow G^*$.
For the first statement, note that 
\[
\ev_F\circ\,(1_{F^*} \ev_G 1_{F})=
\tikzmath{
\fill[gray!30, rounded corners = 5pt] (0,0) rectangle (1,.6);
\draw[] (.2,0) -- +(0,.1) arc (180:0:.3cm) -- +(0,-.1);
\draw[] (.35,0) -- +(0,.09) arc (180:0:.15cm) -- +(0,-.09);
\useasboundingbox;
\node[scale=.9] at (.85,-.15) {$\scriptscriptstyle F$};
\node[scale=.9] at (.62,-.15) {$\scriptscriptstyle G$};
}
\qquad\text{and}\qquad
(1_{G} \coev_F 1_{G^*})\circ\coev_G=
\tikzmath{
\fill[gray!30, rounded corners = 5pt] (0,0) rectangle (1,.6);
\draw[] (.2,.6) -- +(0,-.1) arc (-180:0:.3cm) -- +(0,.1);
\draw[] (.35,.6) -- +(0,-.09) arc (-180:0:.15cm) -- +(0,.09);
\useasboundingbox;
\node[scale=.9] at (1-.85,.7) {$\scriptscriptstyle G$};
\node[scale=.9] at (1-.62,.7) {$\scriptscriptstyle F$};
}
\]
exhibit $F^* G^*$ as a unitary adjoint of $G F$, as
$$
\cA(a\to F^*(G^*(b)))
\underset{\text{(anti-)unitary}}{\cong}
\cA(F(a)\to G^*(b))
\underset{\text{(anti-)unitary}}{\cong}
\cA(G(F(a))\to b).
$$
By the uniqueness statement in Lemma~\ref{lem: unitary right adjoints}, the isomorphism 
$F^* G^* \Rightarrow (G F)^*$ is therefore unitary.

For the second statement,
we show that $\theta_a^{*\dag}=\theta_a^{\dag*}$ for all $a\in\cA$.
For all $f\in \cA(G^*(a)\to b)$ and $g\in \cA(F^*(a)\to b)$, we have:
\begin{align*}
\langle f\circ \theta^{\dag*}_a ,g\rangle
&=
\Tr^\cA_{F^*(a)}(g^\dag\circ f\circ \theta^{\dag*}_a )
\\&=
\Tr^\cA_{a}\left(\,\,
\tikzmath{
\begin{scope}
\clip[rounded corners=6pt] (-.5,-1.65) rectangle (1.6,2);
\fill[gray!30] (-.5,-1.65) -- (1.2,-1.65) -- (1.2,0) -- (1,0) -- (1,1) -- (1.2,1) -- (1.2,2) -- (-.5,2);
\fill[gray!30] (.8,.3) -- (1,.3) -- (1,.7) -- (.8,1.3) arc (0:180:.4cm) -- (0,-1) arc (-180:0:.4cm);
\end{scope}
\draw (.8,1.3) node[right, xshift=-.15cm, yshift=.24cm]{$\scriptstyle F^*$} arc (0:180:.4cm) node[left, yshift=2mm, xshift=1mm]{$\scriptstyle F$} -- (0,.3);
\draw (0,.3) -- (0,-1) arc (-180:0:.4cm) node[right, xshift=-.15cm, yshift=-.25cm]{$\scriptstyle F^*$} -- +(0,.7);
\draw (1.2,-1.65) --node[right, yshift=-4mm]{$\scriptstyle a$} (1.2,-.3);
\draw (1,.3) --node[right]{$\scriptstyle b$} (1,.7);
\draw (1.2,1.3) --node[right]{$\scriptstyle a$} (1.2,2);
\roundNbox{fill=white}{(1,1)}{.3}{.1}{.1}{$g^\dag$}
\roundNbox{fill=white}{(.75,-.8)}{.3}{0}{0}{$\theta^{\dag *}$}
\roundNbox{fill=white}{(1,0)}{.3}{.1}{.1}{$f$}
\node[scale=.8] at (.4,-1.54) {$\scriptstyle \coev$};
\node[scale=.8] at (.45,1.88) {$\scriptstyle \coev^\dagger$};
}
\right)
=   
\Tr^\cA_{a}\left(\,\,
\tikzmath{
\begin{scope}
\clip[rounded corners=6pt] (-.5,-1) rectangle (1.6,2);
\fill[gray!30] (-.5,-1) -- (1.2,-1) -- (1.2,0) -- (1,0) -- (1,1) -- (1.2,1) -- (1.2,2) -- (-.5,2);
\fill[gray!30] (1,.3) -- (1,.7) -- (.8,1.3) arc (0:180:.4cm) -- (0,-.3) arc (-180:0:.4cm);
\end{scope}
\draw (.8,1.3) node[right, xshift=-.15cm, yshift=.24cm]{$\scriptstyle F^*$} arc (0:180:.4cm) node[left, yshift=2mm, xshift=1mm]{$\scriptstyle F$} -- (0,.3);
\draw (0,.3) -- (0,-.3) node[left, yshift=-.2cm, xshift=1mm]{$\scriptstyle G$} arc (-180:0:.4cm) node[right, xshift=-.15cm, yshift=-.23cm]{$\scriptstyle G^*$};
\draw (1.2,-1) --node[right]{$\scriptstyle a$} (1.2,-.3);
\draw (1,.3) --node[right]{$\scriptstyle b$} (1,.7);
\draw (1.2,1.3) --node[right]{$\scriptstyle a$} (1.2,2);
\roundNbox{fill=white}{(1,1)}{.3}{.1}{.1}{$g^\dag$}
\roundNbox{fill=white}{(0,.5)}{.3}{0}{0}{$\theta^\dag$}
\roundNbox{fill=white}{(1,0)}{.3}{.1}{.1}{$f$}
\node[scale=.8] at (.4,-.84) {$\scriptstyle \coev$};
\node[scale=.8] at (.45,1.88) {$\scriptstyle \coev^\dagger$};
}
\right)
=   
\Tr^\cA_{a}\left(\,\,
\tikzmath{
\begin{scope}
\clip[rounded corners=6pt] (-.5,-1) rectangle (1.6,2.7);
\fill[gray!30] (-.5,-1) -- (1.2,-1) -- (1.2,0) -- (1,0) -- (1,1) -- (1.2,1) -- (1.2,2.7) -- (-.5,2.7);
\fill[gray!30] (1,.3) -- (1,.7) -- (.8,.8) -- (.8,1.3+.7) arc (0:180:.4cm) -- (0,-.3) arc (-180:0:.4cm);
\end{scope}
\draw (.8,.8) -- (.8,1.3+.7) node[right, xshift=-.15cm, yshift=.3cm]{$\scriptstyle G^*$} arc (0:180:.4cm) node[left, yshift=2.5mm, xshift=1mm]{$\scriptstyle G$} -- (0,.3);
\draw (0,.3) -- (0,-.3)  arc (-180:0:.4cm) node[right, xshift=-.15cm, yshift=-.23cm]{$\scriptstyle G^*$};
\draw (1.2,-1) --node[right]{$\scriptstyle a$} (1.2,-.3);
\draw (1,.3) --node[right]{$\scriptstyle b$} (1,.7);
\draw (1.2,1.3) --node[right, yshift=.2cm]{$\scriptstyle a$} (1.2,2.7);
\roundNbox{fill=white}{(1,1)}{.3}{.1}{.1}{$g^\dag$}
\roundNbox{fill=white}{(.75,1.8)}{.3}{0}{0}{$\theta^{*\dag}$}
\roundNbox{fill=white}{(1,0)}{.3}{.1}{.1}{$f$}
\node[scale=.8] at (.4,-.84) {$\scriptstyle \coev$};
\node[scale=.8] at (.45,2.58) {$\scriptstyle \coev^\dagger$};
}
\right)
\\&=
\Tr^\cA_{G^*(a)}(\theta_a^{*\dag}\circ g^\dag \circ f)
\\&=
\Tr^\cA_{F^*(a)}(g^\dag \circ f\circ \theta_a^{*\dag})
\\&=
\langle f\circ \theta_a^{*\dag}, g\rangle,
\end{align*}
where the second and fifth equalities hold by Lemmas~\ref{lem: F-|G versus G-|F} and~\ref{lem: Tr as a loop}.
By the non-degeneracy of the pairing, we conclude that $\theta_a^{\dag*}=\theta_a^{*\dag}$.
\end{proof}

\begin{rem}
\label{rem:UnitaryAdjointsGivesUDF}
More generally, unitary adjoints provide a canonical unitary dual functor on the $\rm C^*$ 2-category of semisimple $\rm C^*$ categories, dualizable dagger functors, and bounded natural transformations.
\end{rem}

\subsection{Unitary dual functors and spherical states}
\label{sec:UDF and SphericalStates}

Let $\cC$ be a unitary multitensor category (aka semisimple rigid tensor C* category).
We recall the following definition from \cite{MR4133163}:

\begin{defn}
A \emph{unitary dual functor} on $\cC$ is a choice of dual $(c^\vee, \ev_c,\coev_c)$ for each object $c\in \cC$ (where $\ev_c:c^\vee\otimes c\to 1$, and $\coev_c:1\to c \otimes c^\vee$ satisfy the zigzag identities), such that the corresponding functor $\vee: \cC\to \cC^{\rm mop}$ is a dagger tensor functor:
$f^{\vee \dag} = f^{\dag \vee}$, and $\nu_{a,b}: a^\vee \otimes b^\vee \to (b\otimes a)^\vee$ is unitary.
\end{defn}

\begin{rem}
Unlike dual functors, unitary dual functors are not unique -- see \cite{MR4133163}.
\end{rem}

By \cite[Lem.~7.5]{MR2767048}, \cite[Cor.~3.10]{MR4133163},
a unitary dual functor induces a pivotal structrue on $\cC$ by
\begin{equation}
\label{eq:CanonicalUnitaryPivotalStructure}
\varphi_c:=(1\otimes \ev_c)\circ(\ev_{c^\vee}^\dagger\otimes 1) = (\coev_c^\dagger\otimes 1)\circ(1\otimes \coev_{c^\vee}):c\to c^{\vee\vee},
\end{equation}
and $\varphi_c$ is unitary for all $c\in\cC$.

A unitary dual functor $\vee$ on a unitary multitensor category $\cC$ gives two $\End_\cC(1_\cC)$-valued traces $\tr_{L}$ and $\tr_{R}$ on the underlying $\rm C^*$ category.

\begin{defn}\label{def:spherical state}
A \emph{spherical state} on a unitary multitensor category with unitary dual functor $(\cC,\vee)$ is a state $\psi$ on $\End_\cC(1_\cC)$ such that $\psi\circ \tr_{L}(f) = \psi \circ \tr_{R}(f)$ for every $f\in\End_\cC(c)$.
Given a spherical state $\psi$, we write $\Tr^\cC:= \psi \circ \tr_L$.
\end{defn}


\begin{lem}
\label{lem:UDFsWithSphericalStates}
Suppose $\cC$ is an indecomposable unitary multitensor category.
\begin{enumerate}[label=(\arabic*)]
\item A unitary dual functor admits at most one spherical state.
\item For each state $\psi$ on $\End_\cC(1_\cC)$, there is a unique unitary dual functor with respect to which $\psi$ is spherical.
\end{enumerate}
We thus have a bijection
$$
\left\{
\parbox{5.5cm}{\rm \centerline{Unitary dual functors $\vee$ which} \centerline{admit a spherical state $\psi$}}
\right\}
\cong
\left\{
\parbox{2.4cm}{\rm \centerline{States $\psi$} \centerline{on $\End_\cC(1_\cC)$}}
\right\}.
$$
\end{lem}
\begin{proof}
Let $\cU$ be the universal grading groupoid of $\cC$.
By \cite[Thm.~A]{MR4133163}, unitary dual functors $\vee$ on $\cC$ are classified by 
groupoid homomorphisms $\pi: \cU\to \bbR_{>0}$,
where $\vee$ corresponds to $\pi$ if
\begin{equation}
\label{eq:BalanceByGroupoidHom}
\exists \lambda\in \bbC,
\,\,\,\,
\coev_c^\dag\circ (f \otimes \id_{c^{\vee_\pi}})\circ \coev_c
=\lambda p_i,
\qquad
\ev_c \circ (\id_{c^{\vee_\pi}}\otimes f) \circ \ev_c^\dag
=
\pi_{\operatorname{gr}(c)} \cdot
\lambda p_j
\end{equation}
for all homogeneous $c\in 1_i\otimes \cC\otimes 1_j$ (homogeneous w.r.t.~the $\cU$-grading) and $f: c\to c$.
Here, $p_1,\ldots,p_r\in\End_\cC(1_\cC)$ are the projections onto the simple summands of $1_\cC$, and $\operatorname{gr}_c\in\cU$ is the $\cU$-grading of $c$.

If $\vee$ admits a spherical state $\psi$, then $\pi$ factors through the `matrix groupoid' $\cM_r$ with $r$ objects and a unique isomorphism between any two objects.
Indeed, applying the spherical state $\psi$ to \eqref{eq:BalanceByGroupoidHom}, we have
$\lambda \psi(p_i) = \pi_{\operatorname{gr}(c)} \lambda \psi(p_j)$,
and thus 
\begin{equation}
\label{eq:FactorThroughMatrixGroupoid}
\pi_{\operatorname{gr}(c)} = \psi(p_i)/\psi(p_j) =: \pi_{ij}
\end{equation}

Since $\sum_i\psi(p_i)=1$, \eqref{eq:FactorThroughMatrixGroupoid} completely determines $\psi$ in terms of $\pi$. 
This proves part (1) of the lemma.
Equation \eqref{eq:FactorThroughMatrixGroupoid} also determines $\pi$ in terms of $\psi$, so there is at most one unitary dual functor $\vee$ for which a given state $\psi$ can be spherical.
It remains to verify that $\psi$ is spherical for the unitary dual functor associated to the homomrphism $\pi$ defined in~\eqref{eq:FactorThroughMatrixGroupoid}:
$$
\psi(\tr_L(f))
=
\lambda \pi_{ij}\psi(p_j)
=
\lambda \frac{\psi(p_i)}{\psi(p_j)}\psi(p_j)
=
\lambda \psi(p_i)
=
\psi(\tr_R(f)).
$$
This proves part (2) of the lemma.
\end{proof}

Observe that (2) of Lemma \ref{lem:UDFsWithSphericalStates} holds even when $\cC$ is decomposable.
Indeed, we can just restrict (and rescale) $\psi$ to each indecomposable summand and then apply the lemma.

\begin{defn}
A unitary multitensor category with unitary dual functor $(\cC,\vee)$ is called \emph{spherical} if it admits a spherical state.
\end{defn}

\begin{warn}
\label{warn:TwoNotionsOfSphericality}
There exists an alternative possible definition of sphericality
for a unitary multitensor category $\cC$, that we will not be using in this article:
the unitary dual functor corresponding to $\pi=1$ in \eqref{eq:BalanceByGroupoidHom}.
This is the \emph{balanced}, or \emph{minimal} unitary dual functor studied in \cite{MR3342166}.
\end{warn}

\begin{defn}\label{def.pivotalfunctor.chi}
Suppose $(\cC,\varphi^\cC)$ and $(\cD,\varphi^\cD)$ are pivotal categories.
A monoidal functor $F: \cC\to \cD$ is called \emph{pivotal} if the canonical isomorphism
$$
\chi_c:=
\tikzmath{
\draw (-.5,-1) -- node[left]{$\scriptstyle F(c^\vee)$} (-.5,-.3);
\draw (.5,-.3) node[left, yshift=-.3cm, xshift=.15cm]{$\scriptstyle F(v)$} arc (-180:0:.3cm) -- node[right]{$\scriptstyle F(c)^\vee$} (1.1,.5);
\roundNbox{fill=white}{(0,0)}{.3}{.5}{.5}{$F(\ev_c)$}
}
$$
(where we have suppressed the tensorator of $F$)
satisfies $(\chi_{c})^\vee \circ \varphi^\cD_{F(c)} = \chi_{c^\vee}\circ F(\varphi^\cC_c)$.
\end{defn}

By \cite[Prop.~3.40]{MR4133163}, when $\cC$ and $\cD$ are unitary multitensor categories equipped with unitary dual functors, pivotality of $F$ is equivalent to the canonical isomorphisms $\chi_c$
being unitary.

Recall that a unitary tensor category is a unitary multitensor category with simple unit.

\begin{lem}\label{lem: C-->D and D spherical}
Let
$(\cC,\vee_\cC)$ and $(\cD,\vee_\cD)$ be unitary multitensor categories equipped with unitary dual functors, with $1_\cC$ simple (i.e. $\cC$ is a tensor category), and $\cD$ non-zero.
If $\cD$ is spherical and $F: \cC\to \cD$ is a pivotal dagger tensor functor, then $\cC$ is spherical.
\end{lem}
\begin{proof}
Let $c\in\cC$ be any object.
By the pivotality of $F$ and \cite[Lem.~2.14]{MR4133163},
\begin{equation}
\label{eq:PivotalFunctorPeservesDimensions} 
F(\tr_L^\cC(\id_c))
=
\tr^\cD_L(\id_{F(c)})
\qquad\text{and}\qquad
F(\tr_R^\cC(\id_c))
=
\tr^\cD_R(\id_{F(c)}).
\end{equation}
Since $\End_\cC(1_\cC)$ is one-dimensional, $\tr_L^\cC(\id_c)$ and $\tr_R^\cC(\id_c)$ are scalars. So the two quantities in \eqref{eq:PivotalFunctorPeservesDimensions} are scalar multiples of $\id_{1_\cD}$.
Let $\psi$ be the spherical state of $\cD$.
The above two scalars are unchanged under applying $\psi$, so
$$
\tr^\cD_L(\id_{F(c)})
=
(\psi\circ \tr^\cD_L)(\id_{F(c)}) 
= 
(\psi\circ \tr^\cD_R)(\id_{F(c)}) 
=
\tr^\cD_R(\id_{F(c)}).
$$
It follows that $\tr_L^\cC(\id_c)=\tr_R^\cC(\id_c)$.
\end{proof}

We end this section by explaining how a unitary dual functor on a unitary multitensor category $\cC$ induces a unitary dual functor on its unitary Drinfeld center $Z^\dag(\cC)$ (the full subcategory of the ordinary Drinfeld center $Z(\cC)$ where all half-braidings are unitary):

\begin{lem}
Let $\vee=(\vee,\ev,\coev)$ be a unitary dual functor on $\cC$.
Then $\vee$ canonically induces a unitary dual functor 
on $Z^\dag(\cC)$.
\end{lem}
\begin{proof}
Given an object $(X,\sigma_X)\in Z^\dag(\cC)$, we define $(X,\sigma_X)^\vee:=(X^\vee,\sigma_{X^\vee})$ where  
$$
\sigma_{X^\vee,Y}:=
(\ev_X \otimes \id_Y \otimes \id_{X^\vee})
\circ
(\id_{X^\vee}\otimes \sigma_{X,Y}^{-1}\otimes \id_{X^\vee})\circ (\id_{X^\vee} \otimes \id_Y \otimes \coev_X).
$$
The half-braiding $\sigma_{X^\vee,Y}$ is unitary.
The evaluation and coevaluation are morphisms in $\cZ(\cC)$ hence in $\cZ^\dagger(\cC)$.
Finally, since the (co)evaluations for $(X,\sigma_X)$ are the same as those of the underlying object $X$, the dual functor $(X,\sigma_X)^\vee := (X^\vee, \sigma_{X^\vee})$ is unitary.
\end{proof}

\subsection{Involutive functors}
\label{sec:InvolutiveFunctors}

An \emph{involutive structure} on a category $\cA$ is an anti-linear functor $\overline{\,\cdot\,}: \cA\to \cA$ together with a coherence natural isomorphism $\varphi: \id_\cA \Rightarrow \overline{\overline{\,\cdot\,}}$ satisfying $\overline{\varphi_a} = \varphi_{\overline{a}}$ for all $a\in \cA$.
In the graphical calculus for $\Cat_\bbC$, we denote the anti-linear functor $\overline{\,\cdot\,}$ by a thick red strand, $\varphi$ by a cap, and $\varphi^{-1}$ by a cup.
$$
\tikzmath{
\fill[gray!30, rounded corners = 5pt] (0,0) rectangle (.6,.6);
\draw[thick, red] (.3,0) -- (.3,.6);
}
=
\overline{\,\cdot\,}: \cA\to \cA
\qquad\qquad
\tikzmath{
\fill[gray!30, rounded corners = 5pt] (0,0) rectangle (1,.6);
\draw[thick, red] (.2,.6) arc (-180:0:.3cm);
}
=
\varphi
\qquad\qquad
\tikzmath{
\fill[gray!30, rounded corners = 5pt] (0,0) rectangle (1,.6);
\draw[thick, red] (.2,0) arc (180:0:.3cm);
}
=
\varphi^{-1}
$$
The condition $\overline{\varphi_a} = \varphi_{\overline{a}}$ becomes
$$
\tikzmath{
\fill[gray!30, rounded corners = 5pt] (-.1,0) rectangle (1.4,.6);
\draw[thick, red] (.2,.6) arc (-180:0:.3cm);
\draw[thick, red] (1.1,.6) -- (1.1,0);
}
=
\tikzmath[xscale=-1]{
\fill[gray!30, rounded corners = 5pt] (-.1,0) rectangle (1.4,.6);
\draw[thick, red] (.2,.6) arc (-180:0:.3cm);
\draw[thick, red] (1.1,.6) -- (1.1,0);
}
\,,
$$
which is equivalent to $(\overline{\,\cdot\,},\varphi,\varphi^{-1})$ being an \emph{adjoint equivalence}, i.e., it is equivalent to the zig-zag axioms
$$
\tikzmath{
\fill[gray!30, rounded corners = 5pt] (-.9,-.6) rectangle (.9,.6);
\draw[thick, red] (-.6,.6) -- (-.6,0) arc (-180:0:.3cm) arc (180:0:.3cm) -- (.6,-.6);
}
=
\tikzmath{
\fill[gray!30, rounded corners = 5pt] (-.3,-.6) rectangle (.3,.6);
\draw[thick, red] (0,.6) -- (0,-.6);
}
=
\tikzmath[xscale=-1]{
\fill[gray!30, rounded corners = 5pt] (-.9,-.6) rectangle (.9,.6);
\draw[thick, red] (-.6,.6) -- (-.6,0) arc (-180:0:.3cm) arc (180:0:.3cm) -- (.6,-.6);
}\,.
$$

\begin{defn}\label{defn:RealStructure}
A \emph{conjugate-linear} morphism from $a$ to $b$, denoted $f: a\rightharpoonup b$, is a morphism $a\to \overline b$.

A \emph{real structure} on an object $a\in\cA$ is the data of a conjugate-linear morphism $r: a\rightharpoonup a$ such that $\overline r\circ r = \varphi_{a}$.
\end{defn}

\begin{defn}\label{def: invol funct}
Let $\cA$ and $\cB$ be involutive categories.
An involutive functor $(F,\chi)$ from $\cA$ to $\cB$ is a functor $F:\cA\to\cB$ equipped with a family of isomorphisms $\chi_a : F(\overline{a}) \to \overline{F(a)}$, for $a\in \cA$, satisfying
\begin{itemize}
\item
(involutive) 
$\overline{\chi_a} \circ \chi_{\overline{a}} \circ F(\varphi^\cA_a) = \varphi_{F(a)}^\cB$
\item
(conjugate natural)
$\overline{F(f)} \circ \chi_{a_1} = \chi_{a_2}\circ F(\overline{f})$ for all $f: a_1\to a_2$.
\end{itemize}
A natural transformation $\theta$ between two involutive functors $(F,\chi^F),(G,\chi^G): \cA \to \cB$ is called \emph{involutive} if $\chi^G_{a}\circ\theta_{\overline{a}} = \overline{\theta_a}\circ \chi^F_a : F(\overline{a})\to \overline{G(a)}$ for all $a\in \cA$.
\end{defn}

In the graphical calculus for $\Cat_\bbC$, we denote $\chi^F$ by a crossing:
$$
\tikzmath{
\begin{scope}
\clip[rounded corners=5pt] (-.7,-.4) rectangle (.7,.4);
\fill[gray!30] (-.4,-.4) .. controls ++(90:.4cm) and ++(270:.4cm) .. (.4,.4) -- (.7,.4) -- (.7,-.4);
\fill[gray!50] (-.4,-.4) .. controls ++(90:.4cm) and ++(270:.4cm) .. (.4,.4) -- (-.7,.4) -- (-.7,-.4);
\end{scope}
\draw[thick, red] (.4,-.4) node[below]{$\scriptstyle \overline{\,\cdot\,}$} .. controls ++(90:.4cm) and ++(270:.4cm) .. (-.4,.4) node[above]{$\scriptstyle \overline{\,\cdot\,}$};
\draw (-.4,-.4) node[below]{$\scriptstyle F$} .. controls ++(90:.4cm) and ++(270:.4cm) .. (.4,.4) node[above]{$\scriptstyle F$};
}
=\chi^F.
$$
The involutive and conjugate-natural axioms are denoted graphically by
$$
\tikzmath{
\begin{scope}
\clip[rounded corners=5pt] (-.7,-1) rectangle (1.3,.4);
\fill[gray!30] (-.4,-1) .. controls ++(90:.4cm) and ++(270:.4cm) .. (.8,.4) -- (1.3,.4) -- (1.3,-1);
\fill[gray!50] (-.4,-1) .. controls ++(90:.4cm) and ++(270:.4cm) .. (.8,.4) -- (-.7,.4) -- (-.7,-1);
\end{scope}
\draw[thick, red] (.4,-.4) .. controls ++(90:.4cm) and ++(270:.4cm) .. (-.4,.4) node[above]{$\scriptstyle \overline{\,\cdot\,}$};
\draw[thick, red] (.4,-.4) arc (-180:0:.3cm) (1,-.4) .. controls ++(90:.4cm) and ++(270:.4cm) .. (.2,.4) node[above]{$\scriptstyle \overline{\,\cdot\,}$};
\draw (-.4,-1) node[below]{$\scriptstyle F$} .. controls ++(90:.4cm) and ++(270:.4cm) .. (.8,.4) node[above]{$\scriptstyle F$};
}
=
\tikzmath{
\begin{scope}
\clip[rounded corners=5pt] (-.9,-1) rectangle (.8,.4);
\fill[gray!30] (.5,-1) rectangle (.8,.4);
\fill[gray!50] (.5,-1) rectangle (-.9,.4);
\end{scope}
\draw[thick, red] (-.6,.4) node[above]{$\scriptstyle \overline{\,\cdot\,}$} -- (-.6,0) arc (-180:0:.3cm) -- (0,.4) node[above]{$\scriptstyle \overline{\,\cdot\,}$};
\draw (.5,-1) node[below]{$\scriptstyle F$} -- (.5,.4) node[above]{$\scriptstyle F$};
}
\qquad\text{and}\qquad
\tikzmath{
\begin{scope}
\clip[rounded corners=5pt] (-.7,-.8) rectangle (1.3,.6);
\fill[gray!30] (-.4,-.8) .. controls ++(90:.4cm) and ++(270:.4cm) .. (.4,0) -- (.4,.6) -- (1,.6) -- (1,-.8);
\fill[gray!50] (-.4,-.8) .. controls ++(90:.4cm) and ++(270:.4cm) .. (.4,0) -- (.4,.6) -- (-.7,.6) -- (-.7,-.8);
\end{scope}
\draw[thick, red] (.4,-.8) node[below]{$\scriptstyle \overline{\,\cdot\,}$} .. controls ++(90:.4cm) and ++(270:.4cm) .. (-.4,0) -- (-.4,.6) node[above]{$\scriptstyle \overline{\,\cdot\,}$};
\draw (-.4,-.8) node[below]{$\scriptstyle F$} .. controls ++(90:.4cm) and ++(270:.4cm) .. (.4,0) -- (.4,.6) node[above]{$\scriptstyle F$};
\draw (1,-.8) node[below]{$\scriptstyle a_1$} -- (1,.6)node[above]{$\scriptstyle a_2$};
\roundNbox{fill=white}{(1,.1)}{.3}{0}{0}{$f$}
}
\,\,=
\tikzmath{
\begin{scope}
\clip[rounded corners=5pt] (-.7,-.6) rectangle (1.3,.8);
\fill[gray!30] (-.4,-.6) -- (-.4,0) .. controls ++(90:.4cm) and ++(270:.4cm) .. (.4,.8) -- (1,.8) -- (1,-.6);
\fill[gray!50] (-.4,-.6) -- (-.4,0) .. controls ++(90:.4cm) and ++(270:.4cm) .. (.4,.8) -- (-.7,.8) -- (-.7,-.6);
\end{scope}
\draw[thick, red] (.4,-.6) node[below]{$\scriptstyle \overline{\,\cdot\,}$} -- (.4,0) .. controls ++(90:.4cm) and ++(270:.4cm) .. (-.4,.8) node[above]{$\scriptstyle \overline{\,\cdot\,}$};
\draw (-.4,-.6) node[below]{$\scriptstyle F$} -- (-.4,0) .. controls ++(90:.4cm) and ++(270:.4cm) .. (.4,.8) node[above]{$\scriptstyle F$};
\draw (1,-.6) node[below]{$\scriptstyle a_1$} -- (1,.8)node[above]{$\scriptstyle a_2$};
\roundNbox{fill=white}{(1,-.1)}{.3}{0}{0}{$f$}
}\,.
$$
Composing the first identity with a red $\varphi$ cap on the top left and a red $\varphi^{-1}$ cup on the bottom right, we obtain the equivalent identity
$$
\tikzmath[xscale=-1, yscale=-1]{
\begin{scope}
\clip[rounded corners=5pt] (-.7,-1) rectangle (1.3,.4);
\fill[gray!50] (-.4,-1) .. controls ++(90:.4cm) and ++(270:.4cm) .. (.8,.4) -- (1.3,.4) -- (1.3,-1);
\fill[gray!30] (-.4,-1) .. controls ++(90:.4cm) and ++(270:.4cm) .. (.8,.4) -- (-.7,.4) -- (-.7,-1);
\end{scope}
\draw[thick, red] (.4,-.4) .. controls ++(90:.4cm) and ++(270:.4cm) .. (-.4,.4) node[below]{$\scriptstyle \overline{\,\cdot\,}$};
\draw[thick, red] (.4,-.4) arc (-180:0:.3cm) (1,-.4) .. controls ++(90:.4cm) and ++(270:.4cm) .. (.2,.4) node[below]{$\scriptstyle \overline{\,\cdot\,}$};
\draw (-.4,-1) node[above]{$\scriptstyle F$} .. controls ++(90:.4cm) and ++(270:.4cm) .. (.8,.4) node[below]{$\scriptstyle F$};
}
\,=\,
\tikzmath[xscale=-1, yscale=-1]{
\begin{scope}
\clip[rounded corners=5pt] (-.9,-1) rectangle (.8,.4);
\fill[gray!50] (.5,-1) rectangle (.8,.4);
\fill[gray!30] (.5,-1) rectangle (-.9,.4);
\end{scope}
\draw[thick, red] (-.6,.4) node[below]{$\scriptstyle \overline{\,\cdot\,}$} -- (-.6,0) arc (-180:0:.3cm) -- (0,.4) node[below]{$\scriptstyle \overline{\,\cdot\,}$};
\draw (.5,-1) node[above]{$\scriptstyle F$} -- (.5,.4) node[below]{$\scriptstyle F$};
}\,.
$$
The involutivity of a natural transformation $\theta$ between involutive functors is represented graphically by
$$
\tikzmath{
\begin{scope}
\clip[rounded corners=5pt] (-1,-1.4) rectangle (1,.4);
\fill[gray!30] (-.4,-1.4) -- (-.4,-.4) .. controls ++(90:.4cm) and ++(270:.4cm) .. (.4,.4) -- (1,.4) -- (1,-1.4);
\fill[gray!50] (-.4,-1.4) -- (-.4,-.4) .. controls ++(90:.4cm) and ++(270:.4cm) .. (.4,.4) -- (-1,.4) -- (-1,-1.4);
\end{scope}
\draw[thick, red] (.4,-1.4) node[below]{$\scriptstyle \overline{\,\cdot\,}$} -- (.4,-.4) .. controls ++(90:.4cm) and ++(270:.4cm) .. (-.4,.4) node[above]{$\scriptstyle \overline{\,\cdot\,}$};
\draw (-.4,-.4) node[left, yshift=.2cm]{$\scriptstyle G$} 
.. controls ++(90:.4cm) and ++(270:.4cm) .. (.4,.4) node[above]{$\scriptstyle G$};
\roundNbox{fill=white}{(-.4,-.7)}{.3}{0}{0}{$\theta$}
\draw (-.4,-1) -- (-.4,-1.4) node[below]{$\scriptstyle F$};
}
\,=\,
\tikzmath{
\begin{scope}
\clip[rounded corners=5pt] (-1,1.4) rectangle (1,-.4);
\fill[gray!30] (-.4,-.4) .. controls ++(90:.4cm) and ++(270:.4cm) .. (.4,.4) -- (.4,1.4) -- (1,1.4) -- (1,-.4);
\fill[gray!50] (-.4,-.4) .. controls ++(90:.4cm) and ++(270:.4cm) .. (.4,.4) -- (.4,1.4) -- (-1,1.4) -- (-1,-.4);
\end{scope}
\draw[thick, red] (.4,-.4) node[below]{$\scriptstyle \overline{\,\cdot\,}$} .. controls ++(90:.4cm) and ++(270:.4cm) .. (-.4,.4) -- (-.4,1.4) node[above]{$\scriptstyle \overline{\,\cdot\,}$};
\draw (-.4,-.4) node[below]{$\scriptstyle F$} 
.. controls ++(90:.4cm) and ++(270:.4cm) .. (.4,.4) node[right, yshift=-.2cm]{$\scriptstyle F$};
\roundNbox{fill=white}{(.4,.7)}{.3}{0}{0}{$\theta$}
\draw (.4,1) -- (.4,1.4) node[above]{$\scriptstyle G$};
}\,.
$$

The graphical proof of the following proposition is straightforward and left to the reader.

\begin{prop}
\label{prop:AdjointInvolutive}
Suppose $(F,\chi^F): \cA \to \cB$ is an involutive functor which admits a right adjoint $G: \cB \to \cA$.
Then
$$
\tikzmath{
\begin{scope}
\clip[rounded corners=5pt] (-.7,-.4) rectangle (.7,.4);
\fill[gray!50] (-.4,-.4) .. controls ++(90:.4cm) and ++(270:.4cm) .. (.4,.4) -- (.7,.4) -- (.7,-.4);
\fill[gray!30] (-.4,-.4) .. controls ++(90:.4cm) and ++(270:.4cm) .. (.4,.4) -- (-.7,.4) -- (-.7,-.4);
\end{scope}
\draw[thick, red] (.4,-.4) node[below]{$\scriptstyle \overline{\,\cdot\,}$} .. controls ++(90:.4cm) and ++(270:.4cm) .. (-.4,.4) node[above]{$\scriptstyle \overline{\,\cdot\,}$};
\draw (-.4,-.4) node[below]{$\scriptstyle G$} .. controls ++(90:.4cm) and ++(270:.4cm) .. (.4,.4) node[above]{$\scriptstyle G$};
}
:=
\chi^G :=
\tikzmath{
\begin{scope}
\clip[rounded corners=5pt] (-1.9,-1.2) rectangle (1.9,1.2);
\fill[gray!50] (-1,1.2) -- (-1,-.4) arc(-180:0:.3cm) .. controls ++(90:.4cm) and ++(270:.4cm) .. (.4,.4) arc (180:0:.3cm) -- (1,-1.2)  -- (1.9,-1.2) -- (1.9,1.2);
\fill[gray!30] (-1,1.2) -- (-1,-.4) arc(-180:0:.3cm) .. controls ++(90:.4cm) and ++(270:.4cm) .. (.4,.4) arc (180:0:.3cm) -- (1,-1.2) -- (-1.9,-1.2) -- (-1.9,1.2);
\end{scope}
\draw[thick, red] (-1.6,1.2) node[above]{$\scriptstyle \overline{\,\cdot\,}$} -- (-1.6,-.4) .. controls ++(270:.8cm) and ++(270:.8cm) .. (.4,-.4) .. controls ++(90:.4cm) and ++(270:.4cm) .. (-.4,.4) .. controls ++(90:.8cm) and ++(90:.8cm) .. (1.6,.4) -- (1.6,-1.2) node[below]{$\scriptstyle \overline{\,\cdot\,}$};
\draw (-1,1.2) node[above]{$\scriptstyle G$} -- (-1,-.4) arc(-180:0:.3cm) node[right]{$\scriptstyle F$} .. controls ++(90:.4cm) and ++(270:.4cm) .. (.4,.4) node[left]{$\scriptstyle F$} arc (180:0:.3cm) -- (1,-1.2) node[below]{$\scriptstyle G$};
}
\qquad\text{and}\qquad
\tikzmath[yscale=-1]{
\begin{scope}
\clip[rounded corners=5pt] (-.7,-.4) rectangle (.7,.4);
\fill[gray!50] (-.4,-.4) .. controls ++(90:.4cm) and ++(270:.4cm) .. (.4,.4) -- (.7,.4) -- (.7,-.4);
\fill[gray!30] (-.4,-.4) .. controls ++(90:.4cm) and ++(270:.4cm) .. (.4,.4) -- (-.7,.4) -- (-.7,-.4);
\end{scope}
\draw[thick, red] (.4,-.4) node[above]{$\scriptstyle \overline{\,\cdot\,}$} .. controls ++(90:.4cm) and ++(270:.4cm) .. (-.4,.4) node[below]{$\scriptstyle \overline{\,\cdot\,}$};
\draw (-.4,-.4) node[above]{$\scriptstyle G$} .. controls ++(90:.4cm) and ++(270:.4cm) .. (.4,.4) node[below]{$\scriptstyle G$};
}
:=
(\chi^G)^{-1}=
\tikzmath{
\begin{scope}
\clip[rounded corners=5pt] (-1.3,-1) rectangle (1.3,1);
\fill[gray!50] (-1,1) -- (-1,-.4) arc(-180:0:.3cm) .. controls ++(90:.4cm) and ++(270:.4cm) .. (.4,.4) arc (180:0:.3cm) -- (1,-1)  -- (1.3,-1) -- (1.3,1);
\fill[gray!30] (-1,1) -- (-1,-.4) arc(-180:0:.3cm) .. controls ++(90:.4cm) and ++(270:.4cm) .. (.4,.4) arc (180:0:.3cm) -- (1,-1) -- (-1.3,-1) -- (-1.3,1);
\end{scope}
\draw[thick, red] (.4,-1) node[below]{$\scriptstyle \overline{\,\cdot\,}$} -- (.4,-.4) .. controls ++(90:.4cm) and ++(270:.4cm) .. (-.4,.4) -- (-.4,1) node[above]{$\scriptstyle \overline{\,\cdot\,}$};
\draw (-1,1) node[above]{$\scriptstyle G$} -- (-1,-.4) arc(-180:0:.3cm) node[right]{$\scriptstyle F$} .. controls ++(90:.4cm) and ++(270:.4cm) .. (.4,.4) node[left]{$\scriptstyle F$} arc (180:0:.3cm) -- (1,-1) node[below]{$\scriptstyle G$};
}
$$
endow $G$ with the structure of an involutive functor.
Moreover the unit and counit of the adjunction $F\dashv G$ satisfy
$$
\tikzmath[xscale=-1]{
\begin{scope}
\clip[rounded corners=5pt] (-.7,-1) rectangle (1.3,.4);
\fill[gray!50] (-.4,.4) .. controls ++(270:.4cm) and ++(90:.4cm) .. (.4,-.4) arc (-180:0:.3cm) .. controls ++(90:.4cm) and ++(270:.4cm) .. (.2,.4) ;
\fill[gray!30] (-.4,.4) .. controls ++(270:.4cm) and ++(90:.4cm) .. (.4,-.4) arc (-180:0:.3cm) .. controls ++(90:.4cm) and ++(270:.4cm) .. (.2,.4) -- (1.3,.4) -- (1.3,-1) -- (-.7,-1) -- (-.7,.4);
\end{scope}
\draw (.4,-.4) .. controls ++(90:.4cm) and ++(270:.4cm) .. (-.4,.4) node[above]{$\scriptstyle F$};
\draw (.4,-.4) arc (-180:0:.3cm) (1,-.4) .. controls ++(90:.4cm) and ++(270:.4cm) .. (.2,.4) node[above]{$\scriptstyle G$};
\draw[thick, red] (-.4,-1) node[below]{$\scriptstyle \overline{\,\cdot\,}$} .. controls ++(90:.4cm) and ++(270:.4cm) .. (.8,.4) node[above]{$\scriptstyle \overline{\,\cdot\,}$};
}
=
\tikzmath[xscale=-1]{
\begin{scope}
\clip[rounded corners=5pt] (-.9,-1) rectangle (.8,.4);
\fill[gray!30] (-.6,.4) -- (-.6,0) arc (-180:0:.3cm) -- (0,.4) -- (.8,.4) -- (.8,-1) -- (-.9,-1) -- (-.9,.4);
\fill[gray!50] (-.6,.4) -- (-.6,0) arc (-180:0:.3cm) -- (0,.4);
\end{scope}
\draw (-.6,.4) node[above]{$\scriptstyle F$} -- (-.6,0) arc (-180:0:.3cm) -- (0,.4) node[above]{$\scriptstyle G$};
\draw[thick, red] (.5,-1) node[below]{$\scriptstyle \overline{\,\cdot\,}$} -- (.5,.4) node[above]{$\scriptstyle \overline{\,\cdot\,}$};
}
\qquad\text{and}\qquad
\tikzmath[yscale=-1]{
\begin{scope}
\clip[rounded corners=5pt] (-.7,-1) rectangle (1.3,.4);
\fill[gray!30] (-.4,.4) .. controls ++(270:.4cm) and ++(90:.4cm) .. (.4,-.4) arc (-180:0:.3cm) .. controls ++(90:.4cm) and ++(270:.4cm) .. (.2,.4) ;
\fill[gray!50] (-.4,.4) .. controls ++(270:.4cm) and ++(90:.4cm) .. (.4,-.4) arc (-180:0:.3cm) .. controls ++(90:.4cm) and ++(270:.4cm) .. (.2,.4) -- (1.3,.4) -- (1.3,-1) -- (-.7,-1) -- (-.7,.4);
\end{scope}
\draw (.4,-.4) .. controls ++(90:.4cm) and ++(270:.4cm) .. (-.4,.4) node[below]{$\scriptstyle F$};
\draw (.4,-.4) arc (-180:0:.3cm) (1,-.4) .. controls ++(90:.4cm) and ++(270:.4cm) .. (.2,.4) node[below]{$\scriptstyle G$};
\draw[thick, red] (-.4,-1) node[above]{$\scriptstyle \overline{\,\cdot\,}$} .. controls ++(90:.4cm) and ++(270:.4cm) .. (.8,.4) node[below]{$\scriptstyle \overline{\,\cdot\,}$};
}
=
\tikzmath[yscale=-1]{
\begin{scope}
\clip[rounded corners=5pt] (-.9,-1) rectangle (.8,.4);
\fill[gray!50] (-.6,.4) -- (-.6,0) arc (-180:0:.3cm) -- (0,.4) -- (.8,.4) -- (.8,-1) -- (-.9,-1) -- (-.9,.4);
\fill[gray!30] (-.6,.4) -- (-.6,0) arc (-180:0:.3cm) -- (0,.4);
\end{scope}
\draw (-.6,.4) node[below]{$\scriptstyle F$} -- (-.6,0) arc (-180:0:.3cm) -- (0,.4) node[below]{$\scriptstyle G$};
\draw[thick, red] (.5,-1) node[above]{$\scriptstyle \overline{\,\cdot\,}$} -- (.5,.4) node[below]{$\scriptstyle \overline{\,\cdot\,}$};
}\,.
$$
\end{prop}

\begin{lem}
\label{lem:MateOfBarF}
Let $F: \cA\to \cB$ be an involutive functor
between involutive categories,
and let $G:\cB\to \cA$ be its right adjoint
(with involutive structure $\chi^G$ as in Proposition \ref{prop:AdjointInvolutive}).
Then, for all $f\in \cB(F(a)\to b)$, we have
\[
\overline{\mate(f)}=\chi^G_b\circ\mate(\overline{f}\circ \chi^F_a) \in \cA\big(\overline{a}\to \overline{G(b)}\big).
\]
Similarly, for all $g\in \cA(a\to G(b))$, we have
\[
\overline{\mate(g)}=\mate\big([\chi_b^G]^{-1}\circ \overline g\big)\circ[\chi_a^F]^{-1}\in \cB\big(\overline{F(a)}\to \overline{b}\big).
\]
\end{lem}
\begin{proof}
We only prove the first equality, as the other one is similar.
In the graphical calculus, we have
$$
\overline{f}
=
\tikzmath{
\begin{scope}
\clip[rounded corners=5pt] (-1,-.7) rectangle (.6,.7);
\fill[gray!30] (-.2,-.7) rectangle (.2,-.3);
\fill[gray!50] (-.2,-.7) -- (-.2,-.3) -- (0,.3) -- (0,.7) -- (-1,.7) -- (-1,-.7);
\end{scope}
\draw[thick, red] (-.7,-.7) -- (-.7,.7);
\draw (0,.3) -- (0,.7) node[above]{$\scriptstyle b$};
\draw (.2,-.3) -- (.2,-.7) node[below]{$\scriptstyle a$};
\draw (-.2,-.3) -- (-.2,-.7) node[below]{$\scriptstyle F$};
\roundNbox{fill=white}{(0,0)}{.3}{.2}{.2}{$f$}
}\,,
\qquad
\overline{\mate(f)}
=
\tikzmath{
\begin{scope}
\clip[rounded corners=5pt] (-1.4,-.7) rectangle (.6,.7);
\fill[gray!30] (-.2,-.3) arc (0:-180:.3cm) -- (-.8,.7) -- (-1.4,.7) -- (-1.4,-.7) -- (.2,-.7) -- (.2,-.3);
\fill[gray!50] (-.2,-.3) arc (0:-180:.3cm) -- (-.8,.7) -- (0,.7)-- (0,0);
\end{scope}
\draw[thick, red] (-1.1,-.7) -- (-1.1,.7);
\draw (0,.3) -- (0,.7) node[above]{$\scriptstyle b$};
\draw (.2,-.3) -- (.2,-.7) node[below]{$\scriptstyle a$};
\draw (-.2,-.3) node[right, xshift=-.1cm, yshift=-.2cm]{$\scriptstyle F$} arc (0:-180:.3cm) -- (-.8,.7) node[above]{$\scriptstyle G$};
\roundNbox{fill=white}{(0,0)}{.3}{.2}{.2}{$f$}
}\,,
\qquad
\text{and}
\qquad
\chi^G_b\circ\mate(\overline{f}\circ \chi^F_a)
=
\tikzmath{
\begin{scope}
\clip[rounded corners=5pt] (-1.5,-1) rectangle (.6,.7);
\fill[gray!30] (-.2,-.3) .. controls ++(270:.2cm) and ++(90:.2cm) .. (-.8,-.7) arc (0:-180:.2cm) -- (-1.2,.3) .. controls ++(90:.2cm) and ++(270:.2cm) .. (-.8,.7) -- (-1.5,.7) -- (-1.5,-1) -- (.2,-1) -- (.2,-.3);
\fill[gray!50] (-.2,-.3) .. controls ++(270:.2cm) and ++(90:.2cm) .. (-.8,-.7) arc (0:-180:.2cm) -- (-1.2,.3) .. controls ++(90:.2cm) and ++(270:.2cm) .. (-.8,.7) -- (0,.7)-- (0,0);
\end{scope}
\draw[thick, red] (-.2,-1) -- (-.2,-.7) .. controls ++(90:.2cm) and ++(270:.2cm) .. (-.8,-.3) -- (-.8,.3) .. controls ++(90:.2cm) and ++(270:.2cm) .. (-1.2,.7);
\draw (0,.3) -- (0,.7) node[above]{$\scriptstyle b$};
\draw (.2,-.3) -- (.2,-1) node[below]{$\scriptstyle a$};
\draw (-.2,-.3) node[right, xshift=-.1cm, yshift=-.2cm]{$\scriptstyle F$} .. controls ++(270:.2cm) and ++(90:.2cm) .. (-.8,-.7) arc (0:-180:.2cm) -- (-1.2,.3) .. controls ++(90:.2cm) and ++(270:.2cm) .. (-.8,.7) node[above]{$\scriptstyle G$};
\roundNbox{fill=white}{(0,0)}{.3}{.2}{.2}{$f$}
}\,.
$$
The result is an instance of the final statement in Proposition \ref{prop:AdjointInvolutive}.
\end{proof}

A bi-involutive category is an involutive category which is also a dagger category such that $\overline{\,\cdot\,}$ is a dagger functor and $\varphi$ is unitary.
A bi-involutive functor $F: \cA\to \cB$ between bi-involutive categories is an involutive functor $(F,\chi)$ such that $F$ is a dagger functor and $\chi$ is unitary.

\begin{cor}
\label{cor:UnitaryAdjointBi-involutive}
Let $F: \cA\to \cB$ be a dualizable functor between semisimple bi-involutive categories with unitary traces.
Then its unitary adjoint $G:=F^*$ is also bi-involutive, maning $\chi^G$ (as defined in Proposition \ref{prop:AdjointInvolutive}) is unitary.
\end{cor}
\begin{proof}
Let $\cC:=\cA\oplus\cB$, and identify $F,G$ with functors
$$
\begin{pmatrix}
0 &0 
\\
F & 0 
\end{pmatrix}
\,,\,
\begin{pmatrix}
0 &G 
\\
0 & 0
\end{pmatrix}
\,:\,\cC\to\cC.
$$
$\overline{\,\cdot\,}:\cA\to\cA$ and $\overline{\,\cdot\,}:\cB\to\cB$ assemble to a functor
$$
\begin{pmatrix}
 \overline{\,\cdot\,} & 0
\\
0 &\overline{\,\cdot\,}
\end{pmatrix}
\,:\,\cC\to\cC
$$
that equips $\cC$ with the structure of a bi-involutive category.

Working in the category of dualizable linear or antiliner $\dagger$-functors from $\cC$ to itself (for which unitary adjoint defines a unitary dual functor by Proposition~\ref{prop:UnitaryAdjointsGivesUDF}),
$\chi^G$ is the $2\pi$-rotation of $\chi^F$.
The result follows from the fact that $\chi^F$ is unitary,
and that the $2\pi$-rotation of a unitary under a unitary dual functor is again unitary.
\end{proof}

\subsection{Involutive lax monoidal functors}

An involutive monoidal category \cite{MR2861112} is a monoidal category $\cC$ equipped with an involutive structure $(\overline{\,\cdot\,},\varphi)$, and coherence natural isomorphisms $\nu_{a,b}: \overline{a}\otimes \overline{b}\to \overline{b\otimes a}$ and real structure $r: 1\to \overline{1}$ (equivalently, $r:1\rightharpoonup 1$) which satisfy 
\begin{itemize}
\item (associativity) $\nu_{a,c\otimes b}\circ (\id_{\overline{a}} \otimes \nu_{b,c}) = \nu_{b\otimes a, c} \circ (\nu_{a,b}\otimes \id_{\overline{c}})$
\item (unitality) $\nu_{1,a}\circ (r\otimes \id_{\overline{a}}) = \id_{\overline{a}}=\nu_{a,1}\circ (\id_{\overline{a}}\otimes r)$
\item (compatibility with $\varphi$) 
$\varphi_{a\otimes b} = \overline{\nu_{b,a}}\circ \nu_{\overline{a},\overline{b}}\circ (\varphi_a\otimes \varphi_b)$.
\end{itemize}
By \cite[\S3.5]{MR4133163},
a unitary multitensor category $\cC$ has a canonical involutive structure $\overline{\,\cdot\,}$. This involutive structure has the notable feature that for any choice of unitary dual functor $\vee$, the conjugate
$\overline{x}$ of an object $x$ is canonically unitarily isomorphic to its unitary dual $x^\vee$.

\begin{rem}
Note that the tensor product of a conjugate-linear morphism with an ordinary morphism is meaningless. However, if $f:a\rightharpoonup b$ and $g:a'\rightharpoonup b'$ are two conjugate-linear morphisms, then we may define 
$
f\otimes g
:a\otimes a'\rightharpoonup b'\otimes b
$
as the composite
\[
a\otimes a'
\xrightarrow{f\otimes g} 
\overline{b}\otimes \overline{b'}
\xrightarrow{\nu_{b,b'}}
\overline{b'\otimes b}.
\]
\end{rem}

Recall that a monoidal functor $F:\cC\to \cD$ between monoidal categories involves coherences $\iota^F:1_\cD\to F(1_\cC)$,
and $\mu_{a,b}^F:F(a)\otimes F(b)\to F(a\otimes b)$.

\begin{defn}\label{def: involutive (lax) functor}
Let $\cC$ and $\cD$ be involutive monoidal categories.
An involutive (lax) monoidal functor $(F,\mu,\iota,\chi): \cC \to \cD$ is an involutive functor $(F,\chi):\cC \to \cD$ such that $(F,\mu,\iota): \cC \to \cD$ is  (lax) monoidal, and such that the following additional conditions hold:
\begin{itemize}
\item
(unitality)
$\chi_{1_\cC} \circ F(r_\cC) \circ \iota = \overline{\iota} \circ r_{\cD}$.
\item
(monoidality)
$
\chi_{d\otimes c} \circ F(\nu_{c,d}) \circ \mu_{\overline{c},\overline{d}} 
= 
\overline{\mu_{d,c}}\circ \nu_{F(c),F(d)} \circ (\chi_{c}\otimes \chi_d)
$
\end{itemize}
\end{defn}

Recall that the right adjoint $G: \cD\to \cC$ of a monoidal functor $(F,\mu^F,\iota^F): \cC\to \cD$ is lax monoidal by \cite{MR0360749}.
Indeed, $\mu^G_{d_1,d_2}:G(d_1)\otimes G(d_2) \to G(d_1\otimes d_2)$ is the mate of
$(\varepsilon_{d_1}\otimes \varepsilon_{d_2})\circ (\mu^F_{G(d_1),G(d_2)})^{-1}$ 
under the adjunction
$$
\cC\big(G(d_1)\otimes G(d_2) \to G(d_1\otimes d_2)\big)
\,\cong\,
\cD\big(F(G(d_1)\otimes G(d_2)) \to d_1\otimes d_2\big),
$$
and $\iota^G: 1_\cC \to G(1_\cD)$ is the mate of $(\iota^F)^{-1}$ under the adjunction
$$
\cC\big(1_\cC \to G(1_\cD)\big)
\,\cong\,
\cD\big(F(1_\cC) \to 1_\cD\big).
$$

\begin{prop}
\label{prop:AdjInvolutiveLaxMonoidal}
Let $(F,\mu^F,\iota^F,\chi^F) : \cC \to \cD$ be an involutive monoidal functor between involutive monoidal categories.
Then its right adjoint $(G,\mu,\iota): \cD \to \cC$ is involutive lax monoidal, with 
$\iota^G$ and $\mu^G$ as above, and
$\chi^G$ as in Proposition \ref{prop:AdjointInvolutive}.
\end{prop}
\begin{proof}
We need to check the two conditions in Definition~\ref{def: involutive (lax) functor}.
To prove unitality, we will argue that
$\iota^G \circ r_{\cC}^{-1} = G(r_\cD^{-1}) \circ \chi_{1_\cD}^{-1}\circ \overline{\iota^G} : \overline{1_\cC}\to G(1_\cD)$.
Taking mates under the adjunction$$
\cC\big(\overline{1_\cC} \to G(1_\cD)\big)
\,\cong\,
\cD\big(F(\overline{1_\cC}) \to 1_\cD\big),
$$
the resulting morphisms
fit in the following pasting diagram:
$$
\begin{tikzcd}[column sep=3em]
F(\overline{1_\cC})
\arrow[d, "F(\overline{\iota^G})"]
\arrow[rr, "F(r_\cC^{-1})"]
\arrow[drr, "\chi^F_{1_\cC}"]
&&
F(1_\cC) 
\arrow[r, "F(\iota^G)"]
&
FG(1_\cD)
\arrow[dd, "\varepsilon_{1_\cD}"]
\\
F(\overline{G(1_\cD)})
\arrow[d, "F((\chi_{1_\cD}^G)^{-1})"]
\arrow[r, "\chi^F_{G(1_\cD)}"]
&
\overline{FG(1_\cD)}
\arrow[dr, swap, "\overline{\varepsilon_{1_\cD}}"]
&
\overline{F(1_\cC)}
\arrow[l, "\overline{F(\iota^G)}"]
\\
FG(\overline{1_\cD})
\arrow[rr, "\varepsilon_{\overline{1_\cD}}"]
&&
\overline{1_\cD}
\arrow[u,swap,"\overline{\iota^F}"]
\arrow[r,"r_\cD^{-1}"]
&
1_\cD
\arrow[uul,swap, "\iota^F"]
\end{tikzcd}
$$
Going right and then down is the mate of 
$\iota^G \circ r_{\cC}^{-1}$,
and going down and then right is the mate of 
$G(r_\cD^{-1}) \circ \chi_{1_\cD}^{-1}\circ \overline{\iota^G} $.
The triangles commute by the definition of $\iota^G$.
The bottom left quadrilateral commutes using the dfinition of $(\chi^G)^{-1}$.
And the middle pentagon is the unitality condition for $F$ in Definition~\ref{def: involutive (lax) functor}.

To prove monoidality, we argue that 
$
G(\nu^\cD_{d_1,d_2}) \circ \mu_{\overline{d_1}, \overline{d_2}} \circ (\chi^{-1}_{d_1}\otimes \chi^{-1}_{d_2})
=
\chi_{d_1\otimes d_2}^{-1} \circ \overline{\mu_{d_2,d_1}} \circ \nu^\cC_{G(d_1),G(d_2)}:\overline{G(d_1)}\otimes \overline{G(d_2)} \to G(\overline{d_2\otimes d_1})
$
Taking mates under the adjunction
$$
\cC\big(\overline{G(d_1)}\otimes \overline{G(d_2)} \to G(\overline{d_2\otimes d_1})\big)
\,\cong\,
\cD\big(F(\overline{G(d_1)}\otimes \overline{G(d_2)}) \to \overline{d_2\otimes d_1}\big),
$$
the resulting maps fit in the following pasting diagram:
$$
\begin{tikzpicture}[baseline= (a).base]
\node[scale=.8] (a) at (0,0){
\begin{tikzcd}[column sep=5.25em]
F(\overline{G(d_1)} \otimes \overline{G(d_2)})
\arrow[ddd,"F(\nu_{G(d_1),G(d_2)}^\cC)"]
\arrow[rr,"F((\chi^G_{d_1})^{-1}\otimes (\chi^G_{d_2})^{-1})"]
&&
F(G(\overline{d_1})\otimes G(\overline{d_2}))
\arrow[r,"F(\mu^G_{\overline{d_1},\overline{d_2}})"]
&
FG(\overline{d_1} \otimes \overline{d_2})
\arrow[dd,"\varepsilon_{\overline{d_1}\otimes \overline{d_2}}"]
\\
&
F(\overline{G(d_1)}) \otimes F(\overline{G(d_2)})
\arrow[ul, swap, "\mu^F_{\overline{G(d_1)},\overline{G(d_2)}}"]
\arrow{r}[yshift=4]{F((\chi^G_{d_1})^{-1})\otimes F((\chi^G_{d_2})^{-1})}
\arrow[dr,swap,"\chi^F_{G(d_1)}\otimes \chi^F_{G(d_2)}"]
&
FG(\overline{d_1}) \otimes FG(\overline{d_2})
\arrow[u, swap, "\mu^F_{G(\overline{d_1}),G(\overline{d_2})}"]
\arrow[dr,"\varepsilon_{\overline{d_1}}\otimes \varepsilon_{\overline{d_2}}"]
\\
&&
\overline{FG(d_1)} \otimes \overline{FG(d_2)}
\arrow[d,"\nu^\cD_{FG(d_1), FG(d_2)}"]
\arrow[r,"\overline{\varepsilon_{d_1}}\otimes \overline{\varepsilon_{d_2}}"]
&
\overline{d_1}\otimes \overline{d_2}
\arrow[ddd,"\nu^\cD_{d_1,d_2}"]
\\
F(\overline{G(d_2)\otimes G(d_1)})
\arrow[r,"\chi^F_{G(d_2)\otimes G(d_1)}"]
\arrow[d,"F(\overline{\mu^G_{d_2,d_1}})"]
&
\overline{F(G(d_2)\otimes G(d_1))}
\arrow[d,"\overline{F(\mu^G_{d_2,d_1})}"]
&
\overline{FG(d_2)\otimes FG(d_1)}
\arrow[l,"\overline{\mu^F_{G(d_2),G(d_1)}}"']
\arrow[ddr, "\overline{\varepsilon_{d_2}\otimes\varepsilon_{d_1}}"]
\\
F(\overline{G(d_2\otimes d_1)})
\arrow[r,"\chi^F_{G(d_2\otimes d_1)}"]
\arrow[d,"F((\chi^G_{d_1\otimes d_2})^{-1})"]
&
\overline{FG(d_2\otimes d_1)}
\arrow[drr,"\overline{\varepsilon_{d_2\otimes d_1}}"]
\\
F(G(\overline{d_1\otimes d_2}))
\arrow[rrr,"\varepsilon_{\overline{d_2\otimes d_1}}"]
&&&
\overline{d_2\otimes d_1}
\end{tikzcd}
};\end{tikzpicture}
$$
Going right and then down is the mate of 
$G(\nu^\cD_{d_1,d_2}) \circ \mu_{\overline{d_1}, \overline{d_2}} \circ (\chi^{-1}_{d_1}\otimes \chi^{-1}_{d_2})$,
while going down and then right is the mate of 
$\chi_{d_1\otimes d_2}^{-1} \circ \overline{\mu_{d_2,d_1}} \circ \nu^\cC_{G(d_1),G(d_2)}$.
The hexagon in the top left is the monoidality condition for $F$ in Definition~\ref{def: involutive (lax) functor}.
\end{proof}


\section{Unitary module tensor categories}
\label{sec:UMTC}

\subsection{Module tensor categories}
\label{sec:ModTens}

Let $\cV$ be a braided tensor category.
A module tensor category $\cC$ over $\cV$ is a tensor category $\cC$ equipped with an ``action'' of $\cV$ in the following sense:

\begin{defn}
\label{def: def of ModTC}
A \emph{module tensor category} over $\cV$ is a tensor category $\cC$ together with a braided tensor functor $\Phi^{\scriptscriptstyle \cZ}:\cV\to \cZ(\cC)$ from $\cV$ to the Drinfeld center of $\cC$. 
If $\cV$ and $\cC$ are pivotal and if $\Phi^{\scriptscriptstyle \cZ}$ is a pivotal functor, then we call $\cC$ a pivotal module tensor category.
\end{defn}

Let $\cC$ be a pivotal module tensor category over $\cV$, and let us write $\varphi_a: a\to a^{\vee\vee}$ for the pivotal structure of $\cC$.
A \emph{pointing} of $\cC$ is a choice of object $c\in \cC$ such that $c$ and $\Phi^{\scriptscriptstyle \cZ}(\cV)$ generate $\cC$ under the operations of tensor product, direct sum, and taking direct summands.
We also require that our chosen object $c$ comes equipped with a \emph{symmetric self-duality} isomorphism $r_c: c\to c^\vee$ satisfying $\ev_c\circ (r_c\otimes \id_c) = \ev_{c^\vee}\circ (\varphi_c\otimes r_c)$.\footnote{As explained in \cite[Section 3.2.1]{MR4528312}, this last requirement is not crucial.
We could deal with objects that are not self-dual at the cost of equipping all strands with co-orientations (and see \cite[Section 4.1]{MR4581741} for why co-orientations are preferable to orientations).}

\begin{construction}\label{constr: 3.2}
Let $(\cC,\Phi^{\scriptscriptstyle \cZ})$ be a pivotal module tensor over $\cV$.
Suppose that the functor $\Phi:=\Forget_\cZ \circ \Phi^{\scriptscriptstyle \cZ}: \cV\to \cC$ admits a right adjoint, denoted $\bTr_\cV$ (where $\Forget_\cZ: \cZ(\cC)\to \cC$ is the forgetful functor).
In \cite{MR3578212}, we showed that
$$
\bTr_\cV:\cC\to \cV
$$
has a canonical structure of a \emph{categorified trace}. It comes equipped with the following structure.
\begin{itemize}
\item 
The 
\emph{unit} $\eta_v: v\to \bTr_\cV(\Phi(v))$
and 
\emph{counit} $\varepsilon_x:\Phi(\bTr_\cV(x))\to x$ 
of the adjunction.
$$
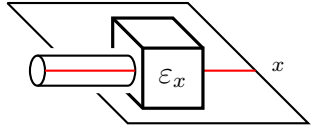
\tikzmath{
	\plane{(-.2,-.8)}{2.4}{1.6}
	\draw[thick, red] (.8,-.1) -- (1.5,-.1);
	\CMbox{box}{(0,-.6)}{.8}{.8}{.4}{$\varepsilon_x$}
	\straightTubeWithString{(-.2,-.3)}{.1}{1.2}{thick, red}
	\node at (1.8,-.05) {$\scriptstyle x$};
}
$$
\item 
A \emph{multiplication map} 
$
\begin{tikzpicture}[baseline=-.3cm]
	\topPairOfPants{(-1,-1)}{}
	\draw[thick, red] (-.7,-1.1) .. controls ++(90:.8cm) and ++(270:.8cm) .. (-.1,.42);		
	\draw[thick, blue] (.7,-1.1) .. controls ++(90:.8cm) and ++(270:.8cm) .. (.1,.42);		
	\node at (-.7,-1.28) {$\scriptstyle x$};
	\node at (.7,-1.3) {$\scriptstyle y$};
\end{tikzpicture}
=\mu_{x,y}:\bTr_\cV(x)\otimes \bTr_\cV(y)\to \bTr_\cV(x\otimes y)$.
\item 
A \emph{unit map} 
$
\tikzmath{
	\coordinate (a1) at (0,0);
	\coordinate (b1) at (0,.4);
	\draw[thick] (a1) -- (b1);
	\draw[thick] ($ (a1) + (.6,0) $) -- ($ (b1) + (.6,0) $);
	\draw[thick] ($ (b1) + (.3,0) $) ellipse (.3 and .1);
	\draw[thick] (a1) arc (-180:0:.3cm);
}
=i:1_\cV\to\bTr_\cV(1_\cC)$.
\item 
A \emph{traciator} natural isomorphism 
$
\begin{tikzpicture}[baseline=-.1cm]
	\draw[thick] (-.3,-1) -- (-.3,1);
	\draw[thick] (.3,-1) -- (.3,1);
	\draw[thick] (0,1) ellipse (.3cm and .1cm);
	\halfDottedEllipse{(-.3,-1)}{.3}{.1}
	
	\draw[thick, red] (-.1,-1.1) .. controls ++(90:.8cm) and ++(270:.8cm) .. (.1,.9);		
	\draw[thick, blue] (.1,-1.1) .. controls ++(90:.2cm) and ++(225:.2cm) .. (.3,-.2);		
	\draw[thick, blue] (-.1,.9) .. controls ++(270:.2cm) and ++(45:.2cm) .. (-.3,.2);
	\draw[thick, blue, dotted] (-.3,.2) -- (.3,-.2);	
	\node at (-.13,-1.28) {$\scriptstyle x$};
	\node at (.13,-1.3) {$\scriptstyle y$};
\end{tikzpicture}
=\tau_{x,y}:\bTr_\cV(x\otimes y)\to\bTr_\cV(y\otimes x)$.
\end{itemize}
This structure satisfies various properties listed in \cite[\S4]{MR3578212}, many of which are summarized in \cite[\S5.1]{MR4528312}.
\end{construction}

\begin{defn}
A 1-morphism $(\cC_1,\Phi_1^{\scriptscriptstyle \cZ},\varphi^1,x_1) \to (\cC_2,\Phi_2^{\scriptscriptstyle \cZ},\varphi^2,x_2)$ 
of pivotal pointed module tensor categories is a pair $(G,\gamma)$ consisting of:
\begin{itemize}
\item
a pivotal tensor functor $G: \cC_1\to \cC_2$ such that $G(x_1)=x_2$ and $r_2 = \chi_{x_1}\circ G(r_1)$, where $r_i: x_i \to x_i^\vee$ is the symmetric self-duality, and $\chi$ is as in Definition~\ref{def.pivotalfunctor.chi}.
\item
an \emph{action coherence} monoidal natural isomorphism $\gamma: \Phi_2\Rightarrow G\circ \Phi_1$ satisfying the following compatibility with the half-braidings:
\[
\begin{tikzcd}
\Phi_2(v)\otimes G(c)
\arrow[r,"\gamma_v\otimes \id"]
\arrow[d,"e_{\Phi_2(v),G(c)}"]
&
G(\Phi_1(v))\otimes G(c)
\arrow[r,"\cong"]
&
G(\Phi_1(v)\otimes c)
\arrow[d,"G(e_{\Phi_1(v),c})"]
\\
G(c) \otimes \Phi_2(v)
\arrow[r,"\id\otimes\gamma_v"]
&
G(c)\otimes G(\Phi_1(v))
\arrow[r,"\cong"]
&
G(c\otimes \Phi_1(v))
\end{tikzcd}
\]
\end{itemize}
\end{defn}

Given two 1-morphisms $(G,\gamma^G),(H,\gamma^H)$ between pointed module tensor categories, by \cite[Lem.~3.6]{MR4528312}, there is at most one monoidal natural transformation $\kappa: (G,\gamma^G)\Rightarrow (H,\gamma^H)$ satisfying the following compatibility with the action coherence morphisms:
\[
\begin{tikzcd}
\Phi_2(v) 
\arrow[dr, "\gamma^G_v"]
\arrow[rr, "\gamma^H_v"]
&&
H(\Phi_1(v))
\\
&
G(\Phi_1(v))
\arrow[ur,"\kappa_{\Phi_1(v)}"]
\end{tikzcd}
\]
When such a $\kappa$ exists, it is necessarily invertible.
Hence the 2-category of pointed pivotal module tensor categories over $\cV$ is 1-truncated, i.e., equivalent to a 1-category.

\subsection{Unitary module tensor categories}
\label{sec:Unitary module tensor categories}

For the remainder of this section, we assume that $(\cV,\vee_\cV)$ is a braided unitary tensor category with a fixed unitary dual functor,
and let $\tr^\cV$ be the associated right pivotal trace. 

\begin{defn}
Let $\cV$ be as above.
A \emph{unitary module (multi)tensor category} over $\cV$ is a unitary (multi)tensor category $\cC$ with a unitary dual functor $\vee_\cC$, a faithful state $\psi_\cC$ on $\End_\cC(1_\cC)$, and
a  braided pivotal dagger tensor functor $\Phi^{\scriptscriptstyle \cZ} : \cV \to \cZ^\dag(\cC)$.
\end{defn}

We call a unitary module multitensor category \emph{spherical} if $\psi_\cC$ is a spherical state (Definition~\ref{def:spherical state}).
See Warning \ref{warn:TwoNotionsOfSphericality}.
Observe that, by Lemma~\ref{lem: C-->D and D spherical},
the existence of a non-zero spherical unitary $\cV$-module tensor category implies that $\cV$ is ribbon, i.e., that the unitary dual functor on $\cV$ is spherical.

The choices of $\vee_\cC$ and $\psi_\cC$ endow $\cC$ with a unitary trace
\begin{equation}\label{eq: heart}
\Tr^\cC:=\psi_\cC \circ \tr^\vee_R.
\end{equation}
Moreover, $(\cC, \Tr^\cC)$ is a pivotal module $\rm C^*$ category (aka unitary module category) \cite[\S3.6.2]{MR4598730} for the underlying unitary tensor category of $\cV$ (without the braiding). 
So unitary module tensor categories are, in particular, unitary module categories.

\begin{rem}
Suppose $(\cC,\vee_\cC,\psi_\cC,\Phi^{\scriptscriptstyle \cZ})$ is a unitary module multitensor category over $\cV$.
If $\Phi$ admits a right adjoint, then by Lemma \ref{lem: unitary right adjoints}, $\Phi$ admits a right unitary adjoint
$\bTr_\cV$ which is automatically bi-involutive lax monoidal by 
Proposition \ref{prop:AdjInvolutiveLaxMonoidal}
and
Corollary \ref{cor:UnitaryAdjointBi-involutive}.
Thus $\bTr_\cV$ comes equipped with canonical unitary isomorphisms $\{\chi^{\bTr_\cV}_c:\bTr_\cV(\overline{c})\to \overline{\bTr_\cV(c)}\}_{c\in\cC}$ satisfying the axioms in Definitions \ref{def: invol funct} and \ref{def: involutive (lax) functor}.
\end{rem}

We warn the reader that $\Tr^\cC$ is a trace in the sense of Definition~\ref{def: 2.1}, whereas $\bTr_\cV$ is a categorified trace as described in Construction~\ref{constr: 3.2}.

As before, a \emph{pointing} of $\cC$ is a choice of object $c\in \cC$ such that $c$ and $\Phi^{\scriptscriptstyle \cZ}(\cV)$ generate $\cC$ under the operations of tensor product, orthogonal direct sum, and taking orthogonal direct summands.
We also require that our symmetric self-duality $r_c: c\to c^\vee$ is unitary; this is also called a \emph{real structure} for $c$.

\begin{defn}
A 1-morphism $(G,\gamma): (\cC_1,\vee_1,\psi_1,\Phi_1^{\scriptscriptstyle \cZ},x_1) \to (\cC_2,\vee_2,\psi_2,\Phi_2^{\scriptscriptstyle \cZ},x_2)$ 
of pointed unitary module tensor categories over $\cV$ is a 1-morphism of underlying pointed pivotal module tensor categories such that
\begin{itemize}
\item 
$G$ is a $\dag$-tensor functor,
\item
the action coherence monoidal natural isomorphism $\gamma: \Phi_2 \Rightarrow G\circ \Phi_1$ is unitary and involutive:
$\forall v\in \cV,\,
\overline{\gamma_v} \circ \chi^{\Phi_2}_v = \chi^G_{\Phi_1(v)}\circ G(\chi^{\Phi_1}_v)\circ \gamma_{\overline{v}}:
\Phi_2(\overline{v}) \to \overline{G(\Phi_1(v))}$,
\item
$\psi_2 \circ G = \psi_1$ on $\End_{\cC_1}(1_{\cC_1})$.
\end{itemize}
\end{defn}

\subsection{The map \texorpdfstring{$i^\dag$}{idag} and a formula for \texorpdfstring{$\coev^\dag_{\bTr_\cV(c)}$}{coev dag of Tr(c)}}
\label{sec:idag}

In \S\ref{subsec:UAPA} below, we will define unitarity for an anchored planar algebra in terms of a certain pairing on $\cP[n]$ from the anchored planar algebra being equal to $\coev^\dag_{\cP[n]}$.

Fix a unitary $\cV$-module multitensor category $(\cC,\vee_\cC,\psi_\cC,\Phi^{\scriptscriptstyle \cZ})$ such that $\Phi: \cV \to \cC$ admits a right unitary adjoint $\bTr_\cV$.
In this section, we will study $i^\dag:\bTr_\cV(1_\cC)\to1_\cV$, and prove the important formula \eqref{eq:FormulaForEvTr} for $\coev^\dag_{\bTr_\cV(c)}$.
We represent $i^\dag$ diagrammatically by
$$
i^\dag =
\tikzmath{
\draw[thick] (-.5,-.5) -- (-.5,0) arc (180:0:.5cm) -- (.5,-.5);
\draw[thick] (-.5,-.5) arc (-180:0:.5 and .2);
\draw[thick, densely dotted] (-.5,-.5) arc (180:0:.5 and .2);
}\,
$$
In Lemma \ref{lem:idagAllowsIsotopy}, we will show that if $\psi_\cC$ is spherical, then all isotopies are allowed for strings on the capped tube:
$$
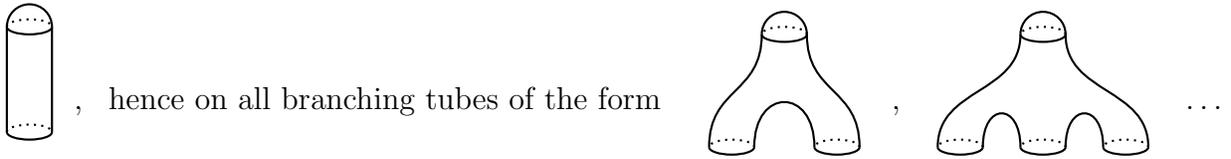
\begin{tikzpicture}[baseline=.9cm, xscale=-1]
	\coordinate (a1) at (0,.6);
	\coordinate (b1) at (0,2);
	\coordinate (c) at (.15,.2);
	\draw[thick] (a1) -- (b1);
	\draw[thick] ($ (a1) + (.6,0) $) -- ($ (b1) + (.6,0) $);
	\halfDottedEllipse{(0,.6)}{.3}{.1}
	\halfDottedEllipse{(0,2)}{.3}{.1}
	\draw[thick] (b1) arc (180:0:.3cm);
\end{tikzpicture}
\,\,\,\,
\text{,\,\,\, hence on all branching tubes of the form}
\quad\,\,\,
\begin{tikzpicture}[baseline=1.5cm]
	\coordinate (a1) at (0,0);
	\coordinate (a2) at (1.4,0);
	\coordinate (b1) at (0,1);
	\coordinate (b2) at (1.4,1);
	\coordinate (c1) at (.7,2.5);
	\draw[thick] (c1) arc (180:0:.3cm);
	\pairOfPants{(b1)}
\end{tikzpicture}
\quad
,
\quad
\begin{tikzpicture}[baseline=1.5cm]
	\coordinate (a1) at (0,0);
	\coordinate (a2) at (1.4,0);
	\coordinate (b1) at (0,1);
	\coordinate (b2) at (1.4,1);
	\coordinate (c1) at (.7,2.5);
	\draw[thick] (c1) arc (180:0:.3cm);
	\draw[thick] ($ (b1) + (-.4,0) $) .. controls ++(90:.8cm) and ++(270:.8cm) .. ($ (b1) + (.7,1.5) $);
	\draw[thick] ($ (b1) + (2.4,0) $) .. controls ++(90:.8cm) and ++(270:.8cm) .. ($ (b1) + (2,0) + (-.7,1.5) $);
	\draw[thick] ($ (b1) + (.2,0) $).. controls ++(90:.6cm) and ++(90:.6cm) .. ($ (b1) + (.7,0) $); 
	\draw[thick] ($ (b1) + (1.3,0) $).. controls ++(90:.6cm) and ++(90:.6cm) .. ($ (b1) + (1.8,0) $); 
	\halfDottedEllipse{($ (b1) + (.7,1.5) $)}{.3}{.1}
	\halfDottedEllipse{($ (b1) + (-.4,0) $)}{.3}{.1}
	\halfDottedEllipse{($ (b1) + (.7,0) $)}{.3}{.1}
	\halfDottedEllipse{($ (b1) + (1.8,0) $)}{.3}{.1}
\end{tikzpicture}
\quad
\ldots
$$

Consider the finite dimensional abelian $\rm C^*$-algebra $\cC(1_\cC\to 1_\cC)$. 
By unitary adjunction, we have an isomorphism
\begin{align}
\label{eq:IdentifyGroundC*Alg}
\cC(1_\cC\to 1_\cC)
\,&\cong\,\, 
\cV(1_\cV\to \bTr_\cV(1_\cC))
\\
\notag
f\quad 
&\!\mapsto\,\,\,
\tikzmath{
\draw[thick] (-.5,.7) -- (-.5,0) arc (-180:0:.5cm) -- (.5,.7);
\draw[thick] (0,.7) ellipse (.5 and .2);
\roundNbox{}{(0,0)}{.3}{0}{0}{$f$}
}
=\mate(f)=
\bTr_\cV(f)\circ i
\end{align}
hence the right hand side is also equipped with the structure of an abelian $\rm C^*$-algebra.
The multiplication and $*$-structure on the right hand side of \eqref{eq:IdentifyGroundC*Alg} are given by
\[
x \cdot y = \mu_{1_\cC,1_\cC}\circ(x\otimes y)
\qquad\quad
\text{and}
\qquad\quad
x^* = \big(\chi^{\bTr_\cV}_{1_\cC}\circ \bTr_\cV(r_\cC)\big)^{-1}\circ \overline{x}\circ r_\cV.
\]
To see that the isomorphism \eqref{eq:IdentifyGroundC*Alg} intertwines the two $*$-structures, i.e.,
\begin{align*}
&\bTr_{\cV}(f^\dag) \circ i 
=
(\chi_{1_{\cC}}^{\bTr_{\cV}} \circ \bTr_{\cV}(r_{\cC}))^{-1}\circ  \overline{\bTr_{\cV}(f) \circ i }  \circ r_{\cV},
&&
\text{equivalently}
\\
&\chi_{1_{\cC}}^{\bTr_{\cV}} \circ \bTr_{\cV}(r_{\cC})  \circ  \bTr_\cV(f^\dag) \circ i
=   
\overline{ \bTr_\cV(f) \circ i }   \circ r_\cV,
\end{align*}
we check that, since $f^\dag=\overline{f}^{\vee}=(r_\cC)^{-1}\circ \overline{f}\circ r_\cC$ on $\End_\cC(1_\cC)$ (as $r_\cC=\coev_1$), we have
\begin{align*}
\chi_{1_{\cC}}^{\bTr_{\cV}} \circ \bTr_{\cV}(r_{\cC})  \circ \bTr_{\cV}(f^\dagger) \circ i
&=
\chi_{1_{\cC}}^{\bTr_{\cV}} \circ \bTr_{\cV}(r_{\cC})  \circ \bTr_\cV(r_{\cC}^{-1}) \circ \bTr_\cV(\overline{f}) \circ \bTr_\cV(r_\cC)\circ i
\\&=
\chi_{1_{\cC}}^{\bTr_{\cV}} \circ \bTr_\cV(\overline{f}) \circ \bTr_\cV(r_\cC)\circ i
\\&=
\overline{\bTr_\cV(f)} \circ \chi_{1_{\cC}}^{\bTr_{\cV}} \circ \bTr_\cV(r_\cC)\circ i
\\&=
\overline{\bTr_\cV(f)} \circ \bar{i} \circ r_\cV.
\end{align*}
The third equality follows from involutivity of $\bTr_\cV$ (Definition~\ref{def: invol funct}),
and the final equality uses the unitality axiom 
$\chi_{1_\cC} \circ  \bTr(r_\cC) \circ i = \bar{i} \circ r_\cV$
from Definition~\ref{def: involutive (lax) functor},
which holds by Proposition \ref{prop:AdjInvolutiveLaxMonoidal}.

\begin{lem}
\label{lem:IdentifyStatesOnGroundC*Alg}
The state $\cV(1_\cV\to \bTr_\cV(1_\cC)) \to \End_\cV(1_\cV)\cong \bbC$ given by $x\mapsto i^\dag\circ x$
corresponds to $\psi_\cC:\End_\cC(1_\cC)\to\bbC$ under the isomorphism \eqref{eq:IdentifyGroundC*Alg}.
\end{lem}
\begin{proof}
Fix $x\in \cV(1_\cV\to \bTr_\cV(1_\cC))$.
We then have
\[
\psi_\cC(\mate(x)) = \langle \mate(x),\id_{1_\cC}\rangle_{1_\cC\to 1_\cC} = \langle x, i\rangle_{1_\cV\to \bTr_\cV(1_\cC)} = i^\dag \circ x,
\]
where the first equality holds by combining \eqref{eq: heart} with the definition of $\langle\cdot\,,\cdot\rangle$ (Definition \ref{def: 2.1}), and the second equality is the
unitarity of the adjunction \eqref{eq:IdentifyGroundC*Alg}, using that $\mate(i)=\id_{1_\cC}$.
\end{proof}

We will use the following lemma to get our equality of pairings in Proposition \ref{prop:FormulaForEvTr} below.

\begin{lem}
\label{lem:TwoPairingsEqual}
Two maps $p,q: v\otimes \overline{v} \to 1_\cV$ are equal if and only if for all $u\in \cV$ and $f,g\in \cV(u\to v)$, 
\begin{equation}\label{eq: inner product equal?}
p\circ (f\otimes \overline{g}) \circ \coev_u = q\circ (f\otimes \overline{g}) \circ \coev_u.
\end{equation}
\end{lem}
\begin{proof}
The forward direction is trivial.
Considering only simple $u\in\Irr(\cV)$,
the equality \eqref{eq: inner product equal?} holds iff
$$
\tikzmath[xscale=-1]{
\draw (.5,-1.7) --node[right]{$\scriptstyle u$} (.5,-1.3);
\draw (.5,-.7) --node[right]{$\scriptstyle v$} (.5,-.3);
\draw (-.5,-.7) --node[left]{$\scriptstyle \overline{v}$} (-.5,-.3);
\draw (-.5,-1.3) node[right, yshift=0.2cm]{$\scriptstyle \overline{u}$} arc (0:-180:.3cm) --node[left]{$\scriptstyle u$} (-1.1,.5);
\roundNbox{fill=white}{(0,0)}{.3}{.5}{.5}{$p$}
\roundNbox{fill=white}{(.5,-1)}{.3}{0}{0}{$f$}
\roundNbox{fill=white}{(-.5,-1)}{.3}{0}{0}{$\overline{g}$}
}\;
=
\tikzmath[xscale=-1]{
\draw (.5,-1.7) --node[right]{$\scriptstyle u$} (.5,-1.3);
\draw (.5,-.7) --node[right]{$\scriptstyle v$} (.5,-.3);
\draw (-.5,-.7) --node[left]{$\scriptstyle \overline{v}$} (-.5,-.3);
\draw (-.5,-1.3) node[right, yshift=0.2cm]{$\scriptstyle \overline{u}$} arc (0:-180:.3cm) --node[left]{$\scriptstyle u$} (-1.1,.5);
\roundNbox{fill=white}{(0,0)}{.3}{.5}{.5}{$q$}
\roundNbox{fill=white}{(.5,-1)}{.3}{0}{0}{$f$}
\roundNbox{fill=white}{(-.5,-1)}{.3}{0}{0}{$\overline{g}$}
}
\qquad\qquad
\forall\,f,g:u\to v
$$
which holds true iff
\[
p\circ (f\otimes \overline{g}) = q\circ (f\otimes\overline{g})
\qquad\qquad
\forall\,f,g:u\to v.
\]
This is true for all $u$ iff $p=q$.
\end{proof}

\begin{prop}
\label{prop:FormulaForEvTr}
$
\coev^\dag_{\bTr_\cV(c)}
=
i^\dag\circ \bTr_\cV(\coev^\dag_c)\circ \mu_{c,\overline{c}} \circ (\id\otimes [\chi^{\bTr_\cV}_c]^{-1}).$
Equivalently,
\begin{equation}\label{eq:FormulaForEvTr}
\coev^\dag_{\bTr_\cV(c)}=
\begin{tikzpicture}[baseline=1.5cm]
	\coordinate (a1) at (0,0);
	\coordinate (a2) at (1.4,0);
	\coordinate (b1) at (0,1);
	\coordinate (b2) at (1.4,1);
	\coordinate (c1) at (.7,2.5);
	\draw[thick] (c1) arc (180:0:.3cm);
	\pairOfPants{(b1)}	
	\draw[thick, blue] ($ (b2) + (.3,-.1) $) node[below]{$\scriptstyle \overline{c}$} to[in=90,out=90, looseness=2.3]($ (b1) + (.3,-.1) $) node[below]{$\scriptstyle c$};
\end{tikzpicture}
\,\,,
\end{equation}
where we've suppressed the inverse of $\chi^{\bTr_\cV}_c$ from the diagram.
\end{prop}
\begin{proof}
Fix $v\in \cV$, $f,g: v\to \bTr_\cV(c)$, and consider the morphism
\begin{equation}
\label{eq:ToApplyidag}
\tikzmath{
	\coordinate (a1) at (0,0);
	\coordinate (a2) at (1.4,0);
	\coordinate (b1) at (0,1);
	\coordinate (b2) at (1.4,1);
	\coordinate (c1) at (.7,2.5);
	\topPairOfPants{(b1)}	
	\draw[thick, blue] ($ (b2) + (.3,-.1) $)  to[in=90,out=90, looseness=2.3] ($ (b1) + (.3,-.1) $) node[left, yshift=.45cm, xshift=.15cm]{$\scriptstyle c$};
  \draw[thick, red] (.3,.6) node[left, yshift=-.2cm]{$\scriptstyle v$} arc (-180:0:.7);
  \roundNbox{fill=white}{(.3,.9)}{.3}{.2}{.2}{$f$}
  \roundNbox{fill=white}{(1.7,.9)}{.3}{.2}{.2}{$\overline{g}$}
}
=
\bTr_\cV(\coev^\dag_c)\circ \mu_{c,\overline{c}} \circ (\id_{\bTr_\cV(c)}\otimes (\chi^{\bTr_\cV}_c)^{-1}) \circ (f\otimes \overline{g})\circ \coev_v.
\end{equation}
Its mate under \eqref{eq:IdentifyGroundC*Alg} is 
$$
\varepsilon_{1_\cC} \circ \Phi(\bTr_\cV(\coev^\dag_c))\circ  \Phi(\mu_{c,\overline{c}}) \circ \Phi\big((\id_{\bTr_\cV(c)}\otimes (\chi^{\bTr_\cV}_c)^{-1}) \circ (f\otimes \overline{g})\big)\circ \Phi(\coev_v).
$$
We claim that the above morphism is equal to
\begin{equation}
\label{eq:ToApplyPsiC}
\coev^\dag_c 
\circ \big({\mate(f)} \otimes \overline{\mate(g)}\big) \circ \coev_{\Phi(v)}.
\end{equation}
Indeed, in the diagrammatic notation of \cite{MR3578212}, we check: 
\begin{align*}
\mate(\eqref{eq:ToApplyidag})
&=\begin{tikzpicture}[baseline=.4cm, scale=.8]
	\plane{(-.4,-.8)}{4.3}{2.7}
	\CMbox{box}{(.2,-.4)}{1.2}{1.4}{.5}{$\varepsilon_{1_\cC}$}
	\fill[unshaded] (-1,1.7) -- (-3.3,1.7) -- (-2.25,.65) -- ++(.2,0) -- ++(.3,-.3) -- ++(-.2,0) -- ++(.8,-.8) -- ++(.5,.1) -- (-.5,.1) -- (0,.1) -- (0,1.1) -- ++(-1,0);
	\draw[thick, blue] (-2.9,1.2) to[out=0,in=180] (-.1,.7) -- (.08-.2,.7) arc(90:-90:.1);
	\draw[thick, blue] (-1.6,-.2) to[out=10,in=180] (-.15,.5) -- (.08-.2,.5);
	\draw[thick] (-3.1,1.56) to[out=0,in=180] (-.1,1);
	\draw[thick] (-1.8,-.6) to[out=5,in=180] (-.1,.2);
	\draw[thick] (-1.8,.2) to[out=10,in=2, looseness=2] (-3.1,.84);
	\draw[thick] (-3.1,1.2) circle (.2 and .36);
	\draw[thick] (-1.8,-.2) circle (.2 and .4);
	\draw[thick] (-.1,.2) arc (-90:90:.2 and .4);
	\draw[thick, red, knot] (-4.6,-.2) node[below, xshift=-.4cm]{$\scriptstyle \Phi(\overline v)$} arc (270:180:.7cm)
	node[left, xshift=0cm]{$\scriptstyle \Phi(\coev_v)$}
	arc (180:90:.7cm) node[above]{$\scriptstyle {\Phi(v)}$} -- (-3.8,1.2);
	\roundNbox{fill=white}{(-2.1,-.2)}{.5}{2}{0}{${\Phi{\textstyle(}\scriptstyle (\chi^{\bTr_\cV}_c)^{-1}\circ \overline{g}{\textstyle)}}$}
	\roundNbox{fill=white}{(-3.2,1.2)}{.45}{.4}{0}{$\Phi(f)$}
\end{tikzpicture}
\!\!
\\
&=\begin{tikzpicture}[baseline=.4cm, scale=.8]
	\plane{(-.4,-.8)}{4.3}{2.7}
	\draw[thick, blue] (1.4,0) -- node[below]{$\scriptstyle \overline{c}$} (2,0) arc (-90:90:.2cm) -- node[above]{$\scriptstyle c$} (1.4,.4);
	\CMbox{box}{(.2,-.4)}{1.2}{1.4}{.5}{$\varepsilon_{{c}\otimes \overline{c}}$}
	\fill[unshaded] (-1,1.7) -- (-3.3,1.7) -- (-2.25,.65) -- ++(.2,0) -- ++(.3,-.3) -- ++(-.2,0) -- ++(.8,-.8) -- ++(.5,.1) -- (-.5,.1) -- (0,.1) -- (0,1.1) -- ++(-1,0);
	\draw[thick, blue] (-2.9,1.2) to[out=0,in=180] (-.1,.7) -- (.08,.7);
	\draw[thick, blue] (-1.6,-.2) to[out=10,in=180] (-.15,.5) -- (.08,.5);
	\draw[thick] (-3.1,1.56) to[out=0,in=180] (-.1,1);
	\draw[thick] (-1.8,-.6) to[out=5,in=180] (-.1,.2);
	\draw[thick] (-1.8,.2) to[out=10,in=2, looseness=2] (-3.1,.84);
	\draw[thick] (-3.1,1.2) circle (.2 and .36);
	\draw[thick] (-1.8,-.2) circle (.2 and .4);
	\draw[thick] (-.1,.2) arc (-90:90:.2 and .4);
	\draw[thick, red, knot] (-5.5,-.2) node[below, xshift=-.4cm]{$\scriptstyle \overline{\Phi(v)}$} arc (270:180:.7cm)
	node[left, xshift=0cm]{$\scriptstyle \coev_{\Phi(v)}$}
	arc (180:90:.7cm) node[above]{$\scriptstyle {\Phi(v)}$} -- (-3.8,1.2);
	\roundNbox{fill=white}{(-2.1,-.2)}{.5}{2.9}{0}{${\scriptstyle \Phi{\textstyle(}(\chi^{\bTr_\cV}_c)^{\text{-}1}\circ \overline{g}{\textstyle)}\circ (\chi_v^\Phi)^{\text{-}1}}$}
	\roundNbox{fill=white}{(-3.2,1.2)}{.45}{.4}{0}{$\Phi(f)$}
\end{tikzpicture}
\!\!
\\
&\hspace{-1.1cm}\underset{\text{\scriptsize \cite[Lem.~4.6]{MR3578212}}}{=}
\begin{tikzpicture}[baseline=.65cm, scale=.8]
	\plane{(-4.2,-.8)}{7.8}{3}
	\draw[thick, blue] (.8,-.1) node[below, xshift=.2cm]{$\scriptstyle \overline{c}$} arc (-90:90:.7cm) -- node[above]{$\scriptstyle {c}$} (-.6,1.3);
	\CMbox{box}{(-1.4,.8)}{.8}{.8}{.4}{$\varepsilon_{c}$}
	\CMbox{box}{(0,-.6)}{.8}{.8}{.4}{$\varepsilon_{\overline{c}}$}
	\straightTubeWithString{(-1.6,1.1)}{.1}{1.2}{thick, blue}
	\straightTubeWithString{(-.2,-.3)}{.1}{1.2}{thick, blue}
	\draw[thick, red, knot] (-4.2,-.1) -- (-5.2,-.1) --node[below,xshift=-.3cm]{$\scriptstyle \overline{\Phi(v)}$} (-5.6,-.1) arc (270:90:.7cm) node[above, xshift=-.3cm]{$\scriptstyle {\Phi(v)}$} -- (-3.8,1.3);
	\roundNbox{fill=white}{(-1.8,-.1)}{.4}{3}{.1}{${\scriptstyle \Phi{\textstyle(}(\chi^{\bTr_\cV}_c)^{\text{-}1}\circ \overline{g}{\textstyle)}\circ (\chi_v^\Phi)^{\text{-}1}}$}
	\roundNbox{fill=white}{(-3.2,1.3)}{.4}{.4}{.1}{$\Phi(f)$}
\end{tikzpicture}
\\
&\hspace{-.5cm}\underset{\text{\scriptsize (Lem.~\ref{lem:MateOfBarF})}}{=}
\;\;
\begin{tikzpicture}[baseline=.65cm, scale=.8]
	\plane{(-1,-.8)}{5.6}{3}
	\draw[thick, blue] (1.8,-.1) node[below, xshift=.2cm]{$\scriptstyle \overline{c}$} arc (-90:90:.7cm) -- node[above]{$\scriptstyle c$} (.4,1.3);
	\CMbox{box}{(-1.4,.8)}{1.8}{.8}{.4}{$\scriptstyle \mate(f)$}
	\CMbox{box}{(0,-.6)}{1.8}{.8}{.4}{$\scriptstyle \overline{\mate(g)}$}
	\draw[thick, red, knot] (-.2,-.1) -- node[below, xshift=.1cm]{$\scriptstyle \overline{\Phi(v)}$} (-1.6,-.1) arc (270:90:.7cm) node[above, xshift=-.6cm, yshift=-.2cm]{$\scriptstyle \Phi(v)$};
\end{tikzpicture}
.
\end{align*}
Now, by Lemma \ref{lem:IdentifyStatesOnGroundC*Alg}, $i^\dag\circ \eqref{eq:ToApplyidag}=\psi_\cC(\eqref{eq:ToApplyPsiC})$. 
By \eqref{eq: heart} and the definition of $\langle\cdot\,,\cdot\rangle$,
this is equal to 
$$
\langle \mate(f),\mate(g)\rangle_{\cC(\Phi(v)\to c)}
=
\langle f, g\rangle_{\cV(v\to \bTr(c))}
=
\tr_\cV(fg^\dag)
=
\coev^\dag_{\bTr_\cV(c)}\circ (f\otimes \overline{g})\circ \coev_{v}.
$$ 
Hence 
$$
i^\dag\circ \eqref{eq:ToApplyidag} = \coev^\dag_{\bTr_\cV(c)}\circ (f\otimes \overline{g})\circ \coev_{v}
\qquad\qquad
\forall\,f,g:v\to \bTr_\cV(c).
$$
The result follows by Lemma \ref{lem:TwoPairingsEqual}.
\end{proof}

\begin{rem}
The problem of showing the pairing on the right hand side in the statement of Proposition \ref{prop:FormulaForEvTr} is non-degenerate was left open in \cite[Rem.~5.4]{MR3578212}.
As it is equal to $\coev^\dag_{\bTr_\cV(c)}$ (up to suppressing $\chi$), this open question is now resolved in the unitary setting.
\end{rem}

\begin{lem}
\label{lem:idagAllowsIsotopy}
If $\psi_\cC$ is spherical, 
then the following maps $\bTr_\cV(c\otimes \overline{c})\to 1_\cV$ are equal:
$$
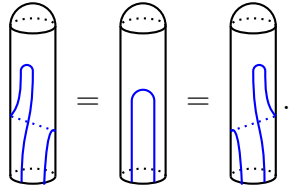
\begin{tikzpicture}[baseline=.9cm, xscale=-1]
	\coordinate (a1) at (0,0);
	\coordinate (b1) at (0,2);
	\coordinate (c) at (.15,.2);
	\draw[thick] (a1) -- (b1);
	\draw[thick] ($ (a1) + (.6,0) $) -- ($ (b1) + (.6,0) $);
	\halfDottedEllipse{(0,0)}{.3}{.1}
	\halfDottedEllipse{(0,2)}{.3}{.1}
	\draw[thick] (b1) arc (180:0:.3cm);		
	\draw[thick, blue] (c)+(.15,1.2) arc (180:0:.08cm)  .. controls ++(270:.2cm) and ++(135:.2cm) .. ($ (c) + (.45,.6) $);	
	\draw[thick, blue] ($ (a1) + (.15,-.08) $) .. controls ++(270:.2cm) and ++(45:.2cm) .. ($ (c) + (-.15,.4) $);
	\draw[thick, blue, dotted] ($ (c) + (-.15,.4) $) -- ($ (c) + (.45,.6) $);	
	\draw[thick, blue] (c)++(.15,1.2) .. controls ++(270:.8cm) and ++(90:.8cm) .. ($ (a1) + (.45,-.08) $);	
\end{tikzpicture}
\,\,=\,\,
\begin{tikzpicture}[baseline=.9cm]
	\coordinate (a1) at (0,0);
	\coordinate (b1) at (0,2);
	\coordinate (c) at (.15,.4);
	\draw[thick] (a1) -- (b1);
	\draw[thick] ($ (a1) + (.6,0) $) -- ($ (b1) + (.6,0) $);
	\halfDottedEllipse{(0,0)}{.3}{.1}
	\halfDottedEllipse{(0,2)}{.3}{.1}
	\draw[thick] (b1) arc (180:0:.3cm);			
	\draw[thick, blue] ($ (a1) + (.15,-.08) $) -- ($ (a1) + (.15,1) $) arc (180:0:.15cm) -- ($ (a1) + (.45,-.08) $);	
\end{tikzpicture}
\,\,=\,\,
\begin{tikzpicture}[baseline=.9cm]
	\coordinate (a1) at (0,0);
	\coordinate (b1) at (0,2);
	\coordinate (c) at (.15,.2);
	\draw[thick] (a1) -- (b1);
	\draw[thick] ($ (a1) + (.6,0) $) -- ($ (b1) + (.6,0) $);
	\halfDottedEllipse{(0,0)}{.3}{.1}
	\halfDottedEllipse{(0,2)}{.3}{.1}
	\draw[thick] (b1) arc (180:0:.3cm);		
	\draw[thick, blue] (c)+(.15,1.2) arc (180:0:.08cm)  .. controls ++(270:.2cm) and ++(135:.2cm) .. ($ (c) + (.45,.6) $);	
	\draw[thick, blue] ($ (a1) + (.15,-.08) $) .. controls ++(270:.2cm) and ++(45:.2cm) .. ($ (c) + (-.15,.4) $);
	\draw[thick, blue, dotted] ($ (c) + (-.15,.4) $) -- ($ (c) + (.45,.6) $);	
	\draw[thick, blue] (c)++(.15,1.2) .. controls ++(270:.8cm) and ++(90:.8cm) .. ($ (a1) + (.45,-.08) $);	
\end{tikzpicture}\,.
$$
\end{lem}
Note that the first and third tubes above represent the same morphism regardless of whether $\psi_\cC$ is spherical or not, because the following two pictures are isotopic:
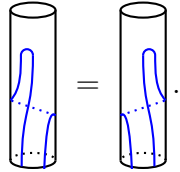
\[
\begin{tikzpicture}[baseline=.9cm, xscale=-1]
	\coordinate (a1) at (0,0);
	\coordinate (b1) at (0,2);
	\coordinate (c) at (.15,.2);
	\draw[thick] (a1) -- (b1);
	\draw[thick] ($ (a1) + (.6,0) $) -- ($ (b1) + (.6,0) $);
	\halfDottedEllipse{(0,0)}{.3}{.1}
	\draw[thick] (.3,2) ellipse (.3 and .1);
	\draw[thick, blue] (c)+(.15,1.2) arc (180:0:.08cm)  .. controls ++(270:.2cm) and ++(135:.2cm) .. ($ (c) + (.45,.6) $);	
	\draw[thick, blue] ($ (a1) + (.15,-.08) $) .. controls ++(270:.2cm) and ++(45:.2cm) .. ($ (c) + (-.15,.4) $);
	\draw[thick, blue, dotted] ($ (c) + (-.15,.4) $) -- ($ (c) + (.45,.6) $);	
	\draw[thick, blue] (c)++(.15,1.2) .. controls ++(270:.8cm) and ++(90:.8cm) .. ($ (a1) + (.45,-.08) $);	
\end{tikzpicture}
\,\,=\,\,
\begin{tikzpicture}[baseline=.9cm]
	\coordinate (a1) at (0,0);
	\coordinate (b1) at (0,2);
	\coordinate (c) at (.15,.2);
	\draw[thick] (a1) -- (b1);
	\draw[thick] ($ (a1) + (.6,0) $) -- ($ (b1) + (.6,0) $);
	\halfDottedEllipse{(0,0)}{.3}{.1}
	\draw[thick] (.3,2) ellipse (.3 and .1);
	\draw[thick, blue] (c)+(.15,1.2) arc (180:0:.08cm)  .. controls ++(270:.2cm) and ++(135:.2cm) .. ($ (c) + (.45,.6) $);	
	\draw[thick, blue] ($ (a1) + (.15,-.08) $) .. controls ++(270:.2cm) and ++(45:.2cm) .. ($ (c) + (-.15,.4) $);
	\draw[thick, blue, dotted] ($ (c) + (-.15,.4) $) -- ($ (c) + (.45,.6) $);	
	\draw[thick, blue] (c)++(.15,1.2) .. controls ++(270:.8cm) and ++(90:.8cm) .. ($ (a1) + (.45,-.08) $);	
\end{tikzpicture}
\,.
\]
\begin{proof}
By the above discussion, it is enough to prove the first equality.
Two maps $f,g:v\to 1_\cV$ are equal if and only if $f\circ h = g\circ h$ for all $h: 1_\cV\to v$,
so we fix $h: 1_\cV \to \bTr_\cV(c\otimes \overline{c})$ and wish to show that
\begin{equation}
\label{eq:TwistTubesNoTopCap}
i^\dagger\circ
\begin{tikzpicture}[baseline=.9cm, xscale=-1]
	\coordinate (a1) at (0,0);
	\coordinate (b1) at (0,2);
	\coordinate (c) at (.15,.2);
	\draw[thick] (a1) -- (b1);
	\draw[thick] ($ (a1) + (.6,0) $) -- ($ (b1) + (.6,0) $);
	\draw[thick] (.3,2) ellipse (.3 and .1);
	\draw[thick, blue] (c)+(.15,1.2) arc (180:0:.08cm)  .. controls ++(270:.2cm) and ++(135:.2cm) .. ($ (c) + (.45,.6) $);	
	\draw[thick, blue] ($ (a1) + (.15,-.08) $) .. controls ++(270:.2cm) and ++(45:.2cm) .. ($ (c) + (-.15,.4) $);
	\draw[thick, blue, dotted] ($ (c) + (-.15,.4) $) -- ($ (c) + (.45,.6) $);	
	\draw[thick, blue] (c)++(.15,1.2) .. controls ++(270:.8cm) and ++(90:.8cm) .. ($ (a1) + (.45,-.08) $);
  \roundNbox{fill=white}{(.3,-.3)}{.3}{.2}{.2}{$h$}
\end{tikzpicture}
\quad=\quad
i^\dagger\circ
\begin{tikzpicture}[baseline=.9cm]
	\coordinate (a1) at (0,0);
	\coordinate (b1) at (0,2);
	\coordinate (c) at (.15,.4);
	\draw[thick] (a1) -- (b1);
	\draw[thick] ($ (a1) + (.6,0) $) -- ($ (b1) + (.6,0) $);
	\draw[thick] (.3,2) ellipse (.3 and .1);
	\draw[thick, blue] ($ (a1) + (.15,-.08) $) -- ($ (a1) + (.15,1) $) arc (180:0:.15cm) -- ($ (a1) + (.45,-.08) $);	
  \roundNbox{fill=white}{(.3,-.3)}{.3}{.2}{.2}{$h$}
\end{tikzpicture}
\end{equation}
By Lemma \ref{lem:IdentifyStatesOnGroundC*Alg}, this is the same as showing that 
\[
\left.\left.\psi_\cC\left(
\begin{tikzpicture}[baseline=.22cm, scale=.67]
	\plane{(0,-.9)}{4.2}{2.7}
	\CMbox{box}{(.23,-.4)}{1.5}{1.4}{.5}{$\,\varepsilon$}
	\fill[white] (-3,.1) rectangle (-1,1.1);
	\straightTubeNoString{(-.1,.2)}{.2}{3}
\pgftransformyshift{17}
\pgftransformyscale{1.3}
\pgftransformxshift{-4}
\pgftransformxscale{2}
\pgftransformrotate{-90}
	\draw[thick, blue] (-.15,-1.07) .. controls ++(90:.4cm) and ++(270:.4cm) .. (.1-.08,-.08);
	\draw[thick, blue] (.15,-1.07) .. controls ++(90:.2cm) and ++(225:.1cm) .. (.3,-.62);		
	\draw[thick, blue] (-.1-.03,-.08) .. controls ++(270:.2cm) and ++(45:.1cm) .. (-.3,-.52);
	\draw[thick, blue, dotted] (.3,-.62) -- (-.3,-.52);	
\pgftransformrotate{90}
\pgftransformxscale{1/2}
\pgftransformxshift{4}
\pgftransformyscale{1/1.3}
\pgftransformyshift{-17}
	\draw[thick] (-1.9,1) -- ++(-.3,0) arc (90:270:.4) -- ++(.3,0);
	\draw [blue,thick] (-.31,.571) arc (-90:90:.098);
	\roundNbox{fill=white}{(-3,.6)}{.6}{.2}{.2}{$\Phi(h)$}
\end{tikzpicture}\,\,\,
\right)
\,\,=\,\,
\psi_\cC\right(
\begin{tikzpicture}[baseline=.22cm, scale=.67]
	\plane{(0,-.9)}{4.2}{2.7}
	\CMbox{box}{(.23,-.4)}{1.5}{1.4}{.5}{$\,\varepsilon$}
	\fill[white] (-3,.1) rectangle (-1,1.1);
	\straightTubeNoString{(-.1,.2)}{.2}{3}
	\draw[thick, blue] (-2.2,.46) -- +(1.5,0)arc (-90:90:.15) -- +(-1.5,0);
	\roundNbox{fill=white}{(-3,.6)}{.6}{.2}{.2}{$\Phi(h)$}
\end{tikzpicture}\,\,\,
\right).
\]
The result follows since
\[
\begin{tikzpicture}[baseline=.22cm, scale=.67]
	\plane{(0,-.9)}{4.2}{2.7}
	\CMbox{box}{(.23,-.4)}{1.5}{1.4}{.5}{$\,\varepsilon$}
	\fill[white] (-3,.1) rectangle (-1,1.1);
	\straightTubeNoString{(-.1,.2)}{.2}{3}
\pgftransformyshift{17}
\pgftransformyscale{1.3}
\pgftransformxshift{-4}
\pgftransformxscale{2}
\pgftransformrotate{-90}
	\draw[thick, blue] (-.15,-1.07) .. controls ++(90:.4cm) and ++(270:.4cm) .. (.1-.08,-.08);
	\draw[thick, blue] (.15,-1.07) .. controls ++(90:.2cm) and ++(225:.1cm) .. (.3,-.62);		
	\draw[thick, blue] (-.1-.03,-.08) .. controls ++(270:.2cm) and ++(45:.1cm) .. (-.3,-.52);
	\draw[thick, blue, dotted] (.3,-.62) -- (-.3,-.52);	
\pgftransformrotate{90}
\pgftransformxscale{1/2}
\pgftransformxshift{4}
\pgftransformyscale{1/1.3}
\pgftransformyshift{-17}
	\draw[thick] (-1.9,1) -- ++(-.3,0) arc (90:270:.4) -- ++(.3,0);
	\draw [blue,thick] (-.31,.571) arc (-90:90:.098);
	\roundNbox{fill=white}{(-3,.6)}{.6}{.2}{.2}{$\Phi(h)$}
\end{tikzpicture}
\!\!\!\!\!\!\!=
\begin{tikzpicture}[baseline=.22cm, scale=.67]
	\plane{(0,-.9)}{4.2}{2.7}
	\CMbox{box}{(.23,-.4)}{1.5}{1.4}{.5}{$\,\varepsilon$}
	\fill[white] (-3,.1) rectangle (-1,1.1);
	\straightTubeNoString{(-.1,.2)}{.2}{3}
\pgftransformyshift{17}
\pgftransformyscale{1.3}
\pgftransformxshift{-4}
\pgftransformxscale{2}
\pgftransformrotate{-90}
	\draw[thick, blue] (-.15,-1.07) .. controls ++(90:.4cm) and ++(270:.4cm) .. (.1-.08,-.08);
	\draw[thick, blue] (.15,-1.07) .. controls ++(90:.2cm) and ++(225:.1cm) .. (.3,-.62);		
	\draw[thick, blue] (-.1-.03,-.08) .. controls ++(270:.2cm) and ++(45:.1cm) .. (-.3,-.52);
	\draw[thick, blue, dotted] (.3,-.62) -- (-.3,-.52);	
\pgftransformrotate{90}
\pgftransformxscale{1/2}
\pgftransformxshift{4}
\pgftransformyscale{1/1.3}
\pgftransformyshift{-17}
	\draw[thick] (-1.9,1) -- ++(-.3,0) arc (90:270:.4) -- ++(.3,0);
	\draw [blue,thick] (-.31,.571) -- (.1,.571);
	\draw [blue,thick] (-.31,.766) -- (.09,.766);
	\roundNbox{fill=white}{(-3,.6)}{.6}{.2}{.2}{$\Phi(h)$}
	\draw[thick, blue] (1.7,-.1) -- (2.4,-.1) arc (-90:90:.15cm) -- (1.7,.2);
\end{tikzpicture}
\!\!\!\!\!\!=
\begin{tikzpicture}[baseline=.133cm, scale=.87]
	\draw[thick, blue] (1.2,-.1) arc (90:-90:.25cm) -- (0,-.6) .. controls ++(180:1cm) and ++(180:1.2cm) .. (-.3,1.15) -- (.9+.05,1.15) .. controls ++(0:1cm) and ++(0:1.2cm) .. (1.2,.1);
	\CMbox{box}{(0,-.4)}{1.2}{1}{.4}{$\varepsilon$}
	\filldraw[line width=3, white] (-.25,0) arc (-90:90:.15 and .3) -- ++(-.8,0) arc (90:270:.3) -- cycle;
	\draw[thick] (-.25,0) arc (-90:90:.15 and .3) -- ++(-1,0) arc (90:270:.3) -- cycle;
	\draw [blue,thick] (-1.2,.4) -- (-.1,.4);
	\draw [blue,thick] (-1.2,.2) -- (-.1,.2);
	\plane{(-.5,-.8)}{4.2}{2.1}
	\roundNbox{fill=white}{(-1.7,.28)}{.47}{.15}{.15}{$\Phi(h)$}
\end{tikzpicture}
\]
and
$$
\psi_\cC\left(
\,\,\,
\begin{tikzpicture}[baseline=.133cm, scale=.87]
	\draw[thick, blue] (1.2,-.1) arc (90:-90:.25cm) -- (0,-.6) .. controls ++(180:1cm) and ++(180:1.2cm) .. (-.3,1.15) -- (.9+.05,1.15) .. controls ++(0:1cm) and ++(0:1.2cm) .. (1.2,.1);
	\CMbox{box}{(0,-.4)}{1.2}{1}{.4}{$\varepsilon$}
	\filldraw[line width=3, white] (-.25,0) arc (-90:90:.15 and .3) -- ++(-.8,0) arc (90:270:.3) -- cycle;
	\draw[thick] (-.25,0) arc (-90:90:.15 and .3) -- ++(-1,0) arc (90:270:.3) -- cycle;
	\draw [blue,thick] (-1.2,.4) -- (-.1,.4);
	\draw [blue,thick] (-1.2,.2) -- (-.1,.2);
	\plane{(-.5,-.8)}{4.2}{2.1}
	\roundNbox{fill=white}{(-1.7,.3)}{.4}{.15}{.15}{$\Phi(h)$}
\end{tikzpicture}
\,\,\,
\right)
\quad
\underset{\text{\scriptsize ($\psi_\cC$ spherical)}}{=}
\quad
\psi_\cC\left(
\begin{tikzpicture}[baseline=.1cm, scale=.9]
	\plane{(0,-.6)}{2.3}{1.6}
	\fill[white] (-1,.4) rectangle (0,.8);
	\draw[thick, blue] (.8,-.1) -- (1.2,-.1) arc (-90:90:.1cm) -- (.8,.1);
	\CMbox{box}{(.1,-.4)}{.85}{.8}{.4}{$\varepsilon$}
	\straightTubeWithCap{(-.1,0)}{.1}{.8}
	\draw[thick, blue] (-.7,.3) -- (0,.3);
	\draw[thick, blue] (-.7,.1) -- (0,.1);
	\roundNbox{fill=white}{(-1.2,.2)}{.4}{.15}{.15}{$\Phi(h)$}
\end{tikzpicture}
\,\,\,
\right).
$$
\end{proof}

\subsection{Unitarity of the traciator}
\label{sec:UnitaryTraciator}

Suppose $(\cC,\vee_\cC,\psi_\cC,\Phi^{\scriptscriptstyle \cZ})$ is a unitary module multitensor category over $\cV$ and $\bTr_\cV:\cC\to \cV$ is the right unitary adjoint of $\Phi: \cV\to \cC$.
Recall that each object $\Phi(v)$ comes with a unitary half-braiding $e_{\Phi(v),x}:\Phi(v)\otimes{x}\,\to\,{x}\otimes \Phi(v)$.
In this section, we explore an extra coherence between these unitary half-braidings and the involutive structure of the categorified trace.

\begin{lem}
\label{lem:HalfBraidingChiCoherence}
We have the following coherence between the half-braiding $e_{\Phi(v),\bullet}$ and $\chi^\Phi_v$:
$$
(\id_{\overline{b}}\otimes \chi^\Phi_v)
\circ
e_{\Phi(\overline{v}),\overline{b}}
=
\big[\nu_{b,\Phi(v)}^{-1}
\circ
\overline{e^{-1}_{\Phi(v),b}}
\circ
\nu_{\Phi(v),b}\big]
\circ
(\chi^\Phi_v\otimes \id_{\overline{b}}).
$$
\end{lem}
\begin{proof}
Letting $(\overline{\Phi(v)},e_{\overline{\Phi(v)},\bullet})$ be the conjugate of $(\Phi(v),e_{\Phi(v),\bullet})$ in $Z^\dagger(\cC)$, we have 
$e_{\overline{\Phi(v)},b}=\nu_{b,\Phi(v)}^{-1}
\circ
\overline{e^{-1}_{\Phi(v),b}}
\circ
\nu_{\Phi(v),b}$, so
the right hand side in the statement of the lemma simplifies to
$e_{\overline{\Phi(v)},b}\circ (\chi^\Phi_v \otimes \id_{\overline{b}})$.
(Note that since the half-braiding is unitary, we have $\overline{e^{-1}_{\Phi(v),b}}=e^\vee_{\Phi(v),b}$.)
The result holds since $\chi^\Phi_v: \Phi(\overline{v})\to \overline{\Phi(v)}$ is a morphism in $Z^\dag(\cC)$.
\end{proof}

The following lemma expresses the coherence from Lemma \ref{lem:HalfBraidingChiCoherence} in terms of traciators.

\begin{lem}
\label{lem:BarTraciatorCompatibility}
We have the following compatibility of the tracitator with the involutive structure of $\bTr_\cV$:
\begin{equation}\label{eq: octopus}
\overline{\tau_{x,y}}
\circ 
\chi_{x\otimes y} 
\circ 
\bTr_\cV(\nu_{y,x}) 
=
\chi_{y\otimes x} \circ \bTr_\cV(\nu_{x,y})\circ \tau_{\overline{x},\overline{y}}^{-1}
\end{equation}
\end{lem}
\begin{proof}
We prove this equality after talking inverses on both sides, and after taking mates under the adjunction
$$
\cV\big(\overline{\bTr_\cV(y\otimes x)} \to \bTr_\cV(\overline{y}\otimes \overline{x})\big)
\cong
\cC\big(\Phi(\overline{\bTr_\cV(y\otimes x)})\to \overline{y}\otimes \overline{x}\big).
$$
The mate of the inverse of the left hand side of \eqref{eq: octopus} is
\begin{align*}
\mate(
\bTr_\cV(\nu_{y,x}^{-1}) 
\circ 
\chi_{x\otimes y}^{-1}
\circ 
\overline{\tau_{x,y}^{-1}}
)
&=
\nu_{y,x}^{-1}
\circ
\mate(\chi_{x\otimes y}^{-1}
)
\circ
\Phi(\overline{\tau_{x,y}^{-1}})
\\&=
\nu_{y,x}^{-1}\circ\varepsilon_{\overline{x\otimes y}}\circ \Phi(\chi_{x\otimes y}^{-1})\circ \Phi(\overline{\tau_{x,y}^{-1}})
\\&=
\nu_{y,x}^{-1}
\circ
\overline{\varepsilon_{x\otimes y}}
\circ 
\chi^\Phi_{\bTr_\cV(x\otimes y)}
\circ
\Phi(\overline{\tau_{x,y}^{-1}})
\\&=
\nu_{y,x}^{-1}
\circ
\overline{\varepsilon_{x\otimes y}}
\circ
\overline{\Phi(\tau_{x,y}^{-1})}
\circ 
\chi^\Phi_{\bTr_\cV(y\otimes x)}
\\&=
\nu_{y,x}^{-1}
\circ 
\overline{\varepsilon_{x\otimes y}
\circ 
\Phi(\tau_{x,y}^{-1})}
\circ
\chi^\Phi_{\bTr_\cV(y\otimes x)}
\\&=
\nu_{y,x}^{-1}\circ \overline{\mate(\tau_{x,y}^{-1})} \circ \chi^\Phi_{\bTr_\cV(y\otimes x)}
\end{align*}
where the third equality holds by the last equation in Proposition~\ref{prop:AdjointInvolutive}.

The mate of the inverse of the right hand side of \eqref{eq: octopus} is 
$\mate(\tau_{\overline{x},\overline{y}}) \circ \Phi(\bTr_\cV(\nu_{x,y}^{-1}))\circ \Phi(\chi_{y\otimes x}^{-1})$.
Expanding, we get the map going right and then down in the pasting diagram below.
To save space, we omit tensor symbols and subscripts.
\[
\begin{tikzpicture}[baseline= (a).base]
\node[scale=.75] (a) at (0,0){
\begin{tikzcd}[column sep = 5.5em]
\Phi(\overline{\bTr_\cV(y\cdot x)})
\arrow[rr,"\Phi(\chi^{-1})"]
\arrow[d,"\chi^\Phi"]
\arrow[dr,"\id\cdot \ev^\dag"]
&&
\Phi(\bTr_\cV(\overline{y\cdot x}))
\arrow[r,"\Phi(\bTr_\cV(\nu^{-1}))"]
\arrow[d,"\id\cdot\ev^\dag"]
&
\Phi(\bTr_\cV(\overline{x}\cdot \overline{y}))
\arrow[d,"\id\cdot\ev^\dag"]
\\
\overline{\Phi(\bTr_\cV(y\cdot x))}
\arrow[d,"\id\cdot\ev^\dag"]
&
\Phi(\overline{\bTr_\cV(y\cdot x)}) \cdot \overline{y}\cdot y
\arrow[r,"\Phi(\chi^{-1})\cdot\id"]
\arrow[d,"e\cdot\id"]
\arrow[dl,"\chi^\Phi\cdot\id",swap]
&
\Phi(\bTr_\cV(\overline{y\cdot x})) \cdot \overline{y}\cdot y
\arrow[r,"\Phi(\bTr_\cV(\nu^{-1}))\cdot \id"]
\arrow[d,"e\cdot\id"]
&
\Phi(\bTr_\cV(\overline{x}\cdot \overline{y})) \cdot \overline{y}\cdot y
\arrow[d,"e\cdot\id"]
\\
\overline{\Phi(\bTr_\cV(y\cdot x))} \cdot \overline{y}\cdot y
\arrow[d,"\nu\cdot\id"]
&
\overline{y}\cdot \Phi(\overline{\bTr_\cV(y\cdot x)}) \cdot y
\arrow[r,"\id\cdot \Phi(\chi^{-1})\cdot\id"]
\arrow[dd,"\id\cdot \chi^\Phi \cdot \id"]
\arrow[ddr,phantom,"\text{\scriptsize{Prop.~\ref{prop:AdjointInvolutive}}}"]
&
\overline{y}\cdot \Phi(\bTr_\cV(\overline{y\cdot x})) \cdot y
\arrow[r,"\id\cdot \Phi(\bTr_\cV(\nu^{-1}))\cdot \id"]
\arrow[dd,"\id\cdot \varepsilon \cdot \id"]
&
\overline{y}\cdot \Phi(\bTr_\cV(\overline{x}\cdot \overline{y})) \cdot y
\arrow[d,"\id\cdot \varepsilon \cdot \id"]
\\
\overline{y\cdot \Phi(\bTr_\cV(y\cdot x))}\cdot y
\arrow[d,"\overline{e^{-1}} \cdot \id"]
&&&
\overline{y}\cdot\overline{x}\cdot\overline{y}\cdot y
\arrow[d,"\id\cdot \ev"]
\\
\overline{\Phi(\bTr_\cV(y\cdot x))\cdot y}\cdot y
\arrow[r,"\nu^{-1}\cdot \id"]
\arrow[uur,phantom,"\text{\scriptsize{Lem.~\ref{lem:HalfBraidingChiCoherence}}}"]
&
\overline{y}\cdot \overline{\Phi(\bTr_\cV(y\cdot x))}\cdot y
\arrow[r,"\id\cdot\overline{\varepsilon}\cdot \id"]
&
\overline{y}\cdot\overline{y\cdot x}\cdot y
\arrow[ur,"\id\cdot \nu^{-1}\cdot \id"]
&
\overline{y}\cdot\overline{x}
\end{tikzcd}
};
\end{tikzpicture}
\]
It remains to prove that the map going down and then right in the above pasting diagram is equal to $\nu_{y,x}^{-1}\circ \overline{\mate(\tau_{x,y}^{-1})} \circ \chi^\Phi_{\bTr_\cV(y\otimes x)}$.
Equivalently, we must show that
\begin{align*}
\overline{\mate(\tau_{x,y}^{-1})}
=
\nu_{y,x}
&\circ
(\id_{\overline{y}\otimes \overline{x}}\otimes \ev_y)
\circ
(\id_{\overline{y}}\otimes (\nu_{x,y}^{-1}\circ \overline{\varepsilon_{y\otimes x}})\otimes \id_{y})
\\&\circ
((\nu_{y, \Phi(\bTr_\cV(y\otimes x))}^{-1}
\circ
\overline{e^{-1}_{\Phi(\bTr_\cV(y\otimes x)),y}})\otimes \id_y)
\circ
(\id_{\overline{\Phi(\bTr_\cV(y\otimes x))}}\otimes \ev_y^\dag).
\end{align*}
By the coherences for $(\overline{\,\cdot\,},\nu)$, the above equality is equivalent to
\begin{align*}
\mate(\tau_{x,y}^{-1})
=
(\coev_y^\dag \otimes \id_{x\otimes y})
&\circ
(\id_{\overline{y}}\otimes \varepsilon_{y\otimes x}\otimes \id_y)
\\&\circ
(\id_{\overline{y}}\otimes e^{-1}_{\Phi(\bTr_\cV(y\otimes x)),y})
\circ
(\coev_y \otimes \id_{\Phi(\bTr_\cV(y\otimes x))}),
\end{align*}
which is exactly \cite[Lem.~4.14]{MR3578212}.
\end{proof}

In Jones' planar algebras \cite{MR4374438}, unitarity of the rotation plays a crucial
role. 
Under the equivalence between anchored planar algebras and pivotal module tensor categories,
the analog of rotation is given by the traciator.
The next result says that traciators, and thus rotations in an anchored planar algebra, are unitary assuming sphericality.

\begin{prop}
If the state $\psi_\cC$ is spherical, then the traciator $\tau_{a,b}$ is unitary.
\end{prop}
\begin{proof}
Since $\tau_{a,b}$ is invertible, it suffices to show that $\tau_{a,b}^\dag \tau_{a,b} = \id_{\bTr_\cV(a\otimes b)}$.
Since $\tau^\dag = \overline{\tau}^\vee$, by Lemma \ref{lem:BarTraciatorCompatibility} (and suppressing various coherences) we have $\tau^\dagger=(\tau^{-1})^\vee$.
Thus:
\begin{align*}
\tau_{a,b}^\dag \tau_{a,b}
&=
\tikzmath{
\draw[thick] (-.3,-1) -- (-.3,1);
\draw[thick] (.3,-1) -- (.3,1);
\draw[thick] (1.2,-2) -- (1.2,1);
\draw[thick] (1.8,-2) -- (1.8,1);
\halfDottedEllipse{(1.2,-2)}{.3}{.1}
\draw[thick] (-1.2,2) -- (-1.2,-1);
\draw[thick] (-1.8,2) -- (-1.8,-1);
\draw[thick] (-1.5,2) ellipse (.3 and .1);
	\draw[thick, red] (.1,-1) .. controls ++(90:.8cm) and ++(270:.8cm) .. (-.1,1);		
	\draw[thick, blue] (-.1,-1) .. controls ++(90:.2cm) and ++(-45:.2cm) .. (-.3,-.2);		
	\draw[thick, blue] (.1,1) .. controls ++(270:.2cm) and ++(135:.2cm) .. (.3,.2);
	\draw[thick, blue, dotted] (-.3,-.2) -- (.3,.2);	
	\draw[thick, red] (1.4,-2.08) -- (1.4,-1) .. controls ++(90:.8cm) and ++(270:.8cm) .. (1.6,1);		
	\draw[thick, blue] (1.6,-2.08) -- (1.6,-1) .. controls ++(90:.2cm) and ++(225:.2cm) .. (1.8,-.2);		
	\draw[thick, blue] (1.4,1) .. controls ++(270:.2cm) and ++(45:.2cm) .. (1.2,.2);
	\draw[thick, blue, dotted] (1.2,.2) -- (1.8,-.2);	
  \draw[thick, blue] (-1.4,1.92) -- (-1.4,-1);		
	\draw[thick, red] (-1.6,1.92) -- (-1.6,-1);
\roundNbox{fill=white}{(.8,1.3)}{.35}{1}{.9}{$\coev^\dag_{\bTr_\cV(b\otimes a)}$}
\roundNbox{fill=white}{(-.8,-1.3)}{.35}{.9}{1}{$\ev^\dag_{\bTr_\cV(a\otimes b)}$}
}
\underset{\text{\scriptsize (Prop.~\ref{prop:FormulaForEvTr})}}{=}
\tikzmath{
\draw[thick] (-.3,-1) -- (-.3,1);
\draw[thick] (.3,-1) -- (.3,1);
\draw[thick] (1.1,-2) -- (1.1,1);
\draw[thick] (1.7,-2) -- (1.7,1);
\halfDottedEllipse{(1.1,-2)}{.3}{.1}
\draw[thick] (-1.2,2) -- (-1.2,-1);
\draw[thick] (-1.8,2) -- (-1.8,-1);
\draw[thick] (-1.5,2) ellipse (.3 and .1);
\draw[thick] (.4,2.5) arc (180:0:.3cm);
	\pairOfPants{(-.3,1)}	
	\draw[thick, red] (.1,-1) .. controls ++(90:.8cm) and ++(270:.8cm) .. (-.1,1);		
	\draw[thick, blue] (-.1,-1) .. controls ++(90:.2cm) and ++(-45:.2cm) .. (-.3,-.2);		
	\draw[thick, blue] (.1,1) .. controls ++(270:.2cm) and ++(135:.2cm) .. (.3,.2);
	\draw[thick, blue, dotted] (-.3,-.2) -- (.3,.2);	
	\draw[thick, red] (1.3,-2.08) -- (1.3,-1) .. controls ++(90:.8cm) and ++(270:.8cm) .. (1.5,1);		
	\draw[thick, blue] (1.5,-2.08) -- (1.5,-1) .. controls ++(90:.2cm) and ++(225:.2cm) .. (1.7,-.2);		
	\draw[thick, blue] (1.3,1) .. controls ++(270:.2cm) and ++(45:.2cm) .. (1.1,.2);
	\draw[thick, blue, dotted] (1.1,.2) -- (1.7,-.2);	
  \draw[thick, red] (-.1,1) .. controls ++(90:1.2cm) and ++(90:1.2cm) .. (1.5,1);
  \draw[thick, blue] (.1,1) .. controls ++(90:1cm) and ++(90:1cm) .. (1.3,1);
  \draw[thick, blue] (-1.4,1.92) -- (-1.4,-1);		
	\draw[thick, red] (-1.6,1.92) -- (-1.6,-1);
\roundNbox{fill=white}{(-.8,-1.3)}{.35}{.9}{1}{$\ev^\dag_{\bTr_\cV(a\otimes b)}$}
}
\displaybreak[1]\\&
\underset{\text{\scriptsize (Lem.~\ref{lem:idagAllowsIsotopy})}}{=}
\tikzmath{
\draw[thick] (-.3,-1) -- (-.3,1);
\draw[thick] (.3,-1) -- (.3,1);
\draw[thick] (1.1,-2) -- (1.1,1);
\draw[thick] (1.7,-2) -- (1.7,1);
\halfDottedEllipse{(1.1,-2)}{.3}{.1}
\draw[thick] (-1.2,2) -- (-1.2,-1);
\draw[thick] (-1.8,2) -- (-1.8,-1);
\draw[thick] (-1.5,2) ellipse (.3 and .1);
\draw[thick] (.4,2.5) arc (180:0:.3cm);
	\pairOfPants{(-.3,1)}	
  \draw[thick, blue] (-.1,1) .. controls ++(90:1.2cm) and ++(90:1.2cm) .. (1.5,1);
  \draw[thick, red] (.1,1) .. controls ++(90:1cm) and ++(90:1cm) .. (1.3,1);
  \draw[thick, red] (.1,-1) -- (.1,1);		
	\draw[thick, blue] (-.1,-1) -- (-.1,1);		
	\draw[thick, red] (1.3,-2.08) -- (1.3,1);
  \draw[thick, blue] (1.5,-2.08) -- (1.5,1);
  \draw[thick, blue] (-1.4,1.92) -- (-1.4,-1);		
	\draw[thick, red] (-1.6,1.92) -- (-1.6,-1);
\roundNbox{fill=white}{(-.8,-1.3)}{.35}{.9}{1}{$\ev^\dag_{\bTr_\cV(a\otimes b)}$}
}
\underset{\text{\scriptsize (Prop.~\ref{prop:FormulaForEvTr})}}{=}
\tikzmath{
\draw[thick] (-.3,-1) -- (-.3,1);
\draw[thick] (.3,-1) -- (.3,1);
\draw[thick] (1.2,-2) -- (1.2,1);
\draw[thick] (1.8,-2) -- (1.8,1);
\halfDottedEllipse{(1.2,-2)}{.3}{.1}
\draw[thick] (-1.2,2) -- (-1.2,-1);
\draw[thick] (-1.8,2) -- (-1.8,-1);
\draw[thick] (-1.5,2) ellipse (.3 and .1);
  \draw[thick, red] (.1,-1) -- (.1,1);		
	\draw[thick, blue] (-.1,-1) -- (-.1,1);		
	\draw[thick, red] (1.4,-2.08) -- (1.4,1);
  \draw[thick, blue] (1.6,-2.08) -- (1.6,1);
  \draw[thick, blue] (-1.4,1.92) -- (-1.4,-1);		
	\draw[thick, red] (-1.6,1.92) -- (-1.6,-1);
\roundNbox{fill=white}{(.8,1.3)}{.35}{1}{.9}{$\coev^\dag_{\bTr_\cV(a\otimes b)}$}
\roundNbox{fill=white}{(-.8,-1.3)}{.35}{.9}{1}{$\ev^\dag_{\bTr_\cV(a\otimes b)}$}
}
=
\tikzmath{
\halfDottedEllipse{(-.3,-2)}{.3}{.1}
\draw[thick] (-.3,2) -- (-.3,-2);
\draw[thick] (.3,2) -- (.3,-2);
\draw[thick] (0,2) ellipse (.3 and .1);
  \draw[thick, red] (-.1,1.92) -- (-.1,-2.08);		
	\draw[thick, blue] (.1,1.92) -- (.1,-2.08);
}\,\,.
\qedhere
\end{align*}
\end{proof}


\section{Unitary anchored planar algebras}
\label{sec:UAPA}

\subsection{Planar algebras}\label{sec: Planar algebras}

We present here an abridged introduction to planar tangles, the planar operad, and planar algebras, and refer the reader to \cite[\S2.1]{MR4528312} for an extended treatment.

An \emph{planar tangle} consists of a collection of (round) parametrized disc $D_1,\ldots,D_n$ inside some biger parametrised disc $D_0$, along with a collection of non-intersecting paths in $D_0\setminus(D_1\cup\ldots\cup D_k)$ called \emph{strings} that start and end at the boundary of one of the circles, or are themselves closed loops. 
Here, a disc $D_i$ is called \emph{parametrised} if $D_i=\varphi_i(\bbD)$ for some chosen affine linear maps $\varphi_i:\bbD\to D_i$ from the standard unit disc.
The points $\varphi_i(1)\in\partial D_i$ are called \emph{anchor points}, and the strands are required to not start or end on the anchor points.
Here is an example of a planar tangle, where we have marked the anchor points in red:
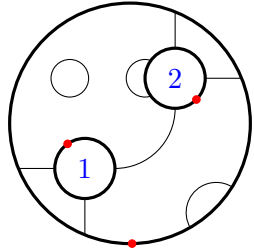
\[
\begin{tikzpicture}[baseline =-.1cm]
\pgftransformyscale{-1}
	\coordinate (a) at (0,0);
	\coordinate (b) at ($ (a) + (1.4,1) $);
	\coordinate (c) at ($ (a) + (.6,-.6) $);
	\coordinate (d) at ($ (a) + (-.6,.6) $);
	\coordinate (e) at ($ (a) + (-.8,-.6) $);
	\ncircle{}{(a)}{1.6}{89}{}
%
%
	\draw (60:1.6cm) arc (150:300:.4cm);
	\draw ($ (c) + (0,.4) $) arc (0:90:.8cm);
	\draw ($ (c) + (-.4,0) $) circle (.25cm);
	\draw ($ (d) + (0,.88) $) -- (d) -- ($ (d) + (-.88,0) $);
	\draw ($ (c) + (0,-.88) $) -- (c) -- ($ (c) + (.88,0) $);
	\draw (e) circle (.25cm);
	\ncircle{unshaded}{(d)}{.4}{235}{}
	\ncircle{unshaded}{(c)}{.4}{225+180}{}
	\node[blue] at (c) {\small 2};
	\node[blue] at (d) {\small 1};
\end{tikzpicture}
\]

There is a composition operation for planar tangles
$$
S\circ_i T := S \cup \phi_i(T)
$$
when the number of external string boundary points of $T$ agrees with the number of internal string boundary points in the $i$-th disk of $S$.
We shrink the tangle $T$, insert it into the $i$-th input disk of $S$ using the map $\phi_i$, and match up the boundary points of the strings.
For an explicit example, see \cite[Ex.~2.2]{MR4528312}. 

The collection of isotopy classes of planar tangles with the operation of tangle composition is called the \emph{planar operad}.
A \emph{planar algebra} is an algebra for this operad.
Unpacking, we have a vector space $\cP[n]$ for each $n\in \bbN_{\geq 0}$,
and each (isotopy class of) planar tangle $T$ gives a linear map 
$$
Z(T): \cP[n_1]\otimes \cdots \otimes \cP[n_k]\longrightarrow \cP[n_0].
$$
Here, $n_i$ for $1\leq i\leq k$ is the number of boundary points of strings on the $i$-th input disk of $T$, and $n_0$ is the number of boundary points of strings on the output disk of $T$.
For future convenience, we abbreviate these properties by saying $T$ has \emph{type} $(n_1,\dots, n_k; n_0)$.

Given planar tangles $S,T$ of types $(m_1,\dots,m_j;m_0), (n_1,\dots, n_k;n_0)$ respectively with $n_0=m_i$, composition of the linear maps $Z(S),Z(T)$ must be compatible with composition of planar tangles:
$$
Z(S\circ_i T)
=
Z(S)\circ (\id_{\cP[m_1]\otimes\ldots\otimes \cP[m_{i-1}]}\otimes Z(T)\otimes\id_{\cP[m_{i+1}]\otimes\ldots\otimes \cP[m_j]}).
$$
We also require that the identity tangle must act as the identity linear map.
Finally, letting $T^\sigma$ be the tangle obtained by renumbering the input disks of $T$ by some permutation $\sigma$,
we should have $Z(T^\sigma)=Z(T)\circ \sigma$ (where we also use $\sigma$ to denote the symmetric braiding of the vector spaces $\cP[n_i]$ corresponding to the permutation $\sigma$).

\begin{rem}
The notion of a planar algebra makes sense in any symmetric tensor category.
\end{rem}

\subsection{Anchored planar algebras}
\label{sec:APA}

We present here an abridged introduction to the anchored planar operad and anchored planar algebras, and we refer the reader to \cite[\S2.2]{MR4528312} for an extended treatment.
An \emph{anchored planar tangle} is a thing like this:
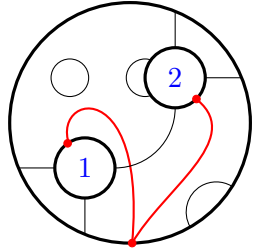
\[
\begin{tikzpicture}[baseline =-.1cm]
\pgftransformyscale{-1}
	\coordinate (a) at (0,0);
	\coordinate (b) at ($ (a) + (1.4,1) $);
	\coordinate (c) at ($ (a) + (.6,-.6) $);
	\coordinate (d) at ($ (a) + (-.6,.6) $);
	\coordinate (e) at ($ (a) + (-.8,-.6) $);
	\ncircle{}{(a)}{1.6}{89}{}
	\draw[thick, red] (d) .. controls ++(225:1.2cm) and ++(275:2.6cm) .. ($ (a) + (89:1.6) $);
	\draw[thick, red] ($ (a) + (89:1.6) $) to[out=-60, in=40] (.9,-.3);
	\draw (60:1.6cm) arc (150:300:.4cm);
	\draw ($ (c) + (0,.4) $) arc (0:90:.8cm);
	\draw ($ (c) + (-.4,0) $) circle (.25cm);
	\draw ($ (d) + (0,.88) $) -- (d) -- ($ (d) + (-.88,0) $);
	\draw ($ (c) + (0,-.88) $) -- (c) -- ($ (c) + (.88,0) $);
	\draw (e) circle (.25cm);
	\ncircle{unshaded}{(d)}{.4}{235}{}
	\ncircle{unshaded}{(c)}{.4}{225+180}{}
	\node[blue] at (c) {\small 2};
	\node[blue] at (d) {\small 1};
\end{tikzpicture}\,.
\]
It differs from a planar tangle in that each internal anchor point is connected to the external anchor point by a red \emph{anchor line} which stays in $D_0\setminus \{D_1,\ldots, D_k\}$.
These anchor lines are transparent to the strings, but they may not intersect each other.

Like before, we say an anchored planar tangle $T$ has \emph{type} $(n_1,\dots, n_k; n_0)$
if $T$ has $k$ input disks, the $i$-th input disk of $T$ meets $n_i$ string boundary points for $i\leq 1\leq k$, and the output disk of $T$ meets $n_0$ string boundary points.

The \emph{anchored planar operad} is the collection of isotopy classes of anchored planar tangles, with the operation of tangle composition.
Here, given anchored planar tangles $S,T$ of type $(m_1,\dots,m_j;m_0), (n_1,\dots, n_k;n_0)$ respectively with $n_0=m_i$, the composition operation $S\circ_i T$ differs from the previous operation in that we must also give new anchor lines.
We do so by replacing the $i$-th anchor line of $T$ by $k$ parallel lines, and composing each one of them with the $i$-th anchor line of $S$, 
as illustrated in the following picture:
\begin{equation}
\label{eq:AnchoredPlanarTangleComposition}
\begin{tikzpicture}[baseline =-.1cm]
\pgftransformyscale{-1}
	\coordinate (a) at (0,0);
	\coordinate (b) at ($ (a) + (1.4,1) $);
	\coordinate (c) at ($ (a) + (.6,-.6) $);
	\coordinate (d) at ($ (a) + (-.6,.6) $);
	\coordinate (e) at ($ (a) + (-.8,-.6) $);
	\ncircle{}{(a)}{1.6}{89}{}
	\draw[thick, red] (d) .. controls ++(225:1.2cm) and ++(275:2.6cm) .. ($ (a) + (89:1.6) $);
	\draw[thick, red] ($ (a) + (89:1.6) $) to[out=-60, in=40] (.9,-.3);
	\draw (60:1.6cm) arc (150:300:.4cm);
	\draw ($ (c) + (0,.4) $) arc (0:90:.8cm);
	\draw ($ (c) + (-.4,0) $) circle (.25cm);
	\draw ($ (d) + (0,.88) $) -- (d) -- ($ (d) + (-.88,0) $);
	\draw ($ (c) + (0,-.88) $) -- (c) -- ($ (c) + (.88,0) $);
	\draw (e) circle (.25cm);
	\ncircle{unshaded}{(d)}{.4}{235}{}
	\ncircle{unshaded}{(c)}{.4}{225+180}{}
	\node[blue] at (c) {\small 2};
	\node[blue] at (d) {\small 1};
\pgftransformyscale{-1}
\foreach \x/\s in {2.1/25, 3.1/22, 3.85/21, 4.55/20, 5.25/21, 5.9/23, 6.5/25, 7/26, 7.5/26, 8/27, 8.6/27}
{\draw[dotted, blue!50] ($(4,0)+(70+14.5*\x:1.615)$) to[bend right=\s] ($(.6,.6) + (130+14.5*\x:.415)$);}
\foreach \x/\s in {0/30,10/28}
{\draw[blue!50, very thin] ($(4,0)+(70+14.5*\x:1.615)$) to[bend right=\s] ($(.6,.6) + (130+14.5*\x:.415)$);}
\draw[blue!50, -stealth] (2.2,.855) to[bend right=7] (1.8,.865);
\pgftransformyscale{-1}
\pgftransformyshift{-1.6}
\pgftransformxshift{17}
\pgftransformscale{4.01}
\pgftransformrotate{46}
	\draw[very thick] (.6,-.6) circle (.4);
	\filldraw[red] (.6,-.6) + (225+178:.4) coordinate(CIRC1) circle (.012); 
	\draw[very thick] (.6,-.75) circle (.1);
	\filldraw[red] (.6,-.75) + (-40:.1) coordinate(CIRC2) circle (.012); 
	\draw[very thick] (.6,-.45) circle (.1);
	\filldraw[red] (.6,-.45) + (-40:.1) coordinate(CIRC3) circle (.012); 
	\draw[thick, red] (CIRC1) to[out=-150, in=-30] ++(180:.1) arc(56:275:.33) to[out=10, in=-30, looseness=1.6] (CIRC2);
	\draw[thick, red] (CIRC1) to[out=-110, in=-30, looseness=1.2] (CIRC3);
	\draw (.6,-1.005) -- (.6,-.85) (.6,-.2+.005) -- (.6,-.35);
	\draw (.68,-.69) to[in=180, out=45] (1.005,-.6);
	\draw (.275,-.839) to[in=180, out=35] (.5,-.75);
	\draw (.275,-.6+0.239) to[in=180, out=-35] (.5,-.45);
	\node[blue] at (.6,-.75) {\small 1};
	\node[blue] at (.6,-.45) {\small 2};
\end{tikzpicture}
=
\begin{tikzpicture}[baseline =-.1cm, yscale = -1]
	\coordinate (a) at (0,0);
	\coordinate (b) at ($ (a) + (1.4,1) $);
	\coordinate (c) at ($ (a) + (.6,-.6) $);
	\coordinate (d) at ($ (a) + (-.6,.6) $);
	\coordinate (e) at ($ (a) + (-.8,-.6) $);
	\ncircle{}{(a)}{1.6}{89}{}
	\draw[thick, red] (d) .. controls ++(225:1.2cm) and ++(275:2.6cm) .. ($ (a) + (89:1.6) $);
	\draw[thick, red] ($ (a) + (89:1.6) $) to[out=-60, in=40] (.9,-.3) coordinate(CIRC1);
	\draw[thick, red] ($ (a) + (89:1.6) $) to[out=-47, in=39] (.96,-.36) coordinate(CIRC1');
	\draw (60:1.6cm) arc (150:300:.4cm);
	\draw ($ (c) + (0,.4) $) arc (0:90:.8cm);
	\draw ($ (c) + (-.4,0) $) circle (.25cm);
	\draw ($ (d) + (0,.88) $) -- (d) -- ($ (d) + (-.88,0) $);
	\draw ($ (c) + (0,-.88) $) -- (c) -- ($ (c) + (.88,0) $);
	\draw (e) circle (.25cm);
	\ncircle{unshaded}{(d)}{.4}{235}{}
	\fill[white] (c) circle (.4);
	\ncircle{unshaded}{(.6,-.75)}{.1}{-45}{}
	\ncircle{unshaded}{(.6,-.45)}{.1}{-45}{}
	\draw (.6,-1.005) -- (.6,-.85) (.6,-.2+.005) -- (.6,-.35);
	\draw (.68,-.69) to[in=180, out=45] (1.005,-.6);
	\draw (.275,-.839) to[in=180, out=25] (.5,-.75);
	\draw (.275,-.6+0.239) to[in=180, out=-25] (.5,-.45);
	\node[blue, xshift=8, yshift=3] at (.6,-.75) {\small 2};
	\node[blue, xshift=-6, yshift=-9] at (.6,-.45) {\small 3};
	\node[blue] at (d) {\small 1};
	\path (.6,-.75) + (-40:.1) coordinate(CIRC2); 
	\path (.6,-.45) + (-40:.1) coordinate(CIRC3); 
	\draw[thick, red] (CIRC1) to[out=-150, in=-30] ++(180:.1) arc(59:275:.35) to[out=10, in=-30, looseness=1.6] (CIRC2);
	\draw[thick, red] (CIRC1') to[out=-140, in=-40] (CIRC3);
\end{tikzpicture}
\end{equation}

Recall from \cite[Appendix A.2]{MR3578212} that a braided pivotal tensor category is also rigid balanced.\footnote{It is in fact so in two ways.}
Hence our pivotal structure $\varphi^\cV$ induces a twist by
\begin{equation}
\label{eq:CanonicalUnitaryTwistFromPivotalStructure}
\theta_{v}
:= 
\begin{tikzpicture}[rotate=180, baseline=.35cm]
	\draw (0,-1.6) -- (0,.8);
  \coordinate (a) at (0,-1);
  \fill[unshaded] ($ (a) - (.3,.3) $) rectangle ($ (a) + (.1,.3) $);
	\draw ($(a) + (-.3,.2) $) arc (90:270:.2cm);
	\draw ($ (a) + (0,-.3) $)  .. controls ++(90:.2cm) and ++(0:.2cm) .. ($ (a) + (-.3,.2) $);
	\draw[super thick, white] ($ (a) + (0,.3) $)  .. controls ++(270:.2cm) and ++(0:.2cm) .. ($ (a) + (-.3,-.2) $);
	\draw ($ (a) + (0,.3) $)  .. controls ++(270:.2cm) and ++(0:.2cm) .. ($ (a) + (-.3,-.2) $);
	\roundNbox{unshaded}{(0,0)}{.35}{0}{0}{$\varphi_v$}	
	\node at (-.2+.4,.6) {\scriptsize{$v$}};
	\node at (-.2+.4,-1.4) {\scriptsize{$v$}};
	\node at (-.3+.6,-.55) {\scriptsize{$v^{\vee\vee}$}};
\end{tikzpicture}
=
(\id_v\otimes \ev_{v^{\vee}})
\circ
(\beta_{v^{\vee\vee},v}\otimes\id_{v^{\vee}})
\circ
(\id_{v^{\vee\vee}}\otimes \coev_v)
\circ
\varphi_v.
\end{equation}

\noindent
(By \cite[Prop.~A.4]{MR3578212}, a braided unitary tensor category $(\cV,\beta,\vee)$ is ribbon if and only if the $\varphi$ induced by $\vee$ is a spherical structure.)

\begin{defn}
\label{def: anchored planar algebra}
Let $(\cV,\beta,\varphi^\cV)$ be a braided pivotal tensor category.
An anchored planar algebra over $\cV$ is an algebra in $\cV$ over the anchored planar operad.
Unpacking, we have a sequence $\cP=(\cP[n])_{n\ge 0}$ of objects of $\cV$, along with operations
$$Z(T):\cP[n_1]\otimes\ldots\otimes \cP[n_k]\to \cP[n_0]$$
for every isotopy class of anchored planar tangle $T$ of type $(n_1,\ldots,n_k;n_0)$,
subject to the following axioms:
\begin{itemize}
\item
(identity) the identity anchored tangle acts as the identity morphism
\item
(composition) if $S$ and $T$ are anchored planar tangles of type $(m_1,\ldots,m_j;m_0)$ and $(n_1,\ldots,n_k;n_0)$, and if $n_0=m_i$, then
\begin{equation}\label{eq: composition of tangles}
Z(S\circ_i T)=Z(S)\circ (\id_{\cP[m_1]\otimes\ldots\otimes \cP[m_{i-1}]}\otimes Z(T)\otimes\id_{\cP[m_{i+1}]\otimes\ldots\otimes \cP[m_j]})
\end{equation}
\item
(anchor dependence) the following relations hold:
\begin{itemize}
\item
(braiding)\hspace{.7cm}
$
Z\left(
\begin{tikzpicture}[baseline = -.1cm]
	\draw (-.6,-.2) -- (-.6,1);
	\draw (0,1) -- (0,.6);
	\draw (.6,-.2) -- (.6,1);
	\draw[thick, red] (0,-.6) -- (0,-1);
	\draw[thick, red] (0,.2) arc (0:-90:.2cm) -- (-.6,0) arc (90:270:.4cm) -- (-.2,-.8) arc (90:0:.2cm);
	\roundNbox{}{(0,0)}{1}{.2}{.2}{}
	\roundNbox{unshaded}{(0,-.4)}{.2}{.6}{.6}{}
	\roundNbox{unshaded}{(0,.4)}{.2}{.2}{.2}{}
	\node at (-.8,.8) {\scriptsize{$i$}};
	\node at (-.2,.8) {\scriptsize{$j$}};
	\node at (.8,.8) {\scriptsize{$k$}};
	\fill[red] (0,.2) circle (.05)  (0,-.6) circle (.05)  (0,-1) circle (.05);
\end{tikzpicture}
\right)
=
Z
\left(
\begin{tikzpicture}[baseline = -.1cm]
	\draw (-.6,-.2) -- (-.6,1);
	\draw (0,1) -- (0,.6);
	\draw (.6,-.2) -- (.6,1);
	\draw[thick, red] (0,-.6) -- (0,-1);
	\draw[thick, red] (0,.2) arc (180:270:.2cm) -- (.6,0) arc (90:-90:.4cm) -- (.2,-.8) arc (90:180:.2cm);
	\roundNbox{}{(0,0)}{1}{.2}{.2}{}
	\roundNbox{unshaded}{(0,-.4)}{.2}{.6}{.6}{}
	\roundNbox{unshaded}{(0,.4)}{.2}{.2}{.2}{}
	\node at (-.8,.8) {\scriptsize{$i$}};
	\node at (-.2,.8) {\scriptsize{$j$}};
	\node at (.8,.8) {\scriptsize{$k$}};
	\fill[red] (0,.2) circle (.05)  (0,-.6) circle (.05)  (0,-1) circle (.05);
\end{tikzpicture}
\right)
\,
\circ \beta_{\cP[j],\cP[i+k]}
$
\item
(twist)\hspace{1.3cm}
$
Z\left(
\begin{tikzpicture}[baseline=-.1cm]
	\ncircle{unshaded}{(0,0)}{1}{270}{}
	\ncircle{unshaded}{(0,0)}{.3}{270}{}
	\draw (90:.3cm) -- (90:1cm);
	\draw[thick, red] (-90:.3cm) .. controls ++(270:.3cm) and ++(270:.5cm) .. (0:.5cm) .. controls ++(90:.8cm) and ++(90:.8cm) .. (180:.7cm) .. controls ++(270:.6cm) and ++(90:.4cm) .. (270:1cm);
	\node at (100:.8cm) {\scriptsize{$n$}};
\end{tikzpicture}
\right)
=\theta_{\cP[n]}
$.
\end{itemize}
\end{itemize}
(Here, a little number $n$ next to a string to indicates $n$ parallel strings.)
We call $\cP[n]$ the $n$\textsuperscript{th} \emph{box object} of the anchored planar algebra $\cP$.
\end{defn}

\begin{nota}
In the sections below, 
we use the notation of an anchored planar tangle $T$ inserted into a coupon to denote the map $Z(T)$ in $\cV$ afforded by an anchored planar algebra.
The strings in these diagrams are usually drawn horizontally for convenience, and we read them \emph{left to right}.
For example:
$$
\tikzmath{
	\coordinate (a) at (0,0);
	\pgfmathsetmacro{\boxWidth}{1};
	\roundNbox{unshaded}{(-2.3,.5)}{.3}{0}{0}{$g$}
	\roundNbox{unshaded}{(-2.3,-.5)}{.3}{0}{0}{$f$}
	\draw[rounded corners=5pt, very thick, unshaded] ($ (a) - (\boxWidth,\boxWidth) $) rectangle ($ (a) + (\boxWidth,\boxWidth) $);
	\draw ($ (a) + 5/6*(0,1) $) -- ($ (a) - 5/6*(0,\boxWidth) $);
	\draw[thick, red] ($ (a) + 1/3*(0,\boxWidth) - 1/5*(\boxWidth,0) $) -- ($ (a) - 2/3*(\boxWidth,0) $);
	\draw[thick, red] ($ (a) - 1/3*(0,\boxWidth) - 1/5*(\boxWidth,0) $) -- ($ (a) - 2/3*(\boxWidth,0) $);
	\draw[very thick] (a) ellipse ({2/3*\boxWidth} and {5/6*\boxWidth});
	\filldraw[very thick, unshaded] ($ (a) + 1/3*(0,\boxWidth) $) circle (1/5*\boxWidth);
	\filldraw[very thick, unshaded] ($ (a) - 1/3*(0,\boxWidth) $) circle (1/5*\boxWidth);
	\node at ($ (a) + (.2,0) $) {\scriptsize{$n$}};
	\node at ($ (a) + (.2,-.65) $) {\scriptsize{$k$}};
	\node at ($ (a) + (.2,.65) $) {\scriptsize{$p$}};
\draw (-2.3,-1.2) --node[left]{$\scriptstyle u$} (-2.3,-.8);
\draw (-2.3,-.2) --node[left]{$\scriptstyle v$} (-2.3,.2);
\draw (-2.3,1.2) --node[left]{$\scriptstyle w$} (-2.3,.8);
\draw (-1,.5) --node[above]{$\scriptstyle \cP[p+n]$} (-2,.5);
\draw (-1,-.5) --node[above]{$\scriptstyle \cP[n+k]$} (-2,-.5);
\draw (1,0) --node[above]{$\scriptstyle \cP[p+k]$} (2,0);
}
=
(\id_w\otimes Z(T))\circ(g\otimes \id_{\cP[n+k]})\circ f
$$
where $T$ is the tangle in the coupon.
\end{nota}

\begin{defn}
A morphism $H: \cP_1\to \cP_2$ of anchored planar algebras is 
a sequence of morphisms $(H[n]:\cP_1[n]\to \cP_2[n])_{n\geq 0}$ such that for every anchored planar tangle $T$ of type $(n_1,\dots, n_k;n_0)$,
$H[n_0]\circ Z(T) = Z(T)\circ (H[n_1]\otimes \cdots \otimes H[n_k])$.
\end{defn}

\subsection{Unitary anchored planar algebras}\label{subsec:UAPA}

Let $(\cV,\vee,\beta)$ be a unitary tensor category equipped with a chosen unitary dual functor $\vee$ and a 
unitary braiding $\beta = \{\beta_{u,v}: u\otimes v \to v\otimes u\}_{u,v\in \cV}$ satisfying the usual axioms.
Using the canonical pivotal structure \eqref{eq:CanonicalUnitaryPivotalStructure} induced by $\vee$, 
the formula \eqref{eq:CanonicalUnitaryTwistFromPivotalStructure} for the twist simplifies to
\[
\theta_{v}
= 
\begin{tikzpicture}[rotate=180, baseline=.35cm]
	\draw (0,-1.6) -- (0,.8);
  \coordinate (a) at (0,-1);
  \fill[unshaded] ($ (a) - (.3,.3) $) rectangle ($ (a) + (.1,.3) $);
	\draw ($(a) + (-.3,.2) $) arc (90:270:.2cm);
	\draw ($ (a) + (0,-.3) $)  .. controls ++(90:.2cm) and ++(0:.2cm) .. ($ (a) + (-.3,.2) $);
	\draw[super thick, white] ($ (a) + (0,.3) $)  .. controls ++(270:.2cm) and ++(0:.2cm) .. ($ (a) + (-.3,-.2) $);
	\draw ($ (a) + (0,.3) $)  .. controls ++(270:.2cm) and ++(0:.2cm) .. ($ (a) + (-.3,-.2) $);
	\roundNbox{unshaded}{(0,0)}{.35}{0}{0}{$\varphi_v$}	
	\node at (-.2+.4,.6) {\scriptsize{$v$}};
	\node at (-.2+.4,-1.4) {\scriptsize{$v$}};
	\node at (-.3+.6,-.55) {\scriptsize{$v^{\vee\vee}$}};
\end{tikzpicture}
\underset{\eqref{eq:CanonicalUnitaryPivotalStructure}}{=}
(\id_v \otimes \coev_v^\dag)
\circ
(\beta_{v,v}\otimes \id_{v^\vee})
\circ
(\id_v \otimes \coev_v).
\]



\begin{defn}\label{Def : unitary anchored planar algebra}
A \emph{unitary anchored planar algebra} in $\cV$ is a triple $(\cP,r,\psi_\cP)$ consisting of
an anchored planar algebra $\cP$ in $\cV$,
a $\dag$-\emph{structure} $r$, which is a family of real structures $r_n : \cP[n] \rightharpoonup \cP[n]$ for each $n\geq 0$,
and a morphism $\psi_\cP: \cP[0] \to 1_\cV$ called a \emph{faithful state}
satisfying 
$
\tikzmath{
\draw (.3,0) --node[above]{$\scriptstyle \cP[0]$} (1,0);
\draw[dotted] (1.6,0) --node[above]{$\scriptstyle 1_\cV$} (2.1,0);
\draw[dotted] (-.8,0) --node[above]{$\scriptstyle 1_\cV$} (-.3,0);
\roundNbox{unshaded}{(0,0)}{.3}{0}{0}{}
\draw[very thick] (0,0) circle (.2cm);
\filldraw[red] (-.2,0) circle (.05cm);
\roundNbox{unshaded}{(1.3,0)}{.3}{0}{0}{$\psi_\cP$}
}
=
\id_{1_\cV}
$ 
such that:
\begin{enumerate}[label=(P\arabic*)]
\item
\label{P:Reflection}
For every anchored planar tangle $T$ of type $(n_1,\dots, n_k; n_0)$,
$$ 
\overline{Z(T)}
\circ\nu^{(k)}\circ 
(r_{n_k}\otimes \cdots \otimes r_{n_1})
=
r_{n_0}\circ 
Z(\overline{T}),
$$
where $\overline{T}$ is the reflection of $T$, and
$\nu^{(k)}:
(\overline{\,\cdot\,}\otimes\ldots\otimes\overline{\,\cdot\,})\to \overline{(\,\cdot\,\otimes\cdots\otimes\,\cdot\,)}
$ is an appropriate composite of $\nu$'s.
When $k=0$, we have $\overline{Z(T)}\circ r_\cV = r_{k_0}\circ Z(\overline{T})$ where $r_\cV: 1_\cV \to \overline{1_\cV}$ is the real structure of $1_\cV$.
(Suppressing the $\dag$-structure $r$, this axiom reads $Z(\overline{T}) = \overline{Z(T)}$.)
\item
\label{P:PositiveDefinite}
For every $n\geq 0$, we have an equality of pairings
\begin{equation}
\label{eq:PlanarPairing}
\tikzmath{
	\coordinate (a) at (0,0);
	\pgfmathsetmacro{\boxWidth}{1};
	\draw (-3.2,-.5) -- (-1,-.5);
	\draw (-3.2,.5) -- (-1,.5);
	\draw (1,0) -- (2,0);
	\draw[dotted] (2,0) -- (3,0);
	\roundNbox{unshaded}{(-2.1,-.5)}{.3}{.1}{.1}{$r_n^{-1}$}
	\roundNbox{unshaded}{(2,0)}{.3}{0}{0}{$\psi_\cP$}
	\draw[rounded corners=5pt, very thick, unshaded] ($ (a) - (\boxWidth,\boxWidth) $) rectangle ($ (a) + (\boxWidth,\boxWidth) $);
	\draw ($ (a) + 1/3*(0,1) $) -- ($ (a) - 1/3*(0,\boxWidth) $);
	\draw[thick, red] ($ (a) + 1/3*(0,\boxWidth) - 1/5*(\boxWidth,0) $) -- ($ (a) - 2/3*(\boxWidth,0) $);
	\draw[thick, red] ($ (a) - 1/3*(0,\boxWidth) - 1/5*(\boxWidth,0) $) -- ($ (a) - 2/3*(\boxWidth,0) $);
	\draw[very thick] (a) ellipse ({2/3*\boxWidth} and {5/6*\boxWidth});
	\filldraw[very thick, unshaded] ($ (a) + 1/3*(0,\boxWidth) $) circle (1/5*\boxWidth);
	\filldraw[very thick, unshaded] ($ (a) - 1/3*(0,\boxWidth) $) circle (1/5*\boxWidth);
	\node at ($ (a) + (.2,0) $) {\scriptsize{$n$}};
	\node at (-2.85,-.3) {\scriptsize{$\overline{\cP[n]}$}};
	\node at (-1.35,-.3) {\scriptsize{$\cP[n]$}};
	\node at (-2,.7) {\scriptsize{$\cP[n]$}};
	\node at (1.35,.2) {\scriptsize{$\cP[0]$}};
	\node at (2.65,.2) {\scriptsize{$1_\cV$}};
}
=
\coev^\dag_{\cP[n]}.
\end{equation}
\end{enumerate}
If $(\cV,\vee,\beta)$ is moreover ribbon (so that $\vee$ is spherical), 
we say that a unitary anchored planar algebra $(\cP,r,\psi_\cP)$ is \emph{spherical} if 
$$
\tikzmath{
	\coordinate (a) at (0,0);
	\pgfmathsetmacro{\boxWidth}{1};
	\draw (-2,0) -- (-1,0);
	\draw (1,0) -- (2,0);
	\draw[dotted] (2,0) -- (3,0);
	\roundNbox{unshaded}{(2,0)}{.3}{0}{0}{$\psi_\cP$}
	\draw[rounded corners=5pt, very thick, unshaded] ($ (a) - (\boxWidth,\boxWidth) $) rectangle ($ (a) + (\boxWidth,\boxWidth) $);
	\draw[thick, red] (180:1/5*\boxWidth) -- (180:4/5*\boxWidth);
	\draw[very thick] (a) circle (4/5*\boxWidth);
	\draw[very thick] (a) circle (1/5*\boxWidth);
	\draw (90:1/5*\boxWidth) .. controls ++(90:.3cm) and ++(90:.5cm) .. (180:1/2*\boxWidth) .. controls ++(270:.5cm) and ++(270:.3cm) ..  (270:1/5*\boxWidth);
	\node at (1.35,.2) {\scriptsize{$\cP[0]$}};
	\node at (2.65,.2) {\scriptsize{$1_\cV$}};
	\node at (-1.5,.2) {\scriptsize{$\cP[2]$}};
}
=
\tikzmath{
	\coordinate (a) at (0,0);
	\pgfmathsetmacro{\boxWidth}{1};
	\draw (-2,0) -- (-1,0);
	\draw (1,0) -- (2,0);
	\draw[dotted] (2,0) -- (3,0);
	\roundNbox{unshaded}{(2,0)}{.3}{0}{0}{$\psi_\cP$}
	\draw[rounded corners=5pt, very thick, unshaded] ($ (a) - (\boxWidth,\boxWidth) $) rectangle ($ (a) + (\boxWidth,\boxWidth) $);
	\draw[thick, red] (180:1/5*\boxWidth) -- (180:4/5*\boxWidth);
	\draw[very thick] (a) circle (4/5*\boxWidth);
	\draw[very thick] (a) circle (1/5*\boxWidth);
	\draw (90:1/5*\boxWidth) .. controls ++(90:.3cm) and ++(90:.5cm) .. (0:1/2*\boxWidth) .. controls ++(270:.5cm) and ++(270:.3cm) ..  (270:1/5*\boxWidth);
	\node at (1.35,.2) {\scriptsize{$\cP[0]$}};
	\node at (2.65,.2) {\scriptsize{$1_\cV$}};
	\node at (-1.5,.2) {\scriptsize{$\cP[2]$}};
}\,.
$$
In this paper, we do not always assume our anchored planar algebras to be spherical.
\end{defn}

\begin{defn}
A morphism $H: (\cP_1,r^1,\psi_1) \to (\cP_2,r^2,\psi_2)$ of unitary anchored planar algebras is a morphism $H$ of anchored planar algebras satisfying
$\overline{H[n]}\circ r^1_n = r^2_n\circ H[n]$ for all $n\geq 0$,
and
$\psi_1 = \psi_2\circ H[0]$.
\end{defn}


When $\cV=\fdHilb$, the definition of unitary anchored planar algebra reduces to the definition of a unitary planar algebra (also known as a semisimple $\rm C^*$ planar algebra) from \cite{MR4598730}, but equipped with a faithful state. (The planar algebra of a bipartite graph \cite{MR1865703} is an example of a $\rm C^*$ planar algebra that is naturally equipped wiht a faithful state.)

The following lemma is straightfoward and left to the reader.

\begin{lem}
If a unitary anchored planar algebra is spherical, then for every $n\geq 0$ we have
\[
\tikzmath{
	\coordinate (a) at (0,0);
	\pgfmathsetmacro{\boxWidth}{1};
	\draw (-2,0) -- (-1,0);
	\draw (1,0) -- (2,0);
	\draw[dotted] (2,0) -- (3,0);
	\roundNbox{unshaded}{(2,0)}{.3}{0}{0}{$\psi_\cP$}
	\draw[rounded corners=5pt, very thick, unshaded] ($ (a) - (\boxWidth,\boxWidth) $) rectangle ($ (a) + (\boxWidth,\boxWidth) $);
	\draw[thick, red] (180:1/5*\boxWidth) -- (180:4/5*\boxWidth);
	\draw[very thick] (a) circle (4/5*\boxWidth);
	\draw[very thick] (a) circle (1/5*\boxWidth);
	\draw (90:1/5*\boxWidth) .. controls ++(90:.3cm) and ++(90:.5cm) .. (180:1/2*\boxWidth) .. controls ++(270:.5cm) and ++(270:.3cm) ..  (270:1/5*\boxWidth);
	\node at (1.35,.2) {\scriptsize{$\cP[0]$}};
	\node at (2.65,.2) {\scriptsize{$1_\cV$}};
	\node at (-1.5,.2) {\scriptsize{$\cP[2n]$}};
  \node at (-.6,-.2) {\scriptsize{$n$}};
}
=
\tikzmath{
	\coordinate (a) at (0,0);
	\pgfmathsetmacro{\boxWidth}{1};
	\draw (-2,0) -- (-1,0);
	\draw (1,0) -- (2,0);
	\draw[dotted] (2,0) -- (3,0);
	\roundNbox{unshaded}{(2,0)}{.3}{0}{0}{$\psi_\cP$}
	\draw[rounded corners=5pt, very thick, unshaded] ($ (a) - (\boxWidth,\boxWidth) $) rectangle ($ (a) + (\boxWidth,\boxWidth) $);
	\draw[thick, red] (180:1/5*\boxWidth) -- (180:4/5*\boxWidth);
	\draw[very thick] (a) circle (4/5*\boxWidth);
	\draw[very thick] (a) circle (1/5*\boxWidth);
	\draw (90:1/5*\boxWidth) .. controls ++(90:.3cm) and ++(90:.5cm) .. (0:1/2*\boxWidth) .. controls ++(270:.5cm) and ++(270:.3cm) ..  (270:1/5*\boxWidth);
	\node at (1.35,.2) {\scriptsize{$\cP[0]$}};
	\node at (2.65,.2) {\scriptsize{$1_\cV$}};
	\node at (-1.5,.2) {\scriptsize{$\cP[2n]$}};
  \node at (.65,0) {\scriptsize{$n$}};
}\,.
\]
\qed
\end{lem}

\begin{prop}\label{prop: inside out reflection}
Let $(\cP,r,\psi_\cP)$ be a spherical unitary anchored planar algebra.
Let $A$ be an annular anchored planar tangle,
and denote the inside-out reflection by $A^\dag$.
Then $Z(A)^\dag=Z(A^\dag)$. 
\end{prop}

\begin{cor}
$Z\left(
\tikzmath{
	\draw [thick, red] (0,0) -- (180:1cm);
	\draw[] (70:.3cm) .. controls ++(70:.3cm) and ++(90:.5cm) .. (0:.5cm) .. controls ++(270:.8cm) and ++(270:.8cm) .. (180:.7cm) .. controls ++(90:.6cm) and ++(270:.4cm) .. (110:1cm);
	\draw (110:.3cm) .. controls ++(110:.3cm) and ++(270:.4cm) .. (70:1cm);
	\draw[very thick] (0,0) circle (1cm);
	\draw[unshaded, very thick] (0,0) circle (.3cm);
	\node at (-80:.8cm) {$\scriptstyle j$};
	\node at (60:.8cm) {$\scriptstyle i$};
}
\right)$
 is unitary.
 \end{cor}
\begin{proof}[Proof of Proposition \ref{prop: inside out reflection}]
Suppose $A$ has $m$ input strings and $n$ output strings.
Since $\psi_\cP$ is spherical, the morphism
\begin{align*}
\tikzmath{
	\coordinate (a) at (0,0);
	\pgfmathsetmacro{\boxWidth}{1};
	\draw (-5,-.5) -- (-1,-.5);
	\draw (-5,.5) -- (-1,.5);
	\draw (1,0) -- (2,0);
	\draw[dotted] (2,0) -- (3,0);
	\roundNbox{unshaded}{(-2,-.5)}{.3}{0}{0}{$\scriptstyle r_n^{-1}$}
	\roundNbox{unshaded}{(-3.6,-.5)}{.3}{.3}{.3}{$\overline{Z(A)}$}
	\roundNbox{unshaded}{(2,0)}{.3}{0}{0}{$\psi_\cP$}
	\draw[rounded corners=5pt, very thick, unshaded] ($ (a) - (\boxWidth,\boxWidth) $) rectangle ($ (a) + (\boxWidth,\boxWidth) $);
	\draw ($ (a) + 1/3*(0,1) $) -- ($ (a) - 1/3*(0,\boxWidth) $);
	\draw[thick, red] ($ (a) + 1/3*(0,\boxWidth) - 1/5*(\boxWidth,0) $) -- ($ (a) - 2/3*(\boxWidth,0) $);
	\draw[thick, red] ($ (a) - 1/3*(0,\boxWidth) - 1/5*(\boxWidth,0) $) -- ($ (a) - 2/3*(\boxWidth,0) $);
	\draw[very thick] (a) ellipse ({2/3*\boxWidth} and {5/6*\boxWidth});
	\filldraw[very thick, unshaded] ($ (a) + 1/3*(0,\boxWidth) $) circle (1/5*\boxWidth);
	\filldraw[very thick, unshaded] ($ (a) - 1/3*(0,\boxWidth) $) circle (1/5*\boxWidth);
	\node at ($ (a) + (.2,0) $) {\scriptsize{$n$}};
	\node at (-2.65,-.3) {\scriptsize{$\overline{\cP[n]}$}};
	\node at (-4.65,-.3) {\scriptsize{$\overline{\cP[m]}$}};
	\node at (-1.35,-.3) {\scriptsize{$\cP[n]$}};
	\node at (-2,.7) {\scriptsize{$\cP[n]$}};
	\node at (1.35,.2) {\scriptsize{$\cP[0]$}};
	\node at (2.65,.2) {\scriptsize{$1_\cV$}};
}
&=
\tikzmath{
	\coordinate (a) at (0,0);
	\pgfmathsetmacro{\boxWidth}{1};
	\draw (-4.8,-.5) -- (-1,-.5);
	\draw (-4.8,.5) -- (-1,.5);
	\draw (1,0) -- (2,0);
	\draw[dotted] (2,0) -- (3,0);
	\roundNbox{unshaded}{(-3.75,-.5)}{.3}{0}{0}{$\scriptstyle r_m^{-1}$}
	\roundNbox{unshaded}{(-2.2,-.5)}{.3}{.2}{.2}{$Z(\overline{A})$}
	\roundNbox{unshaded}{(2,0)}{.3}{0}{0}{$\psi_\cP$}
	\draw[rounded corners=5pt, very thick, unshaded] ($ (a) - (\boxWidth,\boxWidth) $) rectangle ($ (a) + (\boxWidth,\boxWidth) $);
	\draw ($ (a) + 1/3*(0,1) $) -- ($ (a) - 1/3*(0,\boxWidth) $);
	\draw[thick, red] ($ (a) + 1/3*(0,\boxWidth) - 1/5*(\boxWidth,0) $) -- ($ (a) - 2/3*(\boxWidth,0) $);
	\draw[thick, red] ($ (a) - 1/3*(0,\boxWidth) - 1/5*(\boxWidth,0) $) -- ($ (a) - 2/3*(\boxWidth,0) $);
	\draw[very thick] (a) ellipse ({2/3*\boxWidth} and {5/6*\boxWidth});
	\filldraw[very thick, unshaded] ($ (a) + 1/3*(0,\boxWidth) $) circle (1/5*\boxWidth);
	\filldraw[very thick, unshaded] ($ (a) - 1/3*(0,\boxWidth) $) circle (1/5*\boxWidth);
	\node at ($ (a) + (.2,0) $) {\scriptsize{$n$}};
	\node at (-3.1,-.3) {\scriptsize{$\cP[m]$}};
	\node at (-4.45,-.3) {\scriptsize{$\overline{\cP[m]}$}};
	\node at (-1.35,-.3) {\scriptsize{$\cP[n]$}};
	\node at (-2,.7) {\scriptsize{$\cP[n]$}};
	\node at (1.35,.2) {\scriptsize{$\cP[0]$}};
	\node at (2.65,.2) {\scriptsize{$1_\cV$}};
}
\intertext{is equal to}
\\&=
\tikzmath{
	\coordinate (a) at (0,0);
	\pgfmathsetmacro{\boxWidth}{1};
	\draw (-3.6,-.5) -- (-1,-.5);
	\draw (-3.6,.5) -- (-1,.5);
	\draw (1,0) -- (2,0);
	\draw[dotted] (2,0) -- (3,0);
	\roundNbox{unshaded}{(-2,-.5)}{.3}{0}{0}{$\scriptstyle r_m^{-1}$}
	\roundNbox{unshaded}{(-2.3,.5)}{.3}{.3}{.3}{$Z(A^\dag)$}
	\roundNbox{unshaded}{(2,0)}{.3}{0}{0}{$\psi_\cP$}
	\draw[rounded corners=5pt, very thick, unshaded] ($ (a) - (\boxWidth,\boxWidth) $) rectangle ($ (a) + (\boxWidth,\boxWidth) $);
	\draw ($ (a) + 1/3*(0,1) $) -- ($ (a) - 1/3*(0,\boxWidth) $);
	\draw[thick, red] ($ (a) + 1/3*(0,\boxWidth) - 1/5*(\boxWidth,0) $) -- ($ (a) - 2/3*(\boxWidth,0) $);
	\draw[thick, red] ($ (a) - 1/3*(0,\boxWidth) - 1/5*(\boxWidth,0) $) -- ($ (a) - 2/3*(\boxWidth,0) $);
	\draw[very thick] (a) ellipse ({2/3*\boxWidth} and {5/6*\boxWidth});
	\filldraw[very thick, unshaded] ($ (a) + 1/3*(0,\boxWidth) $) circle (1/5*\boxWidth);
	\filldraw[very thick, unshaded] ($ (a) - 1/3*(0,\boxWidth) $) circle (1/5*\boxWidth);
	\node at ($ (a) + (.2,0) $) {\scriptsize{$m$}};
	\node at (-2.9,-.3) {\scriptsize{$\overline{\cP[m]}$}};
	\node at (-1.35,-.3) {\scriptsize{$\cP[m]$}};
	\node at (-3.3,.7) {\scriptsize{$\cP[n]$}};
	\node at (-1.35,.7) {\scriptsize{$\cP[m]$}};
	\node at (1.35,.2) {\scriptsize{$\cP[0]$}};
	\node at (2.65,.2) {\scriptsize{$1_\cV$}};
}\,.
\end{align*}
Indeed, this can be checked directly using sphericality of $\psi_\cP$ by writing $A$ as a composite of generating annular anchored planar tangles (see \S\ref{sec:ModTensToAPA} below), and `pulling them over' one at a time.
By \eqref{eq:PlanarPairing}, the above equality simplifies (after a 90$^\circ$ rotation of the string diagrams) to
\[
\tikzmath[xscale=-1]{
\draw (0,.3) arc (180:0:.5cm) -- node[left]{$\scriptstyle \cP[n]$} (1,-.7);
\draw (0,-.3) -- node[right]{$\scriptstyle \cP[m]$} (0,-.7);
\roundNbox{fill=white}{(0,0)}{.3}{.3}{.3}{$\overline{Z(A)}$}
}
=
\tikzmath[xscale=-1]{
\draw (0,.3) arc (0:180:.5cm) -- node[right]{$\scriptstyle \cP[m]$} (-1,-.7);
\draw (0,-.3) -- node[left]{$\scriptstyle \cP[n]$} (0,-.7);
\roundNbox{fill=white}{(0,0)}{.3}{.3}{.3}{$Z(A^\dag)$}
}\,.
\]
Precomposing with $\id\otimes \ev_{\cP[m]}^\dag$, we get the desired equality
\[
Z(A)^\dag
=
\overline{Z(A)}^\vee
=
\tikzmath[xscale=-1]{
\draw (0,.3) arc (180:0:.5cm) -- node[left]{$\scriptstyle \cP[n]$} (1,-.7);
\draw (0,-.3) arc (0:-180:.5cm) -- node[right]{$\scriptstyle \cP[m]$} (-1,.7);
\roundNbox{fill=white}{(0,0)}{.3}{.3}{.3}{$\overline{Z(A)}$}
}
=
\tikzmath[xscale=-1]{
\draw (0,.3) arc (0:180:.5cm) arc (0:-180:.5cm) -- node[right]{$\scriptstyle \cP[m]$} (-2,1);
\draw (0,-.3) -- node[right]{$\scriptstyle \cP[n]$} (0,-.7);
\roundNbox{fill=white}{(0,0)}{.3}{.3}{.3}{$Z(A^\dag)$}
}
=
Z(A^\dag).
\qedhere
\]
\end{proof}



\section{Extending the equivalence}
\label{sec:Equivalence}

We now prove Theorem \ref{thm:MainUAPA}, extending the equivalence from \cite[Thm.~A]{MR4528312} to an equivalence
\[
\left\{\,\parbox{4.5cm}{\rm Unitary anchored planar algebras in $\cV$}\left\}
\,\,\,\,\cong\,\,
\left\{\,\parbox{6.5cm}{\rm Pointed unitary module multitensor categories over $\cV$}\,\right\}\right.\right.\!\!.
\]
Recall that, by definition (see \S\ref{sec:Unitary module tensor categories}), a unitary module multitensor category comes with a state $\psi_\cC$ and requires that the $\cV$-action $\Phi^{\scriptscriptstyle Z}:\cV\to Z^{\dag}(\cC)$ is pivotal.

\subsection{From unitary module tensor categories to unitary anchored planar algebras}
\label{sec:ModTensToAPA}

Let $(\cC,\Phi^{\scriptscriptstyle \cZ},\vee_\cC,\psi_\cC)$
be a pointed unitary module multitensor category over $\cV$, with real generator $(x,r_x)\in \cC$, 
and $\bTr_\cV:\cC\to \cV$ the unitary right adjoint of $\Phi= \Forget_\cZ\circ \Phi^{\scriptscriptstyle \cZ}:\cV\to \cC$.
We may then construct an ordinary anchored planar algebra, as in \cite[Thm.~5.1]{MR4528312}.
This anchored planar algebra is completely determined by the following information:
\begin{itemize}
\item
$\cP[n]:=\bTr_\cV(x^{\otimes n})$ 
\item 
$Z\left( 
\tikzmath{
\draw[very thick] (1,0) circle (.25cm);
\filldraw[red] (.75,0) circle (.05cm);
}\,
\right)
:=
\tikzmath{
	\coordinate (a1) at (0,0);
	\coordinate (b1) at (0,.4);
	\draw[thick] (a1) -- (b1);
	\draw[thick] ($ (a1) + (.6,0) $) -- ($ (b1) + (.6,0) $);
	\draw[thick] ($ (b1) + (.3,0) $) ellipse (.3 and .1);
	\draw[thick] (a1) arc (-180:0:.3cm);
}
=
i:1_\cV \to \cP[0]
$
\item
$Z\left( 
\tikzmath{
	\draw[thick, red] (-90:1cm) -- (-90:.25cm);
	\draw[very thick] (0,0) circle (1cm);
	\draw[very thick] (0,0) circle (.25cm);
	\draw (115:.25cm) .. controls ++(115:.5cm) and ++(65:.5cm) .. (65:.25cm);
	\draw (155:.25cm) -- (155:1cm);
	\draw (25:.25cm) -- (25:1cm);
\node at (-.63,.1) {$\scriptstyle i$};
\node at (.63,.08) {$\scriptstyle n-i$};
}\right)
:=
\tikzmath{
	\coordinate (a1) at (0,.1);
	\coordinate (b1) at (0,1.6);
	\coordinate (c) at (.35,.6);
	\draw[thick] (a1) -- (b1);
	\draw[thick] ($ (a1) + (1,0) $) -- ($ (b1) + (1,0) $);
	\halfDottedEllipse{(a1)}{.5}{.2}
	\draw[thick] ($ (b1) + (.5,0) $) ellipse (.5cm and .2cm);

	\draw[thick, orange] ($ (a1) + (.15,-.13) $) -- ($ (b1) + (.15,-.13) $);	
	\draw[thick, blue] ($ (a1) + (.35,-.18) $) -- (c) arc (180:0:.15cm) -- ($ (a1) + (.65,-.18) $);	
	\draw[thick, DarkGreen] ($ (a1) + (.85,-.13) $) -- ($ (b1) + (.85,-.13) $);	
\node[orange] at (-.3,.3) {$\scriptstyle x^{\otimes i}$};
\node[blue, scale=.97] at (.5,1) {$\scriptstyle x$};
\node[DarkGreen] at (1.55,.3) {$\scriptstyle x^{\otimes n-i}$};
}
$
\item
$Z\left( 
\tikzmath{
	\draw[thick, red] (-90:1cm) -- (-90:.25cm);
	\draw[very thick] (0,0) circle (1cm);
	\draw[very thick] (0,0) circle (.25cm);
	\draw (115:1cm) .. controls ++(-65:.4cm) and ++(245:.4cm) .. (65:1cm);
	\draw (155:.25cm) -- (155:1cm);
	\draw (25:.25cm) -- (25:1cm);
\node at (-.63,.1) {$\scriptstyle i$};
\node at (.63,.08) {$\scriptstyle n-i$};
}
\right)
:=
\tikzmath{
	\coordinate (a1) at (0,.1);
	\coordinate (b1) at (0,1.6);
	\coordinate (c) at (.35,.9);
	
	\draw[thick] (a1) -- (b1);
	\draw[thick] ($ (a1) + (1,0) $) -- ($ (b1) + (1,0) $);
	\halfDottedEllipse{(a1)}{.5}{.2}
	\draw[thick] ($ (b1) + (.5,0) $) ellipse (.5cm and .2cm);

	\draw[thick, orange] ($ (a1) + (.15,-.13) $) -- ($ (b1) + (.15,-.13) $);	
	\draw[thick, blue] ($ (b1) + (.35,-.18) $) -- (c) arc (-180:0:.15cm) -- ($ (b1) + (.65,-.18) $);	
	\draw[thick, DarkGreen] ($ (a1) + (.85,-.13) $) -- ($ (b1) + (.85,-.13) $);	
\node[orange] at (-.3,.3) {$\scriptstyle x^{\otimes i}$};
\node[blue, scale=.97] at (.5,.5) {$\scriptstyle x$};
\node[DarkGreen] at (1.55,.3) {$\scriptstyle x^{\otimes n-i}$};
}
$
\item
$
Z\left(
\tikzmath{
	\pgfmathsetmacro{\coord}{(2.5,0)};
	\pgfmathsetmacro{\siz}{1.2};
	\draw ($ (2.5,0) + 14/23*(0,\siz) - 1/3*(\siz,0) $) -- ($ (2.5,0) - 1/3*(\siz,0) $);
	\draw ($ (2.5,0) + 14/23*(0,\siz) + 1/3*(\siz,0) $) -- ($ (2.5,0) + 1/3*(\siz,0) $);
	\draw[thick, red] ($ (2.5,0) - 1/3*(\siz,0) - 1/5*(\siz,0) $) -- ($ (2.5,0) - 5/6*(\siz,0) $);
	\draw[thick, red] ($ (2.5,0) + 1/3*(\siz,0) - 1/5*(\siz,0) $) .. controls ++(180:.2cm) and ++(0:.2cm) .. ($ (2.5,0) - 1/3*(\siz,0) - 2/5*(0,\siz) $) .. controls ++(180:.2cm) and ++(0:.2cm) .. ($ (2.5,0) - 5/6*(\siz,0) $);
	\draw[very thick] (2.5,0) ellipse ( {5/6*\siz} and {2/3*\siz});
	\filldraw[very thick, unshaded] ($ (2.5,0) + 1/3*(\siz,0) $) circle (1/5*\siz);
	\filldraw[very thick, unshaded] ($ (2.5,0) - 1/3*(\siz,0) $) circle (1/5*\siz);
	\node at ($ (2.5,0) + (.2,0) - 1/3*(\siz,0) + 1/3*(0,\siz) $) {$\scriptstyle i$};
	\node at ($ (2.5,0) + (.2,0) + 1/3*(\siz,0) + 1/3*(0,\siz) $) {$\scriptstyle j$};
}
\right)
=
\tikzmath{
	\pgfmathsetmacro{\voffset}{.08};
	\pgfmathsetmacro{\hoffset}{.15};
	\pgfmathsetmacro{\hoffsetTop}{.12};
	\coordinate (b2) at (0,0);
	\coordinate (b3) at ($ (b2) + (1.4,0) $);
	\coordinate (c1) at ($ (b1) + (.7,1.5)$);
	\coordinate (c2) at ($ (b2) + (.7,1.5)$);
	\topPairOfPants{(b2)}{}
%
	\draw[thick, red] ($ (b2) + 2*(\hoffset,0) + (0,-\voffset)$) .. controls ++(90:.8cm) and ++(270:.8cm) .. ($ (c2) + (\hoffset,0) + (0,-\voffset) $);	
	\draw[thick, blue] ($ (b3) + 2*(\hoffset,0) + (0,-.1)$) .. controls ++(90:.8cm) and ++(270:.8cm) .. ($ (c2) + 3*(\hoffset,0) + (0,-\voffset) $);	
\node[red] at (.3,-.3) {$\scriptstyle x^{\otimes i}$};
\node[blue] at (1.7,-.3) {$\scriptstyle x^{\otimes j}$};
}
$
\item
$
Z\left(
\tikzmath{
	\draw [thick, red] (0,0) -- (180:1cm);
	\draw[] (70:.3cm) .. controls ++(70:.3cm) and ++(90:.5cm) .. (0:.5cm) .. controls ++(270:.8cm) and ++(270:.8cm) .. (180:.7cm) .. controls ++(90:.6cm) and ++(270:.4cm) .. (110:1cm);
	\draw (110:.3cm) .. controls ++(110:.3cm) and ++(270:.4cm) .. (70:1cm);
	\draw[very thick] (0,0) circle (1cm);
	\draw[unshaded, very thick] (0,0) circle (.3cm);
	\node at (-80:.8cm) {$\scriptstyle j$};
	\node at (60:.8cm) {$\scriptstyle i$};
}
\right)
=
\begin{tikzpicture}[baseline=-.1cm]
	\draw[thick] (-.3,-1) -- (-.3,1);
	\draw[thick] (.3,-1) -- (.3,1);
	\draw[thick] (0,1) ellipse (.3cm and .1cm);
	\halfDottedEllipse{(-.3,-1)}{.3}{.1}
	
	\draw[thick, red] (-.1,-1.1) .. controls ++(90:.8cm) and ++(270:.8cm) .. (.1,.9);		
	\draw[thick, blue] (.1,-1.1) .. controls ++(90:.2cm) and ++(225:.2cm) .. (.3,-.2);		
	\draw[thick, blue] (-.1,.9) .. controls ++(270:.2cm) and ++(45:.2cm) .. (-.3,.2);
	\draw[thick, blue, dotted] (-.3,.2) -- (.3,-.2);	
	\node[red] at (-.3,-1.28) {$\scriptstyle x^{\otimes i}$};
	\node[blue] at (.3,-1.3) {$\scriptstyle x^{\otimes j}$};
\end{tikzpicture}
$
\end{itemize}

\subsubsection{Adding the dagger structure}

It remains to construct the $\dag$-structure $r$ and the state $\psi_\cP$ on $\cP$, satisfying the conditions listed in Definition~\ref{Def : unitary anchored planar algebra}.

\begin{defn}
For $n\geq 0$, we define $r_n : \cP[n]\to \overline{\cP[n]}$ to be the composite
$$
\bTr_\cV(x^{\otimes n})
\xrightarrow{\bTr_\cV(r_x^{\otimes n})}
\bTr_\cV(\overline{x}^{\otimes n})
\xrightarrow{\bTr_\cV(\nu)}
\bTr_\cV(\overline{x^{\otimes n}})
\xrightarrow{\chi^{\bTr_\cV}_{x^{\otimes n}}}
\overline{\bTr_\cV(x^{\otimes n})}.
$$
\end{defn}

\begin{lem}
$\overline{r_n}\circ r_n = \varphi_{\cP[n]}$.
\end{lem}
\begin{proof}
Going right and then down in the following commutative diagram is $\overline{r_n}\circ r_n$
$$
\begin{tikzcd}[column sep = 4em]
&
\bTr_\cV(\overline{x}^{\otimes n})
\arrow[d,"\bTr_\cV(\overline{r_x}^{\otimes n})"]
\arrow[r,"\bTr_\cV(\nu)"]
&
\bTr_\cV(\overline{x^{\otimes n}})
\arrow[r,"\chi^{\bTr_\cV}_{x^{\otimes n}}"]
\arrow[d,"\bTr_\cV(\overline{r_x^{\otimes n}})"]
&
\overline{\bTr_\cV(x^{\otimes n})}
\arrow[d,"\overline{\bTr_\cV(r_x^{\otimes n})}"]
\\
\bTr_\cV(x^{\otimes n})
\arrow[ur,"\bTr_\cV(r_x^{\otimes n})"]
\arrow[r,"\bTr_\cV(\varphi_{x}^{\otimes n})"]
\arrow[drr, "\bTr_\cV(\varphi_{x^{\otimes n}})" yshift=-5]
\arrow[ddrrr, bend right=15, "\varphi_{\bTr_\cV(x^{\otimes n})}" yshift=-5]
&
\bTr_\cV(\overline{\overline{x}}^{\otimes n})
\arrow[r,"\bTr_\cV(\nu)"]
&
\bTr_\cV(\overline{\overline{x}^{\otimes n}})
\arrow[r,"\chi^{\bTr_\cV}_{\overline{x}^{\otimes n}}"]
\arrow[d,"\bTr_\cV(\overline{\nu})"]
&
\overline{\bTr_\cV(\overline{x}^{\otimes n})}
\arrow[d,"\overline{\bTr_\cV(\nu)}"]
\\
&&
\bTr_\cV(\overline{\overline{x^{\otimes n}}})
\arrow[r,"\chi^{\bTr_\cV}_{\overline{x^{\otimes n}}}"]
\arrow[dr,phantom,"\text{\scriptsize{Prop.~\ref{prop:AdjointInvolutive}}}"]
&
\overline{\bTr_\cV(\overline{x^{\otimes n}})}
\arrow[d,"\overline{\chi^{\bTr_\cV}_{x^{\otimes n}}}"]
\\
&&&
\overline{\overline{\bTr_\cV(x^{\otimes n})}}
\end{tikzcd}
$$
This is equal to the curved arrow $\varphi_{\bTr_\cV(x^{\otimes n})}=\varphi_{\cP[n]}$.
\end{proof}

\begin{prop}
The pair $(\cP,r)$ satisfies \textup{\ref{P:Reflection}}:
$$
\overline{Z(T)}\circ \nu^{(k)} \circ (r_{n_k}\otimes \cdots \otimes r_{n_1}) = r_{n_0}\circ Z(\overline{T}).
$$
(When $k=0$, $\overline{Z(T)}\circ r_\cV = r_{n_0}\circ Z(\overline{T})$ where $r_\cV: 1_\cV \to \overline{1_\cV}$ is the real structure of $1_\cV$.)
\end{prop}
\begin{proof}
This essentially follows from the fact that $\bTr_\cV$ is involutive lax monoidal by Proposition \ref{prop:AdjInvolutiveLaxMonoidal}.
It is enough to check \ref{P:Reflection} on each of the following generating tangles:
$$
u=
\tikzmath{
\draw[very thick] (1,0) circle (.25cm);
\filldraw[red] (.75,0) circle (.05cm);
}
\quad
a_i=
\tikzmath{
	\draw[thick, red] (-90:1cm) -- (-90:.25cm);
	\draw[very thick] (0,0) circle (1cm);
	\draw[very thick] (0,0) circle (.25cm);
	\draw (115:.25cm) .. controls ++(115:.5cm) and ++(65:.5cm) .. (65:.25cm);
	\draw (155:.25cm) -- (155:1cm);
	\draw (25:.25cm) -- (25:1cm);
\node at (-.63,.1) {$\scriptstyle i$};
\node at (.63,.08) {$\scriptstyle n-i$};
}
\quad
a_i^\dag=
\tikzmath{
	\draw[thick, red] (-90:1cm) -- (-90:.25cm);
	\draw[very thick] (0,0) circle (1cm);
	\draw[very thick] (0,0) circle (.25cm);
	\draw (115:1cm) .. controls ++(-65:.4cm) and ++(245:.4cm) .. (65:1cm);
	\draw (155:.25cm) -- (155:1cm);
	\draw (25:.25cm) -- (25:1cm);
\node at (-.63,.1) {$\scriptstyle i$};
\node at (.63,.08) {$\scriptstyle n-i$};
}
\quad
m_{i,j}
=
\tikzmath{
	\pgfmathsetmacro{\coord}{(2.5,0)};
	\pgfmathsetmacro{\siz}{1.2};
	\draw ($ (2.5,0) + 14/23*(0,\siz) - 1/3*(\siz,0) $) -- ($ (2.5,0) - 1/3*(\siz,0) $);
	\draw ($ (2.5,0) + 14/23*(0,\siz) + 1/3*(\siz,0) $) -- ($ (2.5,0) + 1/3*(\siz,0) $);
	\draw[thick, red] ($ (2.5,0) - 1/3*(\siz,0) - 1/5*(\siz,0) $) -- ($ (2.5,0) - 5/6*(\siz,0) $);
	\draw[thick, red] ($ (2.5,0) + 1/3*(\siz,0) - 1/5*(\siz,0) $) .. controls ++(180:.2cm) and ++(0:.2cm) .. ($ (2.5,0) - 1/3*(\siz,0) - 2/5*(0,\siz) $) .. controls ++(180:.2cm) and ++(0:.2cm) .. ($ (2.5,0) - 5/6*(\siz,0) $);
	\draw[very thick] (2.5,0) ellipse ( {5/6*\siz} and {2/3*\siz});
	\filldraw[very thick, unshaded] ($ (2.5,0) + 1/3*(\siz,0) $) circle (1/5*\siz);
	\filldraw[very thick, unshaded] ($ (2.5,0) - 1/3*(\siz,0) $) circle (1/5*\siz);
	\node at ($ (2.5,0) + (.2,0) - 1/3*(\siz,0) + 1/3*(0,\siz) $) {$\scriptstyle i$};
	\node at ($ (2.5,0) + (.2,0) + 1/3*(\siz,0) + 1/3*(0,\siz) $) {$\scriptstyle j$};
}
\quad
t_{i,j}
=
\tikzmath{
	\draw [thick, red] (0,0) -- (180:1cm);
	\draw[] (70:.3cm) .. controls ++(70:.3cm) and ++(90:.5cm) .. (0:.5cm) .. controls ++(270:.8cm) and ++(270:.8cm) .. (180:.7cm) .. controls ++(90:.6cm) and ++(270:.4cm) .. (110:1cm);
	\draw (110:.3cm) .. controls ++(110:.3cm) and ++(270:.4cm) .. (70:1cm);
	\draw[very thick] (0,0) circle (1cm);
	\draw[unshaded, very thick] (0,0) circle (.3cm);
	\node at (-80:.8cm) {$\scriptstyle j$};
	\node at (60:.8cm) {$\scriptstyle i$};
}\,.
$$

First, since $\overline{u}=u$, by unitality,
$$
\overline{Z(u)}\circ r_\cV
=
\overline{i}
\circ
r_\cV
=
\chi^{\bTr_\cV}_1
\circ
\bTr_\cV(r_\cC)
\circ 
i
=
r_0\circ Z(\overline{u}).
$$
Since $\overline{m_{i,j}}=m_{j,i}$,
by monoidality,
\begin{align*}
\overline{Z(m_{i,j})}\circ \nu\circ (r_j\otimes r_i)
&=
\overline{\mu_{i,j}}\circ \nu\circ
\big(
(\chi^{\bTr_\cV}_{x^{\otimes j}}\circ \bTr_\cV(\nu)\circ \bTr_\cV(r_x^{\otimes j}))
\otimes
(\chi^{\bTr_\cV}_{x^{\otimes i}}\circ \bTr_\cV(\nu)\circ \bTr_\cV(r_x^{\otimes i}))
\big)
\\&=
\overline{\mu_{i,j}}\circ \nu\circ (\chi^{\bTr_\cV}_{x^{\otimes j}}\otimes \chi^{\bTr_\cV}_{x^{\otimes i}})
\circ
(\bTr_\cV(\nu)\otimes \bTr_\cV(\nu))
\circ
(\bTr_\cV(r_x^{\otimes j})\otimes \bTr_\cV(r_x^{\otimes i}))
\\&=
\chi^{\bTr_\cV}_{x^{\otimes i+j}}
\circ
\bTr_\cV(\nu)
\circ
\mu_{\overline{x^{\otimes j}},\overline{x^{\otimes i}}}
\circ
(\bTr_\cV(\nu)\otimes \bTr_\cV(\nu))
\circ
(\bTr_\cV(r_x^{\otimes j})\otimes \bTr_\cV(r_x^{\otimes i}))
\\&=
\chi^{\bTr_\cV}_{x^{\otimes i+j}}
\circ
\bTr_\cV(\nu)
\circ
\bTr_\cV(r_x^{\otimes i+j})
\circ
\mu_{j,i}
\\&= r_{i+j}\circ Z(\overline{m_{i,j}}).
\end{align*}
The third equality above uses the lax involutive property of $\bTr_\cV$, and the fourth equality uses naturality of $\mu$.
Since $\overline{a_i}=a_{n-i}$ (when $Z(a_i):\bTr_\cV(x^{\otimes n+2})\to\bTr_\cV(x^{\otimes n})$), we see
\begin{align*}
\overline{Z(a_i)} \circ r_{n+2}
&=
\overline{\bTr_\cV(\id_{x^{\otimes i}} \otimes \ev_x \otimes x^{\otimes n-i})}
\circ
\chi^{\bTr_\cV}_{x^{\otimes n+2}}
\circ
 \bTr_\cV(\nu)
\circ
\bTr_\cV(r_x^{\otimes n+2})
\\&=
\chi^{\bTr_\cV}_{x^{\otimes n}}
\circ
\bTr_\cV(\nu)
\circ
\bTr_\cV(r_x^{\otimes n})
\circ
\bTr_\cV(\id_{x^{\otimes n-i}} \otimes \ev_x \otimes x^{\otimes i})
\\&=
r_n\circ Z(a_{n-i}).
\end{align*}
The proof for the $a_i^\dag$ is similar and omitted.
(It also follows formally from the proof for $a_i$ as $\bTr_\cV$ is a dagger functor and $\chi^{\bTr_\cV}$, $r$, $\nu$, and $r_x$ are all unitary.)
Finally, for the tangles $t_{i,j}$, we have $\overline{t_{i,j}}=t_{i,j}^{-1}$, and by Lemma \ref{lem:BarTraciatorCompatibility},
\begin{align*}
\overline{Z(t_{i,j})}\circ r_n
&=
\overline{\tau_{i,j}}
\circ
\chi^{\bTr_\cV}_{x^{\otimes n}}
\circ
\bTr_\cV(\nu)
\circ 
\bTr_\cV(r_x^{\otimes n})
\\&=
\chi^{\bTr_\cV}_{x^{\otimes n}}
\circ
\bTr_\cV(\nu)
\circ
\tau^{-1}_{\overline{x}^{\otimes i},\overline{x}^{\otimes j}} 
\circ
\bTr_\cV(r_x^{\otimes n})
\\&=
\chi^{\bTr_\cV}_{x^{\otimes n}}
\circ
\bTr_\cV(\nu)
\circ
\bTr_\cV(r_x^{\otimes n})
\circ
\tau^{-1}_{x^{\otimes i},x^{\otimes j}} 
\\&=
r_n \circ Z(\overline{t_{i,j}}).
\qedhere
\end{align*}
\end{proof}

\begin{defn}
\label{defn:StateOnPIsIDag}
We define a state $\psi_\cP:=i^\dag: \cP[0] \to 1_\cV$ as in \S\ref{sec:idag}.
\end{defn}

\begin{prop}
The triple $(\cP,r,\psi_\cP)$ satisfies \textup{\ref{P:PositiveDefinite}}.
\end{prop}
\begin{proof}
Suppressing $\chi^{\bTr_\cV}$ and $\nu$,
\[
\tikzmath{
	\coordinate (a) at (0,0);
	\pgfmathsetmacro{\boxWidth}{1};
	\draw (-3,-.5) -- (-1,-.5);
	\draw (-3,.5) -- (-1,.5);
	\draw (1,0) -- (2,0);
	\draw[dotted] (2,0) -- (3,0);
	\roundNbox{unshaded}{(-2,-.5)}{.3}{0}{0}{$\scriptstyle r_n^{-1}$}
	\roundNbox{unshaded}{(2,0)}{.3}{0}{0}{$\psi_\cP$}
	\draw[rounded corners=5pt, very thick, unshaded] ($ (a) - (\boxWidth,\boxWidth) $) rectangle ($ (a) + (\boxWidth,\boxWidth) $);
	\draw ($ (a) + 1/3*(0,1) $) -- ($ (a) - 1/3*(0,\boxWidth) $);
	\draw[thick, red] ($ (a) + 1/3*(0,\boxWidth) - 1/5*(\boxWidth,0) $) -- ($ (a) - 2/3*(\boxWidth,0) $);
	\draw[thick, red] ($ (a) - 1/3*(0,\boxWidth) - 1/5*(\boxWidth,0) $) -- ($ (a) - 2/3*(\boxWidth,0) $);
	\draw[very thick] (a) ellipse ({2/3*\boxWidth} and {5/6*\boxWidth});
	\filldraw[very thick, unshaded] ($ (a) + 1/3*(0,\boxWidth) $) circle (1/5*\boxWidth);
	\filldraw[very thick, unshaded] ($ (a) - 1/3*(0,\boxWidth) $) circle (1/5*\boxWidth);
	\node at ($ (a) + (.2,0) $) {\scriptsize{$n$}};
	\node at (-2.65,-.3) {\scriptsize{$\overline{\cP[n]}$}};
	\node at (-1.35,-.3) {\scriptsize{$\cP[n]$}};
	\node at (-2,.7) {\scriptsize{$\cP[n]$}};
	\node at (1.35,.2) {\scriptsize{$\cP[0]$}};
	\node at (2.65,.2) {\scriptsize{$1_\cV$}};
}
=
\begin{tikzpicture}[baseline=1.5cm]
	\coordinate (a1) at (0,0);
	\coordinate (a2) at (1.4,0);
	\coordinate (b1) at (0,1);
	\coordinate (b2) at (1.4,1);
	\coordinate (c1) at (.7,2.5);
	\draw[thick] (c1) arc (180:0:.3cm);
	\pairOfPants{(b1)}	
	\draw[thick, blue] ($ (b2) + (.3,-.1) $) node[below]{$\overline{\scriptstyle x^{\otimes n}}$} to[in=90,out=90, looseness=2.3]($ (b1) + (.3,-.1) $) node[below]{$\scriptstyle x^{\otimes n}$};
\end{tikzpicture}
\underset{\text{\scriptsize (Prop.~\ref{prop:FormulaForEvTr})}}{=}
\coev^\dag_{\bTr_\cV(x^{\otimes n})}.
\qedhere
\]
\end{proof}

\begin{prop}
If $\psi_\cC$ is spherical,\footnote{The existence of a spherical state $\psi_\cC$ is only possible when $\cV$ is spherical by Lemma \ref{lem: C-->D and D spherical}.} so is $\psi_\cP$.
\end{prop}
\begin{proof}
If $\psi_\cC$ is spherical, then
\[
\tikzmath{
	\coordinate (a) at (0,0);
	\pgfmathsetmacro{\boxWidth}{1};
	\draw (-2,0) -- (-1,0);
	\draw (1,0) -- (2,0);
	\draw[dotted] (2,0) -- (3,0);
	\roundNbox{unshaded}{(2,0)}{.3}{0}{0}{$\psi_\cP$}
	\draw[rounded corners=5pt, very thick, unshaded] ($ (a) - (\boxWidth,\boxWidth) $) rectangle ($ (a) + (\boxWidth,\boxWidth) $);
	\draw[thick, red] (180:1/5*\boxWidth) -- (180:4/5*\boxWidth);
	\draw[very thick] (a) circle (4/5*\boxWidth);
	\draw[very thick] (a) circle (1/5*\boxWidth);
	\draw (90:1/5*\boxWidth) .. controls ++(90:.3cm) and ++(90:.5cm) .. (180:1/2*\boxWidth) .. controls ++(270:.5cm) and ++(270:.3cm) ..  (270:1/5*\boxWidth);
	\node at (1.35,.2) {\scriptsize{$\cP[0]$}};
	\node at (2.65,.2) {\scriptsize{$1_\cV$}};
	\node at (-1.5,.2) {\scriptsize{$\cP[2]$}};
}
=
\begin{tikzpicture}[baseline=.9cm, xscale=-1]
	\coordinate (a1) at (0,0);
	\coordinate (b1) at (0,2);
	\coordinate (c) at (.15,.2);
	\draw[thick] (a1) -- (b1);
	\draw[thick] ($ (a1) + (.6,0) $) -- ($ (b1) + (.6,0) $);
	\halfDottedEllipse{(0,0)}{.3}{.1}
	\halfDottedEllipse{(0,2)}{.3}{.1}
	\draw[thick] (b1) arc (180:0:.3cm);		
	\draw[thick, blue] (c)+(.15,1.2) arc (180:0:.08cm)  .. controls ++(270:.2cm) and ++(135:.2cm) .. ($ (c) + (.45,.6) $);	
	\draw[thick, blue] ($ (a1) + (.15,-.08) $) .. controls ++(270:.2cm) and ++(45:.2cm) .. ($ (c) + (-.15,.4) $);
	\draw[thick, blue, dotted] ($ (c) + (-.15,.4) $) -- ($ (c) + (.45,.6) $);	
	\draw[thick, blue] (c)++(.15,1.2) .. controls ++(270:.8cm) and ++(90:.8cm) .. ($ (a1) + (.45,-.08) $);	
\end{tikzpicture}
\underset{\text{(Lem.~\ref{lem:idagAllowsIsotopy})}}{=}
\begin{tikzpicture}[baseline=.9cm]
	\coordinate (a1) at (0,0);
	\coordinate (b1) at (0,2);
	\coordinate (c) at (.15,.4);
	\draw[thick] (a1) -- (b1);
	\draw[thick] ($ (a1) + (.6,0) $) -- ($ (b1) + (.6,0) $);
	\halfDottedEllipse{(0,0)}{.3}{.1}
	\halfDottedEllipse{(0,2)}{.3}{.1}
	\draw[thick] (b1) arc (180:0:.3cm);			
	\draw[thick, blue] ($ (a1) + (.15,-.08) $) -- ($ (a1) + (.15,1) $) arc (180:0:.15cm) -- ($ (a1) + (.45,-.08) $);	
\end{tikzpicture}
=
\tikzmath{
	\coordinate (a) at (0,0);
	\pgfmathsetmacro{\boxWidth}{1};
	\draw (-2,0) -- (-1,0);
	\draw (1,0) -- (2,0);
	\draw[dotted] (2,0) -- (3,0);
	\roundNbox{unshaded}{(2,0)}{.3}{0}{0}{$\psi_\cP$}
	\draw[rounded corners=5pt, very thick, unshaded] ($ (a) - (\boxWidth,\boxWidth) $) rectangle ($ (a) + (\boxWidth,\boxWidth) $);
	\draw[thick, red] (180:1/5*\boxWidth) -- (180:4/5*\boxWidth);
	\draw[very thick] (a) circle (4/5*\boxWidth);
	\draw[very thick] (a) circle (1/5*\boxWidth);
	\draw (90:1/5*\boxWidth) .. controls ++(90:.3cm) and ++(90:.5cm) .. (0:1/2*\boxWidth) .. controls ++(270:.5cm) and ++(270:.3cm) ..  (270:1/5*\boxWidth);
	\node at (1.35,.2) {\scriptsize{$\cP[0]$}};
	\node at (2.65,.2) {\scriptsize{$1_\cV$}};
	\node at (-1.5,.2) {\scriptsize{$\cP[2]$}};
}\,.
\qedhere
\]
\end{proof}

\subsubsection{Functoriality}
\label{sec:LambdaUnitaryExtension}

Suppose $G=(G,\gamma): (\cC_1,\Phi_1^{\scriptscriptstyle \cZ},\vee_1,\psi_1,x_1) \to (\cC_2,\Phi_2^{\scriptscriptstyle \cZ},\vee_2,\psi_2,x_2)$ is a 1-morphism of pointed unitary module multitensor categories over $\cV$.
(Recall that $\gamma$ is a family of unitary isomorphisms $\gamma_v: \Phi_2(v)\to G(\Phi_1(v))$.)

Let $(\cP_i,r^i,\psi_i)=\Lambda(\cC_i,\Phi_i^{\scriptscriptstyle \cZ},\vee_i,\psi_i,x_i)$ be the unitary anchored algebra constructed in the previous subsection for $i=1,2$.
In \cite[\S5.3]{MR4528312}, we obtained a map of ordinary anchored planar algebras $\Lambda(G): \cP_1\to \cP_2$
as follows.
\begin{itemize}
\item
First, for each $c\in \cC_1$, we define $\zeta_c : \bTr_\cV^1(c) \to \bTr_\cV^2(G(c))$ to be the mate of
$$
\Phi_2(\bTr_\cV^1(c)) \xrightarrow{\gamma_{\bTr_\cV^1(c)}} G(\Phi_1(\bTr_\cV^1(c))) \xrightarrow{G(\varepsilon_c^1)} G(c)
$$
under the unitary adjunction $\Phi_2 \dashv \bTr_\cV^2$.

\item
The map $\Lambda(G):\cP_1\to\cP_2$ associated to $G:(\cC_1, x_1)\to(\cC_2, x_2)$ is the sequence of morphisms
$\Lambda(G)[n]:\cP_1[n]\to\cP_2[n]$
given by
\[
\Lambda(G)[n]\,:\,\cP_1[n] 
= \bTr_\cV^1(x_1^{\otimes n}) 
\xrightarrow{\,\,\,\,\textstyle\zeta_{x_1^{\otimes n}}\,\,\,} 
\bTr_\cV^2(G(x_1^{\otimes n})) 
\xrightarrow{\cong} 
\bTr_\cV^2(x_2^{\otimes n}) 
= 
\cP_2[n].
\]
\end{itemize}

It remains to prove that $\Lambda(G)$ is compatible with $r_n$ and $\psi$, i.e., 
\begin{align}
\label{eq: Lambda and r}
\overline{\Lambda(G)[n]}\circ r_n^1 
&= 
r_n^2\circ \Lambda(G)[n]
&&
\forall\,n\geq 0
\\\text{and}\quad
\label{eq:Lambda0PsiCOmpatible}
\psi_1 &= \psi_2\circ \Lambda(G)[0].
\end{align}
Equation \eqref{eq: Lambda and r} is checked by the following commutative diagram.

\adjustbox{scale=.8,center}{
\hspace*{-1.4cm}
\begin{tikzcd}[column sep = 4em]
\bTr^1_\cV(x_1^{\otimes n})
\arrow[r,"\bTr^1_\cV(r_{x_1}^{\otimes n})"]
\arrow[d,"\zeta_{x_1^{\otimes n}}"]
&
\bTr^1_\cV(\overline{x_1}^{\otimes n})
\arrow[r,"\bTr^1_\cV(\nu)"]
\arrow[d,"\zeta_{\overline{x_1}^{\otimes n}}"]
&
\bTr^1_\cV(\overline{x_1^{\otimes n}})
\arrow[rr,"\chi^{\bTr^1_\cV}_{x_1^{\otimes n}}"]
\arrow[d,"\zeta_{\overline{x_1^{\otimes n}}}"]
&
\arrow[d,phantom,"\text{\scriptsize{(Lem.~\ref{lem:ZetaInvolutive})}}"]
&
\overline{\bTr^1_\cV(x_1^{\otimes n})}
\arrow[d,"\overline{\zeta_{x_1^{\otimes n}}}"]
\\
\bTr_\cV^2(G(x_1^{\otimes n}))
\arrow[r,"\bTr^2_\cV(G(r_{x_1}^{\otimes n}))"]
\arrow[d,"\cong"]
&
\bTr_\cV^2(G(\overline{x_1}^{\otimes n}))
\arrow[r,"\bTr^2_\cV(G(\nu))"]
\arrow[d,"\cong"]
&
\bTr_\cV^2(G(\overline{x_1^{\otimes n}})) 
\arrow[r,"\bTr^2_\cV(\chi^G_{x_1^{\otimes n}})"]
\arrow[dr,"\cong"]
&
\bTr_\cV^2(\overline{G(x_1^{\otimes n})}) 
\arrow[r,"\chi^{\bTr^2_\cV}_{G(x_1^{\otimes n})}"]
\arrow[d,"\cong"]
&
\overline{\bTr_\cV^2(G(x_1^{\otimes n}))} 
\arrow[d,"\cong"]
\\
\bTr_\cV^2(x_2^{\otimes n}) 
\arrow[r,"\bTr^2_\cV(r_{x_2}^{\otimes n})"]
&
\bTr_\cV^2(\overline{x_2}^{\otimes n}) 
\arrow[rr,"\bTr^2_\cV(\nu)"]
&&
\bTr_\cV^2(\overline{x_2^{\otimes n}}) 
\arrow[r,"\chi^{\bTr^2_\cV}_{x_2^{\otimes n}}"]
&
\overline{\bTr_\cV^2(x_2^{\otimes n})} 
\end{tikzcd}
}

\begin{lem}
\label{lem:ZetaInvolutive}
The natural transformation $\zeta : \bTr_\cV^1 \Rightarrow \bTr_\cV^2\circ G$ is involutive, i.e., for all $c\in \cC_1$,
$\overline{\zeta_c}\circ \chi^{\bTr_\cV^1}_{c} = \chi^{\bTr_\cV^2}_{G(c)}\circ \bTr_\cV^2(\chi^G_c) \circ \zeta_{\overline{c}}$.
\end{lem}
\begin{proof}
We prove the equivalent relation:
$$
(\chi^{\bTr_\cV^2}_{G(c)})^{-1}\circ \overline{\zeta_c} = \bTr_\cV^2(\chi^G_c) \circ \zeta_{\overline{c}}\circ (\chi^{\bTr_\cV^1}_{c})^{-1}.
$$
To do so, we take mates under the adjunction
$$
\cV(\overline{\bTr_\cV^1(c)}\to \bTr_\cV^2(\overline{G(c)}))
\cong
\cC_1(\Phi_2(\overline{\bTr_\cV^1(c)})\to \overline{G(c)}).
$$
In the commutative diagram below, the mate of the left hand side above is going down and then right, and the mate of the right hand side is going right and then down.
$$
\begin{tikzcd}[column sep = 4em]
\Phi_2(\overline{\bTr_\cV^1(c)})
\arrow[r,"\Phi_2(\chi^{\bTr_\cV^1}_c)^{-1}"]
\arrow[dr,"\gamma_{\overline{\bTr_\cV^1(c)}}"]
\arrow[ddr,swap,"\chi^{\Phi_2}_{\bTr_\cV^1(c)}"]
\arrow[ddd,"\Phi_2(\overline{\zeta_c})"]
&
\Phi_2(\bTr_\cV^1(\overline{c}))
\arrow[rr,"\gamma_{\bTr_\cV^1(\overline{c})}"]
&&
G(\Phi_1(\bTr_\cV^1(\overline{c}))) 
\arrow[d,"G(\varepsilon_{\overline{c}}^1)"]
\\
&
G(\Phi_1(\overline{\bTr_\cV^1(c)}))
\arrow[urr,"G(\Phi_1(\chi^{\bTr_\cV^1}_c))^{-1}"]
\arrow[r,swap,"G(\chi^{\Phi_1}_{\bTr_\cV^1(c)})"]
&
G(\overline{\Phi_1(\bTr_\cV^1(c))})
\arrow[r,"G(\overline{\varepsilon^1_c})"]
\arrow[d,"\chi^G_{\Phi_1(\bTr_\cV^1(c))}"]
&
G(\overline{c})
\arrow[dd,"\chi^G_c"]
\\
&
\overline{\Phi_2(\bTr^1_\cV(c))}
\arrow[r,"\overline{\gamma_{\bTr_\cV^1(c)}}"]
\arrow[dr,"\overline{\Phi_2(\zeta_c)}"]
&
\overline{G(\Phi_1(\bTr_\cV^1(c)))}
\arrow[dr,"\overline{G(\varepsilon^1_c)}"]
&
\\
\Phi_2(\overline{\bTr_\cV^2(G(c))})
\arrow[rr,"\chi^{\Phi_2}_{\bTr_\cV^2(G(c))}"]
&&
\overline{\Phi_2(\bTr_\cV^2(G(c)))}
\arrow[r,"\overline{\varepsilon^2_{G(c)}}"]
&
\overline{G(c)}
\end{tikzcd}
$$
The pentagon in the diagram above is the involutivity axiom for $\gamma$.
\end{proof}

Checking \eqref{eq:Lambda0PsiCOmpatible} is straightforward.
By Lemma \ref{lem:IdentifyStatesOnGroundC*Alg},
$\psi_j$ on $\End_{\cC_j}(1_{\cC_j})$ is identified with $i_j^\dag\circ -$ on $\Hom_\cV(1\to \cP_j[0])$ under the isomorphism \eqref{eq:IdentifyGroundC*Alg},
which is exactly $\psi_j$ on $\cP_j$ by Definition \ref{defn:StateOnPIsIDag}.
Since $\psi_1=\psi_2 \circ G$ on $\End_{\cC_1}(1_{\cC_1})$, the result follows.

\subsection{From unitary anchored planar algebras to unitary module tensor categories}

Given a unitary anchored planar algebra $(\cP,r,\psi_\cP)$ in $\cV$, we begin by constructing an ordinary pivotal module tensor category $(\cC,\Phi^{\scriptscriptstyle \cZ})$ as in \cite[\S6]{MR4528312}.
First, we construct a full subcategory $\cC_0$, and we obtain $\cC$ by taking the Cauchy completion.
\begin{itemize}
\item
Objects in $\cC_0$ are formal symbols $``\Phi(v)\otimes x^{\otimes n}"$ for $v\in \cV$ and $n\geq 0$.
\item
Hom spaces are defined by
$$
\cC_0(``\Phi(u)\otimes x^{\otimes k}" \to ``\Phi(v)\otimes x^{\otimes n}") := \cV(u \to v\otimes \cP[n+k]).
$$
Morphisms are represented graphically by
$
\tikzmath{
\roundNbox{fill=white}{(0,0)}{.3}{0}{0}{$f$}
\draw (0,-.7) --node[left]{$\scriptstyle u$} (0,-.3);
\draw (0,.7) --node[left]{$\scriptstyle v$} (0,.3);
\draw (.3,0) --node[above]{$\scriptstyle \cP[n+k]$} (1.3,0);
}\,.
$
\item
Composition is given by 
$$
\tikzmath{
\roundNbox{fill=white}{(0,0)}{.3}{0}{0}{$g$}
\draw (0,-.7) --node[left]{$\scriptstyle v$} (0,-.3);
\draw (0,.7) --node[left]{$\scriptstyle w$} (0,.3);
\draw (.3,0) --node[above]{$\scriptstyle \cP[p+n]$} (1.3,0);
}
\circ
\tikzmath{
\roundNbox{fill=white}{(0,0)}{.3}{0}{0}{$f$}
\draw (0,-.7) --node[left]{$\scriptstyle u$} (0,-.3);
\draw (0,.7) --node[left]{$\scriptstyle v$} (0,.3);
\draw (.3,0) --node[above]{$\scriptstyle \cP[n+k]$} (1.3,0);
}
:=
\tikzmath{
	\coordinate (a) at (0,0);
	\pgfmathsetmacro{\boxWidth}{1};
	\roundNbox{unshaded}{(-2.3,.5)}{.3}{0}{0}{$g$}
	\roundNbox{unshaded}{(-2.3,-.5)}{.3}{0}{0}{$f$}
	\draw[rounded corners=5pt, very thick, unshaded] ($ (a) - (\boxWidth,\boxWidth) $) rectangle ($ (a) + (\boxWidth,\boxWidth) $);
	\draw ($ (a) + 5/6*(0,1) $) -- ($ (a) - 5/6*(0,\boxWidth) $);
	\draw[thick, red] ($ (a) + 1/3*(0,\boxWidth) - 1/5*(\boxWidth,0) $) -- ($ (a) - 2/3*(\boxWidth,0) $);
	\draw[thick, red] ($ (a) - 1/3*(0,\boxWidth) - 1/5*(\boxWidth,0) $) -- ($ (a) - 2/3*(\boxWidth,0) $);
	\draw[very thick] (a) ellipse ({2/3*\boxWidth} and {5/6*\boxWidth});
	\filldraw[very thick, unshaded] ($ (a) + 1/3*(0,\boxWidth) $) circle (1/5*\boxWidth);
	\filldraw[very thick, unshaded] ($ (a) - 1/3*(0,\boxWidth) $) circle (1/5*\boxWidth);
	\node at ($ (a) + (.2,0) $) {\scriptsize{$n$}};
	\node at ($ (a) + (.2,-.65) $) {\scriptsize{$k$}};
	\node at ($ (a) + (.2,.65) $) {\scriptsize{$p$}};
\draw (-2.3,-1.2) --node[left]{$\scriptstyle u$} (-2.3,-.8);
\draw (-2.3,-.2) --node[left]{$\scriptstyle v$} (-2.3,.2);
\draw (-2.3,1.2) --node[left]{$\scriptstyle w$} (-2.3,.8);
\draw (-1,.5) --node[above]{$\scriptstyle \cP[p+n]$} (-2,.5);
\draw (-1,-.5) --node[above]{$\scriptstyle \cP[n+k]$} (-2,-.5);
\draw (1,0) --node[above]{$\scriptstyle \cP[p+k]$} (2,0);
}\,.
$$
\item
The adjoint functor pair $\Phi \dashv \bTr_\cV$ is given 
on objects by
$\Phi(u):=``\Phi(u)\otimes x^{\otimes 0}"$
and
$\bTr_\cV(``\Phi(v)\otimes x^{\otimes n}"):=v\otimes \cP[n]$
and given on morphisms by
$$
\hspace*{-1cm}
\Phi\left(
\tikzmath{
\draw (0,-.7) --node[left]{$\scriptstyle u$} (0,-.3);
\draw (0,.3) --node[left]{$\scriptstyle v$} (0,.7);
\roundNbox{unshaded}{(0,0)}{.3}{0}{0}{$g$}
}
\right)
:=
\tikzmath{
\draw (0,-.7) --node[left]{$\scriptstyle u$} (0,-.3);
\draw (0,.3) --node[left]{$\scriptstyle v$} (0,.7);
\draw (1.4,0) --node[above]{$\scriptstyle \cP[0]$} (2,0);
\roundNbox{unshaded}{(0,0)}{.3}{0}{0}{$g$}
\roundNbox{unshaded}{(1,0)}{.4}{0}{0}{}
\draw[very thick] (1,0) circle (.25cm);
\filldraw[red] (.75,0) circle (.05cm);
}
\qquad\qquad
\bTr_\cV\left(
\tikzmath{
\roundNbox{fill=white}{(0,0)}{.3}{0}{0}{$f$}
\draw (0,-.7) --node[left]{$\scriptstyle u$} (0,-.3);
\draw (0,.7) --node[left]{$\scriptstyle v$} (0,.3);
\draw (.3,0) --node[above]{$\scriptstyle \cP[n+k]$} (1.3,0);
}
\right)
:=
\tikzmath{
	\coordinate (a) at (0,0);
	\pgfmathsetmacro{\boxWidth}{1};
	\roundNbox{unshaded}{(-2.3,.5)}{.3}{0}{0}{$f$}
	\draw[rounded corners=5pt, very thick, unshaded] ($ (a) - (\boxWidth,\boxWidth) $) rectangle ($ (a) + (\boxWidth,\boxWidth) $);
	\draw ($ (a) + 5/6*(0,1) $) -- ($ (a) - 1/3*(0,\boxWidth) $);
	\draw[thick, red] ($ (a) + 1/3*(0,\boxWidth) - 1/5*(\boxWidth,0) $) -- ($ (a) - 2/3*(\boxWidth,0) $);
	\draw[thick, red] ($ (a) - 1/3*(0,\boxWidth) - 1/5*(\boxWidth,0) $) -- ($ (a) - 2/3*(\boxWidth,0) $);
	\draw[very thick] (a) ellipse ({2/3*\boxWidth} and {5/6*\boxWidth});
	\filldraw[very thick, unshaded] ($ (a) + 1/3*(0,\boxWidth) $) circle (1/5*\boxWidth);
	\filldraw[very thick, unshaded] ($ (a) - 1/3*(0,\boxWidth) $) circle (1/5*\boxWidth);
	\node at ($ (a) + (.2,0) $) {\scriptsize{$k$}};
	\node at ($ (a) + (.2,.65) $) {\scriptsize{$n$}};
\draw (-2.3,-1.2) --node[left]{$\scriptstyle u$} (-2.3,.2);
\draw (-2.3,1.2) --node[left]{$\scriptstyle v$} (-2.3,.8);
\draw (-1,.5) --node[above]{$\scriptstyle \cP[n+k]$} (-2,.5);
\draw (-1,-.5) arc (90:180:.5cm) --node[left]{$\scriptstyle \cP[k]$} (-1.5,-1.2);
\draw (1,0) --node[above]{$\scriptstyle \cP[n]$} (2,0);
}\,.
$$
Moreover, the identity map
$$
\cC\big(\Phi(u)\to ``\Phi(v)\otimes x^{\otimes n}"\big)
=
\cV\big(u\to v\otimes \cP[n]\big)
=
\cV\big(u\to \bTr_\cV(``\Phi(v)\otimes x^{\otimes n}")\big)
$$
witnesses the adjunction $\Phi \dashv \bTr_\cV$.
\item
Tensor product is given by
$$
\tikzmath{
\roundNbox{fill=white}{(0,0)}{.3}{0}{0}{$f$}
\draw (0,-.7) --node[left]{$\scriptstyle u$} (0,-.3);
\draw (0,.7) --node[left]{$\scriptstyle v$} (0,.3);
\draw (.3,0) --node[above]{$\scriptstyle \cP[n+k]$} (1.3,0);
}
\otimes
\tikzmath{
\roundNbox{fill=white}{(0,0)}{.3}{0}{0}{$g$}
\draw (0,-.7) --node[left]{$\scriptstyle w$} (0,-.3);
\draw (0,.7) --node[left]{$\scriptstyle x$} (0,.3);
\draw (.3,0) --node[above]{$\scriptstyle \cP[q+p]$} (1.3,0);
}
:=
\begin{tikzpicture}[baseline=-.6cm]
	\draw (0,0) -- (2,0);
	\draw (-1,-1) -- (2,-1);
	\draw (4.2,-.5) -- (6.8,-.5);
	\draw (-1,.8) -- (-1,-1.8);
	\draw[knot] (0,.8) -- (0,-1.8);
	\roundNbox{unshaded}{(-1,-1)}{.3}{0}{0}{$f$}
	\roundNbox{unshaded}{(0,0)}{.3}{0}{0}{$g$}
	\tensor{(3.2,-.5)}{1.2}{k}{n}{p}{q}
	\node at (1.2,.2) {\scriptsize{$\cP[q{+}p]$}};
	\node at (1.2,-.8) {\scriptsize{$\cP[n{+}k]$}};
	\node at (5.6,-.3) {$\scriptstyle \cP[q+n+p+k]$};
	\node at (-1.2,-1.6) {\scriptsize{$u$}};
	\node at (-.2,-1.6) {\scriptsize{$w$}};
	\node at (-1.2,.6) {\scriptsize{$v$}};
	\node at (-.2,.6) {\scriptsize{$x$}};
\end{tikzpicture}
$$
and the tensor unit is given by $``\Phi(1)\otimes x^{\otimes 0}"$.
The associators and unitors are inhereted from those of $\cV$, i.e.,
$$
\alpha_{
``\Phi(u)\otimes x^{\otimes i}"
,
``\Phi(v)\otimes x^{\otimes j}"
,
``\Phi(w)\otimes x^{\otimes k}"
}
:=
\tikzmath{
\draw (0,-.7) --node[left]{$\scriptstyle (u\otimes v)\otimes w$} (0,-.3);
\draw (0,.3) --node[left]{$\scriptstyle u\otimes(v\otimes w)$} (0,.7);
\draw (1.4,0) --node[above]{$\scriptstyle \cP[2r]$} (2.2,0);
\roundNbox{unshaded}{(0,0)}{.3}{0}{0}{$\alpha$}
\roundNbox{unshaded}{(1,0)}{.4}{0}{0}{}
\draw[very thick] (1,0) circle (.25cm);
\draw (1,-.25) --node[right,xshift=-.1cm]{$\scriptstyle r$} (1,.25);
\filldraw[red] (.75,0) circle (.05cm);
}
\qquad\qquad
r:=i+j+k,
$$
and similarly for the unitors $\lambda,\rho$.
\item
$\cC_0$ is rigid with duals given by $(``\Phi(v)\otimes x^{\otimes n}")^\vee:=``\Phi(v^\vee)\otimes x^{\otimes n}"$,
and evaluation and coevaluation given by
$$
\ev_{``\Phi(v)\otimes x^n"}
=
\begin{tikzpicture}[baseline=-.1cm]
	\draw (-.4,-.8) node[above, yshift=-12, xshift=1]{$\scriptstyle v^\vee$} -- (-.4,-.4) arc (180:0:.2cm) -- (0,-.8)node[above, yshift=-12]{$\scriptstyle v$};
	\draw (1.3,0) -- (2.3,0);
	\node[xshift=1] at (-.2,0) {\scriptsize{$\ev_v$}};
	\evaluationMap{(.8,0)}{.5}{n}
	\node at (1.8,.2) {\scriptsize{$\cP[2n]$}};
\end{tikzpicture}
\qquad\qquad
\coev_{``\Phi(v)\otimes x^n"}
=
\begin{tikzpicture}[baseline=-.1cm]
	\draw (-.4,.8)node[above]{$\scriptstyle v$} -- (-.4,.4) arc (-180:0:.2cm) -- (0,.8)node[above, xshift=2]{$\scriptstyle v^\vee$};
	\draw (1.3,0) -- (2.3,0);
	\node at (-.2,0) {\scriptsize{$\coev_v$}};
	\coevaluationMap{(.8,0)}{.5}{n}
	\node at (1.8,.2) {\scriptsize{$\cP[2n]$}};
\end{tikzpicture}\,.
$$
Note here that the evaluation and coevaluation in $\cV$ come from our chosen unitary dual functor of $\cV$.
These choices of duals endow $\cC_0$ with a dual functor $\vee_\cC$.
\item
The pivotal structure $\varphi^\cC:``\Phi(v)\otimes x^{\otimes n}"\to (``\Phi(v)\otimes x^{\otimes n}")^{\vee\vee}$ is given by
$$
\begin{tikzpicture}[baseline=-.1cm]
	\draw (0,-1) -- (0,1);
	\roundNbox{unshaded}{(0,0)}{.4}{0}{0}{$\varphi_v$}
	\draw (1.4,0) -- (2.4,0);
	\node at (-.2,-.8) {\scriptsize{$v$}};
	\node at (-.3,.8) {\scriptsize{$v^{\vee\vee}$}};
	\identityMap{(1,0)}{.4}{n\,\,\,}
	\node at (1.9,.2) {\scriptsize{$\cP[2n]$}};
\end{tikzpicture}
=
\begin{tikzpicture}[baseline=-.1cm]
	\draw (-.6,-1) -- (-.6,-.3);
	\draw (-.2,-.3) arc (-180:0:0.3cm) -- (.4,1);
	\roundNbox{unshaded}{(-.4,0)}{.3}{.3}{.3}{$\coev_v^\dag$}
	\draw (1.4,0) -- (2.4,0);
	\node at (-.8,-.8) {\scriptsize{$v$}};
	\node at (-.3,-.6) {\scriptsize{$v^{\vee}$}};
	\node at (.1,.8) {\scriptsize{$v^{\vee\vee}$}};
	\identityMap{(1,0)}{.4}{n\,\,\,}
	\node at (1.9,.2) {\scriptsize{$\cP[2n]$}};
\end{tikzpicture}
\,.
$$
The equality above comes from the fact that $\varphi_v$ is the canonical unitary pivotal structure of $\cV$ coming from $\vee$.
\item
The generator is $x=``\Phi(1_\cV)\otimes x"$.
Since we identify $1_\cV^\vee$ with $1_\cV$, we may identify $x^\vee=x$.
The symmetric self duality $r_x:x\to x^\vee$ is the identity map.
\end{itemize}

\subsubsection{Adding the dagger structure}

It remains to perform the following tasks:
\begin{enumerate}[label=(C\arabic*)]
\item 
\label{C:Dagger}
construct a dagger structure on $\cC$ making it a unitary multitensor category
\item
\label{C:State}
construct a faithful state $\psi_\cC$ on $\End_\cC(1_\cC)$,
\item
\label{C:UDF}
check that $\vee_\cC$ is a unitary dual functor and that 
the canonical unitary pivotal structure induced by $\vee_\cC$ is $\varphi^\cC$.
\item
\label{C:Real}
check that $r_x:x\to x^\vee$ is a real structure.
\item
\label{C:Adjunction}
check that the adjunction $\Phi \dashv \bTr_\cV$ is unitary.
\end{enumerate}


\begin{defn}
We define a dagger structure on $\cC_0$ by
\begin{equation}
\label{eq:DaggerOnC0}
\left(
\tikzmath{
\roundNbox{fill=white}{(0,0)}{.3}{0}{0}{$f$}
\draw (0,-.7) --node[left]{$\scriptstyle u$} (0,-.3);
\draw (0,.7) --node[left]{$\scriptstyle v$} (0,.3);
\draw (.3,0) --node[above]{$\scriptstyle \cP[n+k]$} (1.3,0);
}
\right)^*
=
\tikzmath{
\roundNbox{fill=white}{(0,0)}{.3}{0}{0}{$f^*$}
\draw (0,-.7) --node[left]{$\scriptstyle v$} (0,-.3);
\draw (0,.7) --node[left]{$\scriptstyle u$} (0,.3);
\draw (.3,0) --node[above]{$\scriptstyle \cP[n+k]$} (1.3,0);
}
:=
\tikzmath{
\roundNbox{fill=white}{(0,0)}{.3}{.1}{.1}{$f^\dag$}
\roundNbox{fill=white}{(1,0)}{.3}{.2}{.2}{$\scriptstyle r_{n+k}^{-1}$}
\draw (-.2,-.7) --node[left]{$\scriptstyle v$} (-.2,-.3);
\draw (0,.7) --node[left]{$\scriptstyle u$} (0,.3);
\draw (.2,-.3) 
arc (-180:0:.4cm) node[right, yshift=-.2cm]{$\scriptstyle \overline{\cP[n+k]}$};
\draw (1,.3) arc (180:90:.3cm) -- (1.6,.6) node[right]{$\scriptstyle \cP[k+n]$};
}\,.
\end{equation}
It is straightforward to verify that $f^{**}=f$ for all morphisms $f$.
To check the remainder of involutivity, we see that $*$ on the composite
$$
``\Phi(u)\otimes x^{\otimes k}"
\xrightarrow{f}
``\Phi(v)\otimes x^{\otimes n}"
\xrightarrow{g}
``\Phi(w)\otimes x^{\otimes p}"
$$
is given by
\begin{align*}
\left(
\tikzmath{
\roundNbox{fill=white}{(0,0)}{.3}{0}{0}{$g$}
\draw (0,-.7) --node[left]{$\scriptstyle v$} (0,-.3);
\draw (0,.7) --node[left]{$\scriptstyle w$} (0,.3);
\draw (.3,0) --node[above]{$\scriptstyle \cP[p+n]$} (1.3,0);
}
\circ
\tikzmath{
\roundNbox{fill=white}{(0,0)}{.3}{0}{0}{$f$}
\draw (0,-.7) --node[left]{$\scriptstyle u$} (0,-.3);
\draw (0,.7) --node[left]{$\scriptstyle v$} (0,.3);
\draw (.3,0) --node[above]{$\scriptstyle \cP[n+k]$} (1.3,0);
}
\right)^*
&=
\tikzmath{
	\draw (.4,-.3) arc (-180:-90:.2cm) --node[above]{$\scriptstyle \overline{\cP[p+n]}$} (2,-.5);
	\draw (.2,-1.3) arc (-180:-90:.2cm) --node[above]{$\scriptstyle \overline{\cP[n+k]}$} (2,-1.5);
	\draw (4,-1) --node[above]{$\scriptstyle \overline{\cP[p+k]}$} (5.4,-1);
	\draw (6,-1) --node[above]{$\scriptstyle \cP[p+k]$} (7.4,-1);
	\draw (.2,.3) --node[left]{$\scriptstyle w$} (.2,.8);
	\draw (0,-.7) --node[left]{$\scriptstyle v$} (0,-.3);
	\draw (-.2,-1.3) --node[left]{$\scriptstyle u$} (-.2,-1.8);
	\roundNbox{unshaded}{(.2,0)}{.3}{.1}{.1}{$f^\dag$}
	\roundNbox{unshaded}{(0,-1)}{.3}{.1}{.1}{$g^\dag$}
	\roundNbox{unshaded}{(5.7,-1)}{.3}{0}{0}{$\scriptstyle r^{-1}$}
  \coordinate (a) at (3,-1);
	\pgfmathsetmacro{\boxWidth}{1};
	\draw[rounded corners=5pt, very thick, unshaded] ($ (a) - (\boxWidth,\boxWidth) $) rectangle ($ (a) + (\boxWidth,\boxWidth) + (0,.3) $);
	\draw ($ (a) + 5/6*(0,1) $) -- ($ (a) - 5/6*(0,\boxWidth) $);
	\draw[thick, red] ($ (a) + 1/3*(0,\boxWidth) - 1/5*(\boxWidth,0) $) -- ($ (a) - 2/3*(\boxWidth,0) $);
	\draw[thick, red] ($ (a) - 1/3*(0,\boxWidth) - 1/5*(\boxWidth,0) $) -- ($ (a) - 2/3*(\boxWidth,0) $);
	\draw[very thick] (a) ellipse ({2/3*\boxWidth} and {5/6*\boxWidth});
	\filldraw[very thick, unshaded] ($ (a) + 1/3*(0,\boxWidth) $) circle (1/5*\boxWidth);
	\filldraw[very thick, unshaded] ($ (a) - 1/3*(0,\boxWidth) $) circle (1/5*\boxWidth);
	\node at ($ (a) + (.2,0) $) {\scriptsize{$n$}};
	\node at ($ (a) + (.2,-.65) $) {\scriptsize{$k$}};
	\node at ($ (a) + (.2,.65) $) {\scriptsize{$p$}};
	\draw (2.2,0) -- (3.8,0);
}
\displaybreak[1]\\&=
\tikzmath{
	\coordinate (a) at (0,0);
	\pgfmathsetmacro{\boxWidth}{1};
	\roundNbox{unshaded}{(-2.3,.5)}{.3}{0}{0}{$f^*$}
	\roundNbox{unshaded}{(-2.3,-.5)}{.3}{0}{0}{$g^*$}
	\draw[rounded corners=5pt, very thick, unshaded] ($ (a) - (\boxWidth,\boxWidth) $) rectangle ($ (a) + (\boxWidth,\boxWidth) $);
	\draw ($ (a) + 5/6*(0,1) $) -- ($ (a) - 5/6*(0,\boxWidth) $);
	\draw[thick, red] ($ (a) + 1/3*(0,\boxWidth) - 1/5*(\boxWidth,0) $) -- ($ (a) - 2/3*(\boxWidth,0) $);
	\draw[thick, red] ($ (a) - 1/3*(0,\boxWidth) - 1/5*(\boxWidth,0) $) -- ($ (a) - 2/3*(\boxWidth,0) $);
	\draw[very thick] (a) ellipse ({2/3*\boxWidth} and {5/6*\boxWidth});
	\filldraw[very thick, unshaded] ($ (a) + 1/3*(0,\boxWidth) $) circle (1/5*\boxWidth);
	\filldraw[very thick, unshaded] ($ (a) - 1/3*(0,\boxWidth) $) circle (1/5*\boxWidth);
	\node at ($ (a) + (.2,0) $) {\scriptsize{$n$}};
	\node at ($ (a) + (.2,-.65) $) {\scriptsize{$p$}};
	\node at ($ (a) + (.2,.65) $) {\scriptsize{$k$}};
\draw (-2.3,-1.2) --node[left]{$\scriptstyle w$} (-2.3,-.8);
\draw (-2.3,-.2) --node[left]{$\scriptstyle v$} (-2.3,.2);
\draw (-2.3,1.2) --node[left]{$\scriptstyle u$} (-2.3,.8);
\draw (-1,.5) --node[above]{$\scriptstyle \cP[n+k]$} (-2,.5);
\draw (-1,-.5) --node[above]{$\scriptstyle \cP[p+n]$} (-2,-.5);
\draw (1,0) --node[above]{$\scriptstyle \cP[k+p]$} (2,0);
}\,,
\end{align*}
which is exactly the composite of $f^*$ and $g^*$.
We now observe that
\begin{align*}
\left(
\tikzmath{
\roundNbox{fill=white}{(0,0)}{.3}{0}{0}{$f$}
\draw (0,-.7) --node[left]{$\scriptstyle u$} (0,-.3);
\draw (0,.7) --node[left]{$\scriptstyle v$} (0,.3);
\draw (.3,0) --node[above]{$\scriptstyle \cP[n+k]$} (1.3,0);
}
\otimes
\tikzmath{
\roundNbox{fill=white}{(0,0)}{.3}{0}{0}{$g$}
\draw (0,-.7) --node[left]{$\scriptstyle w$} (0,-.3);
\draw (0,.7) --node[left]{$\scriptstyle x$} (0,.3);
\draw (.3,0) --node[above]{$\scriptstyle \cP[q+p]$} (1.3,0);
}
\right)^*
&=
\tikzmath{
	\draw (-.8,-.3) arc (-180:-90:.2cm) --node[above]{$\scriptstyle \overline{\cP[n+k]}$} (2,-.5);
	\draw (.2,-1.3) arc (-180:-90:.2cm) --node[above]{$\scriptstyle \overline{\cP[q+p]}$} (2,-1.5);
	\draw (4.4,-1) --node[above]{$\scriptstyle \overline{\cP[q+n+p+k]}$} (6.4,-1);
	\draw (7,-1) --node[above]{$\scriptstyle \cP[q+n+p+k]$} (8.8,-1);
	\draw (-1,.3) -- (-1,.8);
	\draw (-1.2,-1.8) -- (-1.2,-.3);
	\draw[knot] (0,.8) -- (0,-.7);
	\draw (-.2,-1.3) -- (-.2,-1.8);
	\roundNbox{unshaded}{(-1,0)}{.3}{.1}{.1}{$f^\dag$}
	\roundNbox{unshaded}{(0,-1)}{.3}{.1}{.1}{$g^\dag$}
	\roundNbox{unshaded}{(6.7,-1)}{.3}{0}{0}{$\scriptstyle r^{-1}$}
	\tensor{(3.2,-1)}{1.2}{k}{n}{p}{q}
	\draw (2.2,0) -- (4.2,0);
	\node at (-1.4,-1.6) {\scriptsize{$v$}};
	\node at (-.4,-1.6) {\scriptsize{$x$}};
	\node at (-1.2,.6) {\scriptsize{$u$}};
	\node at (-.2,.6) {\scriptsize{$w$}};}
\displaybreak[1]\\&=
\tikzmath[yscale=-1]{
	\draw (0,0) --node[above]{$\scriptstyle \cP[q+p]$} (2,0);
	\draw (-1,-1) --node[above, xshift=.4cm]{$\scriptstyle \cP[n+k]$} (2,-1);
	\draw (4.2,-.5) --node[above]{$\scriptstyle \cP[q+n+p+k]$} (6.8,-.5);
	\draw (-1,.8) -- (-1,-1.8);
	\draw[knot] (0,.8) -- (0,-1.8);
	\roundNbox{unshaded}{(-1,-1)}{.3}{0}{0}{$f^*$}
	\roundNbox{unshaded}{(0,0)}{.3}{0}{0}{$g^*$}
	\tensor{(3.2,-.5)}{1.2}{k}{n}{p}{q}
	\node at (-1.2,-1.6) {\scriptsize{$u$}};
	\node at (-.2,-1.6) {\scriptsize{$w$}};
	\node at (-1.2,.6) {\scriptsize{$v$}};
	\node at (-.2,.6) {\scriptsize{$x$}};
}
\displaybreak[1]\\&=
\tikzmath{
	\draw (0,0) -- (2,0);
	\draw (-1,-1) -- (2,-1);
	\draw (4.2,-.5) -- (6.8,-.5);
	\draw (-1,.8) -- (-1,-1.8);
	\draw[knot] (0,.8) -- (0,-1.8);
	\roundNbox{unshaded}{(-1,-1)}{.3}{0}{0}{$f^*$}
	\roundNbox{unshaded}{(0,0)}{.3}{0}{0}{$g^*$}
	\tensor{(3.2,-.5)}{1.2}{k}{n}{p}{q}
	\node at (1.2,.2) {\scriptsize{$\cP[q{+}p]$}};
	\node at (1.2,-.8) {\scriptsize{$\cP[n{+}k]$}};
	\node at (5.6,-.3) {$\scriptstyle \cP[q+n+p+k]$};
	\node at (-1.2,-1.6) {\scriptsize{$v$}};
	\node at (-.2,-1.6) {\scriptsize{$x$}};
	\node at (-1.2,.6) {\scriptsize{$u$}};
	\node at (-.2,.6) {\scriptsize{$w$}};
}
\end{align*}
where we use the braiding axiom of an anchored planar algebra in the final equality.
Finally, the associators and unitors are visibly unitary.
\end{defn}

\begin{lem}
\label{lem:LinkingAlgebraState}
For all $f\in \cC_0(``\Phi(u)\otimes x^{\otimes k}" \to ``\Phi(v)\otimes x^{\otimes n}")$,
$$
\tikzmath{
	\coordinate (a) at (0,0);
	\pgfmathsetmacro{\boxWidth}{1};
	\draw (-2.5,0) -- (-1,0);
	\draw (1,0) -- (2,0);
	\draw[dotted] (2,0) -- (3,0);
	\draw[knot] (-2.8,.3) node[left, yshift=.2cm, xshift=.1cm]{$\scriptstyle v$} arc (180:0:.4cm) --node[right, yshift=-.4cm]{$\scriptstyle v^\vee$} (-2,-.3) arc (0:-180:.4cm) node[left, yshift=-.2cm, xshift=.1cm]{$\scriptstyle v$};
	\roundNbox{unshaded}{(-2.8,0)}{.3}{.1}{.1}{$ff^*$}
	\roundNbox{unshaded}{(2,0)}{.3}{0}{0}{$\psi_\cP$}
	\draw[rounded corners=5pt, very thick, unshaded] ($ (a) - (\boxWidth,\boxWidth) $) rectangle ($ (a) + (\boxWidth,\boxWidth) $);
	\draw[thick, red] (180:1/5*\boxWidth) -- (180:4/5*\boxWidth);
	\draw[very thick] (a) circle (4/5*\boxWidth);
	\draw[very thick] (a) circle (1/5*\boxWidth);
	\draw (90:1/5*\boxWidth) .. controls ++(90:.3cm) and ++(90:.5cm) .. (0:1/2*\boxWidth) .. controls ++(270:.5cm) and ++(270:.3cm) ..  (270:1/5*\boxWidth);
	\node at (1.35,.2) {\scriptsize{$\cP[0]$}};
	\node at (2.65,.2) {\scriptsize{$1_\cV$}};
	\node at (-1.4,.2) {\scriptsize{$\cP[2k]$}};
	\node at (.65,0) {\scriptsize{$k$}};
}
\geq 0
$$
with equality if and only if $f=0$.
\end{lem}
\begin{proof}
Expanding the quantity in question, we have
\begin{align*}
\tikzmath{
	\coordinate (a) at (0,0);
	\coordinate (b) at (3,0);
	\pgfmathsetmacro{\boxWidth}{1};
\draw (-2.8,-.2) --node[left]{$\scriptstyle u$} (-2.8,.2);
\draw (-1,.5) --node[above]{$\scriptstyle \cP[k+n]$} (-2.2,.5) -- (-2.5,.5);
\draw (-1,-.5) --node[above]{$\scriptstyle \cP[n+k]$} (-2.2,-.5) -- (-2.5,-.5);
\draw (1,0) --node[above]{$\scriptstyle \cP[2k]$} (2,0);
\draw[knot] (-2.8,.8) node[left, yshift=.2cm, xshift=.1cm]{$\scriptstyle v$} arc (180:0:.3cm) -- (-2.2,-.8) node[right]{$\scriptstyle v^\vee$} arc (0:-180:.3cm) node[left, yshift=-.2cm, xshift=.1cm]{$\scriptstyle v$};
	\roundNbox{unshaded}{(-2.8,.5)}{.3}{0}{0}{$f$}
	\roundNbox{unshaded}{(-2.8,-.5)}{.3}{0}{0}{$f^*$}
	\roundNbox{unshaded}{($ (b) + (2,0) $)}{.3}{0}{0}{$\psi_\cP$}
	\draw[rounded corners=5pt, very thick, unshaded] ($ (a) - (\boxWidth,\boxWidth) $) rectangle ($ (a) + (\boxWidth,\boxWidth) $);
	\draw ($ (a) + 5/6*(0,1) $) -- ($ (a) - 5/6*(0,\boxWidth) $);
	\draw[thick, red] ($ (a) + 1/3*(0,\boxWidth) - 1/5*(\boxWidth,0) $) -- ($ (a) - 2/3*(\boxWidth,0) $);
	\draw[thick, red] ($ (a) - 1/3*(0,\boxWidth) - 1/5*(\boxWidth,0) $) -- ($ (a) - 2/3*(\boxWidth,0) $);
	\draw[very thick] (a) ellipse ({2/3*\boxWidth} and {5/6*\boxWidth});
	\filldraw[very thick, unshaded] ($ (a) + 1/3*(0,\boxWidth) $) circle (1/5*\boxWidth);
	\filldraw[very thick, unshaded] ($ (a) - 1/3*(0,\boxWidth) $) circle (1/5*\boxWidth);
	\node at ($ (a) + (.2,0) $) {\scriptsize{$n$}};
	\node at ($ (a) + (.2,-.65) $) {\scriptsize{$k$}};
	\node at ($ (a) + (.2,.65) $) {\scriptsize{$k$}};
	\draw[rounded corners=5pt, very thick, unshaded] ($ (b) - (\boxWidth,\boxWidth) $) rectangle ($ (b) + (\boxWidth,\boxWidth) $);
	\draw[thick, red] ($ (b) + (180:1/5*\boxWidth) $) -- ($ (b) + (180:4/5*\boxWidth) $);
	\draw[very thick] (b) circle (4/5*\boxWidth);
	\draw[very thick] (b) circle (1/5*\boxWidth);
	\draw ($ (b) + (90:1/5*\boxWidth) $) .. controls ++(90:.3cm) and ++(90:.5cm) .. ($ (b) + (0:1/2*\boxWidth) $) .. controls ++(270:.5cm) and ++(270:.3cm) ..  ($ (b) + (270:1/5*\boxWidth) $);
	\node at ($ (b) + (.65,0) $) {\scriptsize{$k$}};
\draw[densely dotted] (5.3,0) --node[above]{$\scriptstyle 1_\cV$}  (6,0);
\draw (4,0) --node[above]{$\scriptstyle \cP[0]$}  (4.7,0);
}
&=
\tikzmath{
	\coordinate (a) at (0,0);
	\pgfmathsetmacro{\boxWidth}{1};
\draw (-2.8,-.2) --node[left]{$\scriptstyle u$} (-2.8,.2);
\draw (-1,.5) --node[above]{$\scriptstyle \cP[k+n]$} (-2.2,.5) -- (-2.5,.5);
\draw (-1,-.5) --node[above]{$\scriptstyle \cP[n+k]$} (-2.2,-.5) -- (-2.5,-.5);
\draw (1,0) -- 
(2,0);
\draw[knot] (-2.8,.8) node[left, yshift=.2cm, xshift=.1cm]{$\scriptstyle v$} arc (180:0:.3cm) -- (-2.2,-.8) node[right]{$\scriptstyle v^\vee$} arc (0:-180:.3cm) node[left, yshift=-.2cm, xshift=.1cm]{$\scriptstyle v$};
	\roundNbox{unshaded}{(-2.8,.5)}{.3}{0}{0}{$f$}
	\roundNbox{unshaded}{(-2.8,-.5)}{.3}{0}{0}{$f^*$}
	\roundNbox{unshaded}{($ (a) + (2,0) $)}{.3}{0}{0}{$\psi_\cP$}
	\draw[rounded corners=5pt, very thick, unshaded] ($ (a) - (\boxWidth,\boxWidth) $) rectangle ($ (a) + (\boxWidth,\boxWidth) $);
	\draw ($ (a) + 1/3*(0,1) $) -- ($ (a) - 1/3*(0,\boxWidth) $);
	\draw[thick, red] ($ (a) + 1/3*(0,\boxWidth) - 1/5*(\boxWidth,0) $) -- ($ (a) - 2/3*(\boxWidth,0) $);
	\draw[thick, red] ($ (a) - 1/3*(0,\boxWidth) - 1/5*(\boxWidth,0) $) -- ($ (a) - 2/3*(\boxWidth,0) $);
	\draw[very thick] (a) ellipse ({2/3*\boxWidth} and {5/6*\boxWidth});
	\filldraw[very thick, unshaded] ($ (a) + 1/3*(0,\boxWidth) $) circle (1/5*\boxWidth);
	\filldraw[very thick, unshaded] ($ (a) - 1/3*(0,\boxWidth) $) circle (1/5*\boxWidth);
	\node at ($ (a) + (.3,0) $) {$\scriptstyle n+k$};
\draw[densely dotted] (2.3,0) --node[above]{$\scriptstyle 1_\cV$}  (3,0);
\draw (1,0) --node[above]{$\scriptstyle \cP[0]$}  (1.7,0);
}
\\&\underset{(\ref{eq:PlanarPairing},\ref{eq:DaggerOnC0})}{=}
\tikzmath{
\roundNbox{unshaded}{(0,.5)}{.3}{.1}{.1}{$f$}
\roundNbox{unshaded}{(0,-.5)}{.3}{.1}{.1}{$f^\dag$}
\draw (-.2,.8) node[left, yshift=.2cm]{$\scriptstyle v$} 
.. controls ++(90:.7cm) and ++(90:.7cm) .. (2,.8) --node[right]{$\scriptstyle v^\vee$}(2,-.8) 
.. controls ++(270:.7cm) and ++(270:.7cm) ..
(-.2,-.8) node[left, yshift=-.2cm]{$\scriptstyle v$};
\draw (.2,.8) node[right, yshift=.15cm, xshift=.15cm]{$\scriptstyle \cP[n]$} 
.. controls ++(90:.5cm) and ++(90:.5cm) .. 
(1.6,.8) --node[left,xshift=.1cm]{$\scriptstyle \cP[n]^\vee$} (1.6,-.8)
.. controls ++(270:.5cm) and ++(270:.5cm) .. (.2,-.8) node[right, yshift=-.15cm, xshift=.15cm]{$\scriptstyle \cP[n]$};
\draw (0,-.2) --node[left]{$\scriptstyle u$} (0,.2);
}
\geq 0
\end{align*}
with equality if and only if $f=0$.
\end{proof}

\begin{prop}[\ref{C:Dagger}]
\label{prop:C0isC*}
The dagger structure \eqref{eq:DaggerOnC0} on $\cC_0$ is $\rm C^*$.
\end{prop}
\begin{proof}
To show $\cC_0$ is $\rm C^*$, we prove that given arbitrary $``\Phi(u)\otimes x^{\otimes k}"$ and $``\Phi(v)\otimes x^{\otimes n}"$ in $\cC_0$, the linking algebra
$$
L:=
\begin{pmatrix}
\cC_0(``\Phi(u)\otimes x^{\otimes k}" \to ``\Phi(u)\otimes x^{\otimes k}")
&
\cC_0(``\Phi(v)\otimes x^{\otimes n}" \to ``\Phi(u)\otimes x^{\otimes k}")
\\
\cC_0(``\Phi(u)\otimes x^{\otimes k}" \to ``\Phi(v)\otimes x^{\otimes n}")
&
\cC_0(``\Phi(v)\otimes x^{\otimes n}" \to ``\Phi(v)\otimes x^{\otimes n}")
\end{pmatrix}
$$
is a $\rm C^*$-algebra.
To do this, we show that the map $\phi:L\to \bbC$ given by
$$
\phi
\begin{pmatrix}
a & b
\\
c & d
\end{pmatrix}
:=
\tikzmath{
	\coordinate (a) at (0,0);
	\pgfmathsetmacro{\boxWidth}{1};
	\draw (-2.1,0) -- (-1,0);
	\draw (1,0) -- (2,0);
	\draw[dotted] (2,0) -- (3,0);
	\draw[knot] (-2.4,.3) node[left, yshift=.2cm, xshift=.1cm]{$\scriptstyle u$} arc (180:0:.3cm) --node[right, yshift=-.3cm]{$\scriptstyle u^\vee$} (-1.8,-.3) arc (0:-180:.3cm) node[left, yshift=-.2cm, xshift=.1cm]{$\scriptstyle u$};
	\roundNbox{unshaded}{(-2.4,0)}{.3}{0}{0}{$a$}
	\roundNbox{unshaded}{(2,0)}{.3}{0}{0}{$\psi_\cP$}
	\draw[rounded corners=5pt, very thick, unshaded] ($ (a) - (\boxWidth,\boxWidth) $) rectangle ($ (a) + (\boxWidth,\boxWidth) $);
	\draw[thick, red] (180:1/5*\boxWidth) -- (180:4/5*\boxWidth);
	\draw[very thick] (a) circle (4/5*\boxWidth);
	\draw[very thick] (a) circle (1/5*\boxWidth);
	\draw (90:1/5*\boxWidth) .. controls ++(90:.3cm) and ++(90:.5cm) .. (0:1/2*\boxWidth) .. controls ++(270:.5cm) and ++(270:.3cm) ..  (270:1/5*\boxWidth);
	\node at (1.35,.2) {\scriptsize{$\cP[0]$}};
	\node at (2.65,.2) {\scriptsize{$1_\cV$}};
	\node at (-1.4,.2) {\scriptsize{$\cP[2k]$}};
	\node at (.65,0) {\scriptsize{$k$}};
}
+
\tikzmath{
	\coordinate (a) at (0,0);
	\pgfmathsetmacro{\boxWidth}{1};
	\draw (-2.1,0) -- (-1,0);
	\draw (1,0) -- (2,0);
	\draw[dotted] (2,0) -- (3,0);
	\draw[knot] (-2.4,.3) node[left, yshift=.2cm, xshift=.1cm]{$\scriptstyle v$} arc (180:0:.3cm) --node[right, yshift=-.3cm]{$\scriptstyle v^\vee$} (-1.8,-.3) arc (0:-180:.3cm) node[left, yshift=-.2cm, xshift=.1cm]{$\scriptstyle v$};
	\roundNbox{unshaded}{(-2.4,0)}{.3}{0}{0}{$d$}
	\roundNbox{unshaded}{(2,0)}{.3}{0}{0}{$\psi_\cP$}
	\draw[rounded corners=5pt, very thick, unshaded] ($ (a) - (\boxWidth,\boxWidth) $) rectangle ($ (a) + (\boxWidth,\boxWidth) $);
	\draw[thick, red] (180:1/5*\boxWidth) -- (180:4/5*\boxWidth);
	\draw[very thick] (a) circle (4/5*\boxWidth);
	\draw[very thick] (a) circle (1/5*\boxWidth);
	\draw (90:1/5*\boxWidth) .. controls ++(90:.3cm) and ++(90:.5cm) .. (0:1/2*\boxWidth) .. controls ++(270:.5cm) and ++(270:.3cm) ..  (270:1/5*\boxWidth);
	\node at (1.35,.2) {\scriptsize{$\cP[0]$}};
	\node at (2.65,.2) {\scriptsize{$1_\cV$}};
	\node at (-1.4,.2) {\scriptsize{$\cP[2k]$}};
	\node at (.65,0) {\scriptsize{$k$}};
}
$$
is a faithful positive linear functional, i.e., $\phi(M^*M)\geq0$ with equality if and only if $M=0$.
Indeed, $\phi(M^*M)$ is a sum of four terms, each being positive by
Lemma \ref{lem:LinkingAlgebraState}, and zero iff each term is individually zero.
\end{proof}

\begin{rem}
Observe that the unitary Cauchy completion of $\cC_0$ is the same as the ordinary Cauchy completion of $\cC_0$ \cite[Rem.~3.6]{MR4598730}.
In light of Proposition \ref{prop:C0isC*}, $\cC$ is a unitary multitensor category.
\end{rem}

\begin{cor}[\ref{C:State}]
The map $\psi_\cC$ on $\End_{\cC_0}(1_\cC)$ given by
$$
\psi_\cC\left(
\tikzmath{
\roundNbox{fill=white}{(0,0)}{.3}{0}{0}{$f$}
\draw[densely dotted] (0,-.7) --node[left]{$\scriptstyle 1_\cV$} (0,-.3);
\draw[densely dotted] (0,.7) --node[left]{$\scriptstyle 1_\cV$} (0,.3);
\draw (.3,0) --node[above]{$\scriptstyle \cP[0]$} (1.3,0);
}
\right)
:=
\tikzmath{
\roundNbox{fill=white}{(0,0)}{.3}{0}{0}{$f$}
\roundNbox{fill=white}{(1.2,0)}{.3}{0}{0}{$\psi_\cP$}
\draw[densely dotted] (0,-.7) --node[left]{$\scriptstyle 1_\cV$} (0,-.3);
\draw[densely dotted] (0,.7) --node[left]{$\scriptstyle 1_\cV$} (0,.3);
\draw (.3,0) --node[above]{$\scriptstyle \cP[0]$} (.9,0);
\draw[densely dotted] (1.5,0) --node[above]{$\scriptstyle 1_\cV$} (2,0);
}
$$
is a faithful state.
\end{cor}
\begin{proof}
Apply Lemma \ref{lem:LinkingAlgebraState} to the case $u=v=1_\cV$ and $k=n=0$ to see that $\psi_\cC$ is positive and faithful.
Finally, the condition $\psi_\cC(\id_{\cC})=1$ is exactly 
$
\tikzmath{
\draw (.3,0) --node[above]{$\scriptstyle \cP[0]$} (1,0);
\draw[dotted] (1.6,0) --node[above]{$\scriptstyle 1_\cV$} (2.1,0);
\draw[dotted] (-.8,0) --node[above]{$\scriptstyle 1_\cV$} (-.3,0);
\roundNbox{unshaded}{(0,0)}{.3}{0}{0}{}
\draw[very thick] (0,0) circle (.2cm);
\filldraw[red] (-.2,0) circle (.05cm);
\roundNbox{unshaded}{(1.3,0)}{.3}{0}{0}{$\psi_\cP$}
}
=
\id_{1_\cV}
$.
\end{proof}

\begin{prop}[\ref{C:UDF}]
The dual functor $\vee_\cC$ is unitary.
Moreover, the pivotal structure $\varphi^\cC$ is the canonical unitary pivotal structure coming from $\vee_\cC$.
\end{prop}
\begin{proof}
By \cite[Lem.~7.5]{MR2767048} (see also \cite[Prop.~3.9]{MR4133163}), it suffices to check that 
for $c:=``\Phi(v)\otimes x^{n}"$, 
the pivotal structure $\varphi^\cC_c$ above is given by 
$(\coev_c^\dag\otimes \id_{c^{\vee\vee}})\circ (\id_c\otimes \coev_{c^\vee})$.
This is verified as follows.
First, we observe that
\[
\coev_{``\Phi(v)\otimes x^n"}^\dag
=\,\,
\left(
\begin{tikzpicture}[baseline=-.1cm]
	\draw (-.4,.8)node[above]{$\scriptstyle v$} -- (-.4,.4) arc (-180:0:.2cm) -- (0,.8)node[above, xshift=2]{$\scriptstyle v^\vee$};
	\draw (1.3,0) -- (2.3,0);
	\node at (-.2,0) {\scriptsize{$\coev_v$}};
	\coevaluationMap{(.8,0)}{.5}{n}
	\node at (1.8,.2) {\scriptsize{$\cP[2n]$}};
\end{tikzpicture}
\right)^\dag
=
\begin{tikzpicture}[baseline=-.1cm]
	\draw (-.6,-1) --node[left]{\scriptsize{$v$}} (-.6,-.3);
	\draw (-.2,-.3) --node[right]{\scriptsize{$v^{\vee}$}} (-.2,-1);
	\roundNbox{unshaded}{(-.4,0)}{.3}{.3}{.3}{$\coev_v^\dag$}
	\roundNbox{unshaded}{(2,-.1)}{.3}{.1}{.1}{$r_{2n}^{-1}$}
	\draw (1,-.4) arc (-180:0:.5cm);
  \draw (2,.2) arc (180:90:.3cm) --node[above]{\scriptsize{$\cP[2n]$}}(2.6,.5);
	\identityMap{(1,0)}{.4}{n\,\,\,}
  \node at (1.5,.5) {$\scriptstyle\dag$};
\end{tikzpicture}
=
\begin{tikzpicture}[baseline=-.1cm]
	\draw (-.6,-1) --node[left]{\scriptsize{$v$}} (-.6,-.3);
	\draw (-.2,-.3) --node[right]{\scriptsize{$v^{\vee}$}} (-.2,-1);
	\roundNbox{unshaded}{(-.4,0)}{.3}{.3}{.3}{$\coev_v^\dag$}
  \draw (1.4,0) --node[above]{\scriptsize{$\cP[2n]$}} (2.2,0);
	\identityMap{(1,0)}{.4}{n\,\,\,}
\end{tikzpicture}
\,.
\]
We then calculate with $c=``\Phi(v)\otimes x^{n}"$
that $(\coev_c^\dag\otimes \id_{c^{\vee\vee}})\circ (\id_c\otimes \coev_{c^\vee})$ is given by
\[
\begin{tikzpicture}[baseline=-.1cm]
  \draw (0,.6) -- (5,.6);
	\draw (1.2,-.6) -- (5,-.6);
  \draw (-1.4,-1.2) --node[left]{$\scriptstyle v$} (-1.4,.3);
  \draw[knot] (-1,.3) --node[right]{$\scriptstyle v^\vee$} (-1,-.6) .. controls ++(270:.6cm) and ++(270:.6cm) .. (.6,-.6) --node[right]{$\scriptstyle v^{\vee\vee}$} (.6,1.2);
	\draw (5.2,0) -- (5.6,0);
	\roundNbox{unshaded}{(-1.2,.6)}{.3}{.3}{.3}{$\coev_v^\dag$}
	\evaluationMap{(0,.6)}{.4}{n}
	\coevaluationMap{(1.2,-.6)}{.4}{n}
	\tensorLeftIdEv{(2.4,.6)}{.4}{}{}
	\tensorRightIdCoev{(2.4,-.6)}{.4}{}{}
	\multiplication{(4.2,0)}{1}{n}{3n}{n}
\end{tikzpicture}
\,=
\begin{tikzpicture}[baseline=-.1cm]
	\draw (-.6,-1) -- (-.6,-.3);
	\draw (-.2,-.3) arc (-180:0:0.3cm) -- (.4,1);
	\roundNbox{unshaded}{(-.4,0)}{.3}{.3}{.3}{$\coev_v^\dag$}
	\draw (1.4,0) -- (2.4,0);
	\node at (-.8,-.8) {\scriptsize{$v$}};
	\node at (-.3,-.6) {\scriptsize{$v^{\vee}$}};
	\node at (.1,.8) {\scriptsize{$v^{\vee\vee}$}};
	\identityMap{(1,0)}{.4}{n\,\,\,}
	\node at (1.9,.2) {\scriptsize{$\cP[2n]$}};
\end{tikzpicture}
\,=
\begin{tikzpicture}[baseline=-.1cm]
	\draw (0,-1) -- (0,1);
	\roundNbox{unshaded}{(0,0)}{.4}{0}{0}{$\varphi_v$}
	\draw (1.4,0) -- (2.4,0);
	\node at (-.2,-.8) {\scriptsize{$v$}};
	\node at (-.3,.8) {\scriptsize{$v^{\vee\vee}$}};
	\identityMap{(1,0)}{.4}{n\,\,\,}
	\node at (1.9,.2) {\scriptsize{$\cP[2n]$}};
\end{tikzpicture}
\,.
\]
This completes the proof.
\end{proof}

\begin{cor}[\ref{C:Real}]
The symmetric self-duality $r_x: x\to x^\vee$ is unitary, i.e., $(x,r_x)$ is real.
\end{cor}
\begin{proof}
We used the identification between $1_\cV^\vee = 1_\cV$ to identify $x=x^\vee$ and $r_x = \id_x$, but the former identification is unitary, so the latter is too.
\end{proof}

Recall from Lemma \ref{lem: C-->D and D spherical} that 
sphericality of the unitary dual functor $\vee_\cV$
is a necessary condition for 
the state $\psi_\cC$ to be spherical.

\begin{prop}
If $\vee_\cV$ is a spherical unitary dual functor on $\cV$ and $(\cP,j,\psi_\cP)$ is spherical, then so is $\psi_\cC$.
\end{prop}
\begin{proof}
For $f\in \End_{\cC_0}(``\Phi(u)\otimes x^{\otimes n}")$,
its left and right traces in $\cC$ are given by
\begin{align*}
\tr_L^\cC(f)
&=
\tikzmath{
	\draw (2,.4) -- (6.4,.4);
	\draw (1,-.7) --node[below]{$\scriptstyle \cP[2n]$} (4,-.7);
	\draw (1,1.5) --node[above]{$\scriptstyle \cP[2n]$} (4,1.5);
	\draw (.7,.6) --node[above, yshift=-.1cm]{$\scriptstyle \cP[2n]$} (1.6,.6);
	\draw (-.3,.1) -- (.7,.1) --node[above, yshift=-.1cm]{$\scriptstyle \cP[2n]$} (1.6,.1);
	\draw[knot] (.4,.9) --node[left]{$\scriptstyle u$}  (.4,1.3) .. controls ++(90:.5cm) and ++(90:.5cm) .. (-1.1,1.3) --node[left]{$\scriptstyle u^\vee$}  (-1.1,-.5) .. controls ++(270:.5cm) and ++(270:.5cm) .. (.4,-.5) --node[left]{$\scriptstyle u$}  (.4,.3);
	\roundNbox{fill=white}{(.4,.6)}{.3}{0}{0}{$f$}
	\identityMap{(-.6,.1)}{.3}{n\,\,\,\,}
	\evaluationMap{(1,1.5)}{.4}{n}
	\coevaluationMap{(1,-.7)}{.4}{n}
	\tensor{(2.4,.4)}{.8}{n}{n}{n}{n}
	\coordinate (a) at (4,-1);
	\coordinate (b) at ($ (a) + (.2,1.4) $);
	\pgfmathsetmacro{\width}{.2}
	\pgfmathsetmacro{\height}{.6}
	\coordinate (c) at ($ (a) + (\width,\height) $);
	\draw[rounded corners=5pt, very thick, unshaded] (a) rectangle ($ (a) + (1.8,2.8) $);
	\draw[thick, red] ($ (a) + (.7,.7) $) -- (b);
	\draw[thick, red] ($ (a) + (.7,1.4) $) -- (b);
	\draw[thick, red] ($ (a) + (.7,2.1) $) -- (b);
	\draw[very thick] ($ (a) + (.9,1.4) $) ellipse (.7 and 1.3);
	\draw ($ (a) + (.9,.7) $) -- ($ (a) + (.9, 2.1) $);
	\filldraw[very thick, unshaded] ($ (a) + (.9,.7) $) circle (.2cm);
	\filldraw[very thick, unshaded] ($ (a) + (.9,1.4) $) circle (.2cm);
	\filldraw[very thick, unshaded] ($ (a) + (.9,2.1) $) circle (.2cm);
	\node at ($(a) + (1.1,1.05)$) {\scriptsize{$2n$}};
	\node at ($(a) + (1.1,1.75)$) {\scriptsize{$2n$}};
	\node at (6.2,.6) {\scriptsize{$\cP[0]$}};
	\node at (3.6,.6) {\scriptsize{$\cP[4n]$}};
}
=
\tikzmath{
	\coordinate (a) at (0,0);
	\pgfmathsetmacro{\boxWidth}{1};
	\draw (-2,0) -- (-1,0);
	\draw (1,0) -- (2,0);
	\draw (-2.1,.3) node[right, yshift=.2cm, xshift=-.15cm]{$\scriptstyle u$} arc (0:180:.3cm) --node[left]{$\scriptstyle u^\vee$} (-2.7,-.3) arc (-180:0:.3cm) node[right, yshift=-.2cm, xshift=-.15cm]{$\scriptstyle u$};
	\roundNbox{unshaded}{(-2.1,0)}{.3}{0}{0}{$f$}
	\draw[rounded corners=5pt, very thick, unshaded] ($ (a) - (\boxWidth,\boxWidth) $) rectangle ($ (a) + (\boxWidth,\boxWidth) $);
	\draw[thick, red] (180:1/5*\boxWidth) -- (180:4/5*\boxWidth);
	\draw[very thick] (a) circle (4/5*\boxWidth);
	\draw[very thick] (a) circle (1/5*\boxWidth);
	\draw (90:1/5*\boxWidth) .. controls ++(90:.3cm) and ++(90:.5cm) .. (180:1/2*\boxWidth) .. controls ++(270:.5cm) and ++(270:.3cm) ..  (270:1/5*\boxWidth);
	\node at (1.35,.2) {\scriptsize{$\cP[0]$}};
	\node at (-1.4,.2) {\scriptsize{$\cP[2n]$}};
	\node at (-.6,-.2) {\scriptsize{$n$}};
}
\displaybreak[1]\\
\tr^\cC_R(f)
&=
\tikzmath{
	\draw (2,.4) -- (6.4,.4);
	\draw (.5,-.7) --node[below]{$\scriptstyle \cP[2n]$} (4,-.7);
	\draw (.5,1.5) --node[above]{$\scriptstyle \cP[2n]$} (4,1.5);
	\draw (-.8,.6) -- (.7,.6) --node[above, yshift=-.1cm]{$\scriptstyle \cP[2n]$} (1.6,.6);
	\draw (.7,.1) --node[above, yshift=-.1cm]{$\scriptstyle \cP[2n]$} (1.6,.1);
	\draw[knot] (-1.1,.9) --node[left]{$\scriptstyle u$}  (-1.1,1.3) .. controls ++(90:.5cm) and ++(90:.5cm) .. (-.1,1.3) node[left]{$\scriptstyle u^\vee$} --  (-.1,-.5) .. controls ++(270:.5cm) and ++(270:.5cm) .. (-1.1,-.5) --node[left]{$\scriptstyle u$}  (-1.1,.3);
	\roundNbox{fill=white}{(-1.1,.6)}{.3}{0}{0}{$f$}
	\identityMap{(.4,.1)}{.3}{n\,\,\,}
	\evaluationMap{(.5,1.5)}{.4}{n}
	\coevaluationMap{(.5,-.7)}{.4}{n}
	\tensor{(2.4,.4)}{.8}{n}{n}{n}{n}
	\coordinate (a) at (4,-1);
	\coordinate (b) at ($ (a) + (.2,1.4) $);
	\pgfmathsetmacro{\width}{.2}
	\pgfmathsetmacro{\height}{.6}
	\coordinate (c) at ($ (a) + (\width,\height) $);
	\draw[rounded corners=5pt, very thick, unshaded] (a) rectangle ($ (a) + (1.8,2.8) $);
	\draw[thick, red] ($ (a) + (.7,.7) $) -- (b);
	\draw[thick, red] ($ (a) + (.7,1.4) $) -- (b);
	\draw[thick, red] ($ (a) + (.7,2.1) $) -- (b);
	\draw[very thick] ($ (a) + (.9,1.4) $) ellipse (.7 and 1.3);
	\draw ($ (a) + (.9,.7) $) -- ($ (a) + (.9, 2.1) $);
	\filldraw[very thick, unshaded] ($ (a) + (.9,.7) $) circle (.2cm);
	\filldraw[very thick, unshaded] ($ (a) + (.9,1.4) $) circle (.2cm);
	\filldraw[very thick, unshaded] ($ (a) + (.9,2.1) $) circle (.2cm);
	\node at ($(a) + (1.1,1.05)$) {\scriptsize{$2n$}};
	\node at ($(a) + (1.1,1.75)$) {\scriptsize{$2n$}};
	\node at (6.2,.6) {\scriptsize{$\cP[0]$}};
	\node at (3.6,.6) {\scriptsize{$\cP[4n]$}};
}
=
\tikzmath{
	\coordinate (a) at (0,0);
	\pgfmathsetmacro{\boxWidth}{1};
	\draw (-2.4,0) -- (-1,0);
	\draw (1,0) -- (2,0);
	\draw[knot] (-2.7,.3) node[left, yshift=.2cm]{$\scriptstyle u$} arc (180:0:.3cm) --node[right, yshift=-.2cm]{$\scriptstyle u^\vee$} (-2.1,-.3) arc (0:-180:.3cm) node[left, yshift=-.2cm]{$\scriptstyle u$};
	\roundNbox{unshaded}{(-2.7,0)}{.3}{0}{0}{$f$}
	\draw[rounded corners=5pt, very thick, unshaded] ($ (a) - (\boxWidth,\boxWidth) $) rectangle ($ (a) + (\boxWidth,\boxWidth) $);
	\draw[thick, red] (180:1/5*\boxWidth) -- (180:4/5*\boxWidth);
	\draw[very thick] (a) circle (4/5*\boxWidth);
	\draw[very thick] (a) circle (1/5*\boxWidth);
	\draw (90:1/5*\boxWidth) .. controls ++(90:.3cm) and ++(90:.5cm) .. (0:1/2*\boxWidth) .. controls ++(270:.5cm) and ++(270:.3cm) ..  (270:1/5*\boxWidth);
	\node at (1.35,.2) {\scriptsize{$\cP[0]$}};
	\node at (-1.4,.2) {\scriptsize{$\cP[2n]$}};
	\node at (.65,0) {\scriptsize{$n$}};
}
\end{align*}
Now since $\vee_\cV$ and $\psi_\cP$ are spherical, we see that
$$
\tikzmath{
	\coordinate (a) at (0,0);
	\pgfmathsetmacro{\boxWidth}{1};
	\draw (-2,0) -- (-1,0);
	\draw (1,0) -- (2,0);
	\draw[dotted] (2,0) -- (3,0);
	\draw (-2.1,.3) node[right, yshift=.2cm, xshift=-.15cm]{$\scriptstyle u$} arc (0:180:.3cm) --node[left]{$\scriptstyle u^\vee$} (-2.7,-.3) arc (-180:0:.3cm) node[right, yshift=-.2cm, xshift=-.15cm]{$\scriptstyle u$};
	\roundNbox{unshaded}{(-2.1,0)}{.3}{0}{0}{$f$}
	\roundNbox{unshaded}{(2,0)}{.3}{0}{0}{$\psi_\cP$}
	\draw[rounded corners=5pt, very thick, unshaded] ($ (a) - (\boxWidth,\boxWidth) $) rectangle ($ (a) + (\boxWidth,\boxWidth) $);
	\draw[thick, red] (180:1/5*\boxWidth) -- (180:4/5*\boxWidth);
	\draw[very thick] (a) circle (4/5*\boxWidth);
	\draw[very thick] (a) circle (1/5*\boxWidth);
	\draw (90:1/5*\boxWidth) .. controls ++(90:.3cm) and ++(90:.5cm) .. (180:1/2*\boxWidth) .. controls ++(270:.5cm) and ++(270:.3cm) ..  (270:1/5*\boxWidth);
	\node at (1.35,.2) {\scriptsize{$\cP[0]$}};
	\node at (2.65,.2) {\scriptsize{$1_\cV$}};
	\node at (-1.4,.2) {\scriptsize{$\cP[2n]$}};
	\node at (-.6,-.2) {\scriptsize{$n$}};
}
=
\tikzmath{
	\coordinate (a) at (0,0);
	\pgfmathsetmacro{\boxWidth}{1};
	\draw (-2.4,0) -- (-1,0);
	\draw (1,0) -- (2,0);
	\draw[dotted] (2,0) -- (3,0);
	\draw[knot] (-2.7,.3) node[left, yshift=.2cm]{$\scriptstyle u$} arc (180:0:.3cm) --node[right, yshift=-.2cm]{$\scriptstyle u^\vee$} (-2.1,-.3) arc (0:-180:.3cm) node[left, yshift=-.2cm]{$\scriptstyle u$};
	\roundNbox{unshaded}{(-2.7,0)}{.3}{0}{0}{$f$}
	\roundNbox{unshaded}{(2,0)}{.3}{0}{0}{$\psi_\cP$}
	\draw[rounded corners=5pt, very thick, unshaded] ($ (a) - (\boxWidth,\boxWidth) $) rectangle ($ (a) + (\boxWidth,\boxWidth) $);
	\draw[thick, red] (180:1/5*\boxWidth) -- (180:4/5*\boxWidth);
	\draw[very thick] (a) circle (4/5*\boxWidth);
	\draw[very thick] (a) circle (1/5*\boxWidth);
	\draw (90:1/5*\boxWidth) .. controls ++(90:.3cm) and ++(90:.5cm) .. (0:1/2*\boxWidth) .. controls ++(270:.5cm) and ++(270:.3cm) ..  (270:1/5*\boxWidth);
	\node at (1.35,.2) {\scriptsize{$\cP[0]$}};
	\node at (2.65,.2) {\scriptsize{$1_\cV$}};
	\node at (-1.4,.2) {\scriptsize{$\cP[2k]$}};
	\node at (.65,0) {\scriptsize{$k$}};
}
$$
and thus
$\psi_\cC(\tr^\cC_L(f))=\psi_\cC(\tr^\cC_R(f))$.
\end{proof}

Now that $\cC_0$ is endowed with a unitary dual functor and faithful state, we get a unitary trace on $\End_{\cC_0}(``\Phi(v)\otimes x^{\otimes n}")$ by $\psi_\cC\circ \tr^\cC_R$.

\begin{prop}[\ref{C:Adjunction}]
The adjunction $\Phi\dashv \bTr_\cV$ is unitary.
In particular, $\Phi$ and $\bTr_\cV$ are dagger functors.
\end{prop}
\begin{proof}
We show the two inner products on $\cV\big(u\to v\otimes \cP[n]\big)$ coming from $\cC$ and $\cV$ agree:
\begin{align*}
\langle f,g\rangle_\cC
&=
(\psi_\cC\circ \tr^\cC_R)(g^*\circ f)
=
\tikzmath{
	\coordinate (a) at (0,0);
	\pgfmathsetmacro{\boxWidth}{1};
\draw (-2.8,-.2) --node[left]{$\scriptstyle u$} (-2.8,.2);
\draw (-1,.5) --node[above]{$\scriptstyle \cP[n]$} (-2.2,.5) -- (-2.5,.5);
\draw (-1,-.5) --node[above]{$\scriptstyle \cP[n]$} (-2.2,-.5) -- (-2.5,-.5);
\draw (1,0) -- 
(2,0);
\draw[knot] (-2.8,.8) node[left, yshift=.2cm, xshift=.1cm]{$\scriptstyle v$} arc (180:0:.3cm) -- (-2.2,-.8) node[right]{$\scriptstyle v^\vee$} arc (0:-180:.3cm) node[left, yshift=-.2cm, xshift=.1cm]{$\scriptstyle v$};
	\roundNbox{unshaded}{(-2.8,.5)}{.3}{0}{0}{$f$}
	\roundNbox{unshaded}{(-2.8,-.5)}{.3}{0}{0}{$g^*$}
	\roundNbox{unshaded}{($ (a) + (2,0) $)}{.3}{0}{0}{$\psi_\cP$}
	\draw[rounded corners=5pt, very thick, unshaded] ($ (a) - (\boxWidth,\boxWidth) $) rectangle ($ (a) + (\boxWidth,\boxWidth) $);
	\draw ($ (a) + 1/3*(0,1) $) -- ($ (a) - 1/3*(0,\boxWidth) $);
	\draw[thick, red] ($ (a) + 1/3*(0,\boxWidth) - 1/5*(\boxWidth,0) $) -- ($ (a) - 2/3*(\boxWidth,0) $);
	\draw[thick, red] ($ (a) - 1/3*(0,\boxWidth) - 1/5*(\boxWidth,0) $) -- ($ (a) - 2/3*(\boxWidth,0) $);
	\draw[very thick] (a) ellipse ({2/3*\boxWidth} and {5/6*\boxWidth});
	\filldraw[very thick, unshaded] ($ (a) + 1/3*(0,\boxWidth) $) circle (1/5*\boxWidth);
	\filldraw[very thick, unshaded] ($ (a) - 1/3*(0,\boxWidth) $) circle (1/5*\boxWidth);
	\node at ($ (a) + (.3,0) $) {$\scriptstyle n$};
\draw[densely dotted] (2.3,0) --node[above]{$\scriptstyle 1_\cV$}  (3,0);
\draw (1,0) --node[above]{$\scriptstyle \cP[0]$}  (1.7,0);
}
\\&\underset{(\ref{eq:PlanarPairing},\ref{eq:DaggerOnC0})}{=}
\tikzmath{
\roundNbox{unshaded}{(0,.5)}{.3}{.1}{.1}{$f$}
\roundNbox{unshaded}{(0,-.5)}{.3}{.1}{.1}{$g^\dag$}
\draw (-.2,.8) node[left, yshift=.2cm]{$\scriptstyle v$} 
.. controls ++(90:.7cm) and ++(90:.7cm) .. (2,.8) --node[right]{$\scriptstyle v^\vee$}(2,-.8) 
.. controls ++(270:.7cm) and ++(270:.7cm) ..
(-.2,-.8) node[left, yshift=-.2cm]{$\scriptstyle v$};
\draw (.2,.8) node[right, yshift=.15cm, xshift=.15cm]{$\scriptstyle \cP[n]$} 
.. controls ++(90:.5cm) and ++(90:.5cm) .. 
(1.6,.8) --node[left,xshift=.1cm]{$\scriptstyle \cP[n]^\vee$} (1.6,-.8)
.. controls ++(270:.5cm) and ++(270:.5cm) .. (.2,-.8) node[right, yshift=-.15cm, xshift=.15cm]{$\scriptstyle \cP[n]$};
\draw (0,-.2) --node[left]{$\scriptstyle u$} (0,.2);
}
=
\langle f,g\rangle_\cV
\end{align*}
as in the proof of Lemma \ref{lem:LinkingAlgebraState}.
\end{proof}

\subsubsection{Functoriality}
\label{sec:DeltaUnitaryExtension}

We now suppose $H: (\cP_1,r^1,\psi_1)\to (\cP_2,r^2,\psi_2)$ is a map of unitary anchored planar algebras, and
let $(\cC_i,\Phi_i^{\scriptscriptstyle \cZ},\vee_i,\psi_i,x_i)=\Delta(\cP_i)$ for $i=1,2$ be the pointed unitary module multitensor categories constructed in the last subsection.
In \cite[\S6.6]{MR4528312}, we constructed a strict map of ordinary pointed pivotal module tensor categories $\Delta(H): \cC_1\to \cC_2$ by
$$
\Delta(H)\left(
\tikzmath{
\roundNbox{fill=white}{(0,0)}{.3}{0}{0}{$f$}
\draw (0,-.7) --node[left]{$\scriptstyle u$} (0,-.3);
\draw (0,.7) --node[left]{$\scriptstyle v$} (0,.3);
\draw (.3,0) --node[above]{$\scriptstyle \cP_1[n+k]$} (1.5,0);
}
\right)
:=
\tikzmath{
\roundNbox{fill=white}{(0,0)}{.3}{0}{0}{$f$}
\roundNbox{fill=white}{(1.8,0)}{.3}{0}{0}{$H$}
\draw (0,-.7) --node[left]{$\scriptstyle u$} (0,-.3);
\draw (0,.7) --node[left]{$\scriptstyle v$} (0,.3);
\draw (.3,0) --node[above]{$\scriptstyle \cP_1[n+k]$} (1.5,0);
\draw (2.1,0) --node[above]{$\scriptstyle \cP_2[n+k]$} (3.3,0);
}
\qquad\qquad
\forall
\,
f: ``\Phi(u)\otimes x^{\otimes k}"
\to ``\Phi(v)\otimes x^{\otimes n}"
$$
where the action-coherence morphism $\gamma^{\Delta(H)}: \Phi_2\Rightarrow \Delta(H)\circ \Phi_1$ is the identity.

It remains to prove that:
\begin{itemize}
\item 
$\Delta(H)$ is a $\dag$-tensor functor,
\item
the identity action coherence monoidal natural isomorphism $\id: \Phi_2 \Rightarrow \Delta(H)\circ \Phi_1$ is involutive,
i.e.,
for all $v\in \cV$,
$\chi^{\Phi_2}_v = \chi^{\Delta(H)}_{\Phi_1(v)}\circ \Delta(H)(\chi^{\Phi_1}_v):
\Phi_2(\overline{v}) \to \overline{\Delta(H)(\Phi_1(v))}$,
and
\item
$\psi_2 \circ \Delta(H) = \psi_1$ on $\End_{\cC_1}(1_{\cC_1})$.
\end{itemize}

For the first condition, since $\Delta(H)$ is strict, it suffices to check it is a $\dag$-functor.
Indeed,
\begin{align*}
\Delta(H)\left(
\tikzmath{
\roundNbox{fill=white}{(0,0)}{.3}{0}{0}{$f^*$}
\draw (0,-.7) --node[left]{$\scriptstyle v$} (0,-.3);
\draw (0,.7) --node[left]{$\scriptstyle u$} (0,.3);
\draw (.3,0) --node[above]{$\scriptstyle \cP_1[k+n]$} (1.5,0);
}
\right)
:&=
\tikzmath{
\roundNbox{fill=white}{(0,0)}{.3}{0}{0}{$f^*$}
\roundNbox{fill=white}{(1.8,0)}{.3}{0}{0}{$H$}
\draw (0,-.7) --node[left]{$\scriptstyle v$} (0,-.3);
\draw (0,.7) --node[left]{$\scriptstyle u$} (0,.3);
\draw (.3,0) --node[above]{$\scriptstyle \cP_1[k+n]$} (1.5,0);
\draw (2.1,0) --node[above]{$\scriptstyle \cP_2[k+n]$} (3.3,0);
}
=
\tikzmath{
\roundNbox{fill=white}{(0,0)}{.3}{.1}{.1}{$f^\dag$}
\roundNbox{fill=white}{(2.9,.6)}{.3}{0}{0}{$H$}
\roundNbox{fill=white}{(1.2,0)}{.3}{.35}{.35}{$\scriptstyle (r^1_{n+k})^{-1}$}
\draw (-.2,-.9) --node[left]{$\scriptstyle v$} (-.2,-.3);
\draw (0,.7) --node[left]{$\scriptstyle u$} (0,.3);
\draw (.2,-.3) arc (-180:0:.5cm) node[right, yshift=-.2cm]{$\scriptstyle \overline{\cP_1[n+k]}$};
\draw (1.2,.3) arc (180:90:.3cm) --node[above, xshift=-.2cm]{$\scriptstyle \cP_1[k+n]$}  (2.6,.6) ;
\draw (3.2,.6) --node[above]{$\scriptstyle \cP_1[k+n]$} (4.4,.6);
}
\\&=
\tikzmath{
\roundNbox{fill=white}{(0,0)}{.3}{.1}{.1}{$f^\dag$}
\roundNbox{fill=white}{(2.2,-.7)}{.3}{0}{0}{$\overline{H}$}
\roundNbox{fill=white}{(3,0)}{.3}{.35}{.35}{$\scriptstyle (r^2_{n+k})^{-1}$}
\draw (-.2,-1) --node[left]{$\scriptstyle v$} (-.2,-.3);
\draw (0,.7) --node[left]{$\scriptstyle u$} (0,.3);
\draw (.2,-.3) arc (180:270:.4cm) --node[above]{$\scriptstyle \overline{\cP_1[n+k]}$} (1.9,-.7);
\draw (3,.3) arc (180:90:.3cm) --node[above, xshift=-.2cm]{$\scriptstyle \cP_2[k+n]$}  (4.4,.6) ;
\draw (2.5,-.7) arc (-90:0:.4cm) node[right, yshift=-.2cm]{$\scriptstyle \overline{\cP_2[n+k]}$};
}
=
\tikzmath{
\roundNbox{fill=white}{(0,0)}{.3}{.3}{.3}{$f^\dag$}
\roundNbox{fill=white}{(.4,-1)}{.3}{0}{0}{$H^\dag$}
\roundNbox{fill=white}{(2.5,-.5)}{.3}{.35}{.35}{$\scriptstyle (r^2_{n+k})^{-1}$}
\draw (-.2,-1.8) --node[left]{$\scriptstyle v$} (-.2,-.3);
\draw (0,.7) --node[left]{$\scriptstyle u$} (0,.3);
\draw (.4,-.3) --node[right]{$\scriptstyle \cP_1[n+k]$} (.4,-.7);
\draw (.4,-1.3) arc (180:270:.3cm) --node[above]{$\scriptstyle \overline{\cP_2[n+k]}$} (2.2,-1.6) arc (-90:0:.3cm) -- (2.5,-.8);
\draw (2.5,-.2) arc (180:90:.3cm) --node[above, xshift=-.2cm]{$\scriptstyle \cP_2[k+n]$}  (3.4,.1) ;
}
=
\Delta(H)(f)^*.
\end{align*}

For the second condition, 
since these are the canonical coherences $\chi$ associated to a \emph{strict} pivotal $\dag$-tensor functor, they are all identities, and the result holds trivially.

Finally, for the third condition,
for all $f\in \End_{\cC_1}(1_{\cC_1})$,
$$
(\psi_2\circ \Delta(H))(f)
=
\tikzmath{
\roundNbox{fill=white}{(0,0)}{.3}{0}{0}{$f$}
\roundNbox{fill=white}{(1.4,0)}{.3}{0}{0}{$H$}
\roundNbox{fill=white}{(2.8,0)}{.3}{0}{0}{$\psi_2$}
\draw[densely dotted] (0,-.7) --node[left]{$\scriptstyle 1_\cV$} (0,-.3);
\draw[densely dotted] (0,.7) --node[left]{$\scriptstyle 1_\cV$} (0,.3);
\draw (.3,0) --node[above]{$\scriptstyle \cP_1[0]$} (1.1,0);
\draw (1.7,0) --node[above]{$\scriptstyle \cP_2[0]$} (2.5,0);
\draw[densely dotted] (3.1,0) --node[above]{$\scriptstyle 1_\cV$} (3.6,0);
}=
\tikzmath{
\roundNbox{fill=white}{(0,0)}{.3}{0}{0}{$f$}
\roundNbox{fill=white}{(1.4,0)}{.3}{0}{0}{$\psi_1$}
\draw[densely dotted] (0,-.7) --node[left]{$\scriptstyle 1_\cV$} (0,-.3);
\draw[densely dotted] (0,.7) --node[left]{$\scriptstyle 1_\cV$} (0,.3);
\draw (.3,0) --node[above]{$\scriptstyle \cP_1[0]$} (1.1,0);
\draw[densely dotted] (1.7,0) --node[above]{$\scriptstyle 1_\cV$} (2.2,0);
}
=
\psi_1(f).
$$

\subsection{Equivalence}

Fix a braided unitary tensor category $\cV$ with a choice of unitary dual functor $\vee_\cV$.
For this section, we use the following notation.
\begin{itemize}
\item
$\APA$ is the 1-category of anchored planar algebras over $\cV$,
and $\UAPA$ is the 1-subcategory of unitary anchored planar algebras over $\cV$.
\item
$\ModTens$ is the 1-truncation of the 2-category of pointed pivotal module tensor categories over $\cV$,
and 
$\UModTens$ is the 1-truncation of the 2-category of pointed unitary module multitensor categories over $\cV$.
\end{itemize}

In \cite{MR4528312}, we constructed functors $\Delta: \APA \to \ModTens$ and $\Lambda: \ModTens \to \APA$.
In \S\ref{sec:DeltaUnitaryExtension} above, we showed $\Delta$ extends to a functor $\UAPA\to \UModTens$, 
and 
in \S\ref{sec:LambdaUnitaryExtension} above, we showed $\Lambda$ extends to a functor $\UModTens\to \UAPA$.
Moreover, these extensions are clearly compatible with the forgetful functors which forget our extra unitary structure.
\[
\begin{tikzcd}
\UAPA
\arrow[d,"\Forget"]
\arrow[rr, bend left=10,"\Delta"]
&&
\UModTens
\arrow[d,"\Forget"]
\arrow[ll, bend left=10,"\Lambda"]
\\
\APA
\arrow[rr, bend left=10,"\Delta"]
&&
\ModTens
\arrow[ll, bend left=10,"\Lambda"]
\end{tikzcd}
\]

In \cite[\S7]{MR4528312}, we showed that $\Lambda\circ \Delta=\id_{\APA}$, and it is immediate that $\Lambda\circ \Delta=\id_{\UAPA}$ upstairs.
We also constructed a natural isomorphism $\Psi: \Delta\circ\Lambda \Rightarrow \id_{\ModTens}$ by
$$
\Psi(f)=
\begin{tikzpicture}[baseline=0cm, scale=.8]
	\plane{(-.4,-.9)}{2.8}{1.8}
	\node at (-2.4,1.7) {\scriptsize{$u$}};
	\node at (1,1.7) {\scriptsize{$v$}};
	\draw[thick, red] (-2.8,1.5) -- (-1.2,1.5);
	\draw[thick, orange] (-1.2,1.5) -- (2,1.5);
	\coordinate (a) at (0,-.5);
	\draw ($ (a) + (.8,.4) $)  -- ($ (a) + (1.6,.4) $) ;
	\draw ($ (a) + (.8,.2) $)  arc (90:-90:.2cm) -- ($ (a) + (-.6,-.2) $) ;
	\node at ($ (a) + (1.1,.6) $)  {\scriptsize{$j$}};
	\node at ($ (a) + (1.2,0) $)  {\scriptsize{$i$}};
	\CMbox{box}{(a)}{.8}{.8}{.4}{$\varepsilon$}
	\curvedTubeNoString{($ (a) + (-.3,.4) $)}{1}{0}{$f$}
\end{tikzpicture}
\qquad\qquad
\begin{aligned}
f&\in\cC'\big(``\Phi(u)\otimes x^{\otimes i}" \to ``\Phi(v)\otimes x^{\otimes j}"\big)
\\
&\hspace{1cm}=\cV\big(u\to v\otimes \cP[i+j]\big).
\end{aligned}
$$
where $\cC'=\Lambda(\Delta(\cC))$, which is the unitary Cauchy completion of $\cC'_0$ (recall $\cC_0'$ has objects $``\Phi(v)\otimes x^{\otimes n}"$ for $v\in \cV$).

It remains to show that $\Psi$ lifts to a unitary natural isomorphism upstairs.
To do so, it is enough (in view of the fact that $\Psi$ is an isomorphism) to show that $f\mapsto \Psi(f)$ is norm-preserving, i.e., 
$$
\big\|f\big\|_{\cV(u\to v\otimes \cP[i+j])} = \big\|\Psi(f)\big\|_{\cC(\Phi(u)\otimes x^{\otimes i} \to \Phi(v)\otimes x^{\otimes j})}.
$$

By replacing $f$ by $\ev_v\circ (1_{\overline v}\otimes f)$ (i.e. folding the $v$ strand to the left), we may assume without loss of generality that $v=1_\cV$.
Similarly, by replacing 
$f\in\cC'\big(``\Phi(u)\otimes x^{\otimes i}" \to ``\Phi(v)\otimes x^{\otimes j}"\big)$
by the corresponding morphism in $\cC'\big(``\Phi(u)\otimes x^{\otimes 0}" \to ``\Phi(v)\otimes x^{\otimes i+j}"\big)$, we may assume without loss of generality that $i=0$.
We may now proceed as in the proof of Proposition \ref{prop:FormulaForEvTr} to show:
\begin{align*}
\left\|
\begin{tikzpicture}[baseline=0cm, scale=.8]
	\plane{(-.4,-.9)}{2.8}{1.8}
	\node[red] at (-2.7,.3) {\scriptsize{$\Phi(u)$}};
	\draw[thick, red] (-3,.1) -- (-1.2,.1);
	\coordinate (a) at (0,-.5);
	\draw ($ (a) + (.8,.4) $)  -- ($ (a) + (1.6,.4) $) ;
	\node at ($ (a) + (1.1,.6) $)  {\scriptsize{$j$}};
	\CMbox{box}{(a)}{.8}{.8}{.4}{$\varepsilon$}
	\straightTubeWithString{($ (a) + (-.3,.4) $)}{.1}{1.5}{}
	\roundNbox{fill=white}{($ (a) + (-1.6,.6) $)}{.4}{.2}{.2}{$\Phi(f)$}
\end{tikzpicture}
\right\|^2
&=
\begin{tikzpicture}[baseline=.65cm, scale=.8]
	\plane{(-1,-.8)}{5.6}{3}
	\draw (1.8,-.1) node[below, xshift=.2cm]{$\scriptstyle j$} arc (-90:90:.7cm) -- node[above]{$\scriptstyle j$} (.4,1.3);
	\CMbox{box}{(-1.4,.8)}{1.8}{.8}{.4}{$\scriptstyle \mate(f)$}
	\CMbox{box}{(0,-.6)}{1.8}{.8}{.4}{$\scriptstyle \overline{\mate(f)}$}
	\draw[thick, red, knot] (-.2,-.1) -- node[below,xshift=.1cm]{$\scriptstyle \overline{\Phi(u)}$} (-1.6,-.1) arc (270:90:.7cm) node[above, xshift=-.6cm, yshift=-.2cm]{$\scriptstyle \Phi(u)$};
\end{tikzpicture}
\\&=
\|\mate(f)\|^2_{\cC(\Phi(u)\to x^{\otimes j})}
\\&=
\|f\|^2_{\cV(u\to \Phi(x^{\otimes j}))}
\end{align*}
as desired.



\bibliographystyle{alpha}
{\footnotesize{
 \IfFileExists{../../bibliography/bibliography.bib}%
{\bibliography{../../bibliography/bibliography}}
{\bibliography{../../../../../../Documents/bibliography/bibliography}}
}}

\end{document}